\documentclass[11pt,twoside]{amsart}
\usepackage{amsmath,amssymb,amsthm,amscd,stmaryrd}
\usepackage{mathrsfs}
\usepackage{enumitem}

\usepackage{hyperref}
\hypersetup{bookmarksdepth=3}

\numberwithin{equation}{section}
\newtheorem{lemma}[equation]{Lemma}
\newtheorem{theorem}[equation]{Theorem}
\newtheorem{corollary}[equation]{Corollary}
\newtheorem{proposition}[equation]{Proposition}
\theoremstyle{definition}
\newtheorem{definition}[equation]{Definition}
\newtheorem{remark}[equation]{Remark}

\newcommand{\R}{{\mathbb R}}
\newcommand{\Z}{{\mathbb Z}}

\newcommand{\bH}{{\mathbf H}}
\newcommand{\bI}{{\mathbf I}}
\newcommand{\bZ}{{\mathbf Z}}
\newcommand{\M}{{\mathbf M}}

\newcommand{\cC}{{\mathscr C}}
\newcommand{\cD}{{\mathscr D}}
\newcommand{\cF}{{\mathscr F}}

\newcommand{\cL}{{\mathscr L}}
\newcommand{\cN}{{\mathscr N}}
\newcommand{\cP}{{\mathscr P}}
\newcommand{\cS}{{\mathscr S}}
\newcommand{\cZ}{{\mathscr Z}}

\newcommand{\del}{\delta}
\newcommand{\Del}{\Delta}
\newcommand{\eps}{\epsilon}
\newcommand{\gam}{\gamma}
\newcommand{\Gam}{\Gamma}
\newcommand{\kap}{\kappa}
\newcommand{\lam}{\lambda}
\newcommand{\Lam}{\Lambda}
\newcommand{\om}{\omega}
\renewcommand{\rho}{\varrho}
\newcommand{\si}{\sigma}
\newcommand{\Sig}{\Sigma}

\newcommand{\B}[2]{B_{#1}(#2)}   
\newcommand{\Sph}[2]{S_{#1}(#2)} 

\newcommand{\F}{F}
\newcommand{\G}{\Theta}
\newcommand{\Fi}{\F_\infty}
\newcommand{\Gi}{\G_\infty}

\newcommand{\olX}{\,\overline{\!X}}

\newcommand{\la}{\langle}
\newcommand{\ra}{\rangle}
\newcommand{\gp}[2]{\la #2\ra_{#1}}
\newcommand{\bb}[1]{\llbracket #1\rrbracket} 
  
\newcommand{\asrk}{\operatorname{asrk}}
\newcommand{\ben}{\begin{enumerate}}
\newcommand{\een}{\end{enumerate}}
\renewcommand{\d}{\partial}
\newcommand{\di}{\d_\infty} 
\newcommand{\on}{\:\mbox{\rule{0.1ex}{1.2ex}\rule{1.1ex}{0.1ex}}\:}
\newcommand{\ol}{\overline}
\newcommand{\can}{\operatorname{can}}
\newcommand{\const}{\operatorname{const}}
\newcommand{\diam}{\operatorname{diam}}
\newcommand{\cs}{\text{\rm c}} 
\newcommand{\C}{\text{\rm C}} 
\newcommand{\spt}{\operatorname{spt}}

\newcommand{\T}{\text{\rm T}}
\newcommand{\CT}{\cC_\T} 
\newcommand{\bT}{\partial_{\,\T}} 
\newcommand{\dT}{d_\T} 
\newcommand{\aT}{\angle_\T} 

\newcommand{\CAT}{\operatorname{CAT}}
\newcommand{\Hd}{d_\text{\rm H}} 
\newcommand{\id}{\operatorname{id}}
\newcommand{\im}{\operatorname{im}}
\newcommand{\loc}{\text{\rm loc}}
\newcommand{\Ray}{{\rm R}}
\newcommand{\sub}{\subset}
\newcommand{\sm}{\setminus}
\newcommand{\til}{\tilde}
\newcommand{\es}{\emptyset}
\newcommand{\Lip}{\operatorname{Lip}}

\begin{document}

\title{Higher rank hyperbolicity}
\author{Bruce Kleiner}
\address{Courant Institute of Mathematical Sciences\\
251 Mercer St.\\
New York, NY 10012}
\email{bkleiner@cims.nyu.edu}
\thanks{B.K.~was supported by NSF grant DMS-1711556, and a Simons Collaboration grant.}
\author{Urs Lang}
\address{Department of Mathematics\\ETH Zurich\\
R\"amistrasse 101\\
8092 Zurich}
\email{lang@math.ethz.ch}
\date{January 25, 2019}

\begin{abstract}
The large-scale geometry of hyperbolic metric spaces exhibits many 
distinctive features, such as the stability of quasi-geodesics
(the Morse Lemma), the visibility property, and the homeomorphism between
visual boundaries induced by a quasi-isometry. 
We prove a number of closely analogous results for spaces of rank~$n \ge 2$ 
in an asymptotic sense, under some weak assumptions reminiscent of 
nonpositive curvature. For this purpose we replace quasi-geodesic lines
with quasi-minimizing (locally finite) $n$-cycles of $r^n$~volume growth;
prime examples include $n$-cycles associated with $n$-quasiflats.  
Solving an asymptotic Plateau problem and producing unique tangent cones
at infinity for such cycles, we show in particular that
every quasi-isometry between two proper $\CAT(0)$ spaces of asymptotic
rank~$n$ extends to a class of $(n-1)$-cycles in the Tits boundaries.
\end{abstract}

\maketitle
  
\tableofcontents


\section{Introduction}

\subsection{Overview}
\mbox{}
Since the appearance of Gromov's seminal paper \cite{Gro2} more than thirty
years ago, hyperbolicity has played a central role in geometric group theory,
and inspired a number of generalizations and variations.
These include, among others, relative hyperbolicity
\cite{Bow2, DruS, Far, Gro2, Osi1},
various notions of ``directional'' hyperbolicity inherent in
stability/contraction properties of \mbox{(quasi-)}geodesics
\cite{BesF2, ChaS, Cor, KapKlL, KapLP, Sis, Sul}
(this in fact goes back to the notion of rank~one geodesics~\cite{Bal, BalBE}
which predates hyperbolicity), acylindrical hyperbolicity
\cite{BesF1, Bow1, DahGO, Osi2}, and hierarchical hyperbolicity
\cite{BehHS1, BehHS2, HagS, MasM}.
(The literature is far richer than indicated here --- we apologize for
omissions.)
These approaches provide unified descriptions of certain hyperbolicity
phenomena in a variety of non-hyperbolic settings such as non-uniform
lattices in rank one symmetric spaces, mapping class groups, Teichm\"uller
space, and some $\CAT(0)$ cube complexes and three-manifold groups. 

In this paper we develop a notion of higher rank hyperbolicity that
complements, and partly overlaps with, the concepts mentioned above.
We show that for metric spaces of asymptotic rank $n \ge 2$ satisfying
certain weak convexity assumptions (see Section~\ref{subsect:setup} below),
characteristics of hyperbolicity such as slimness of (quasi-)geodesic 
triangles, stability of quasi-geodesics, and visibility remain valid when
properly reformulated in terms of $n$-dimensional (relative) cycles.
In particular, our results hold for proper and cocompact $\CAT(0)$ spaces 
of Euclidean rank~$n$ and in that case they confirm several aspects of
Gromov's discussion in Section~6 of~\cite{Gro4}, and also the well-known
principle that in nonpositively curved spaces hyperbolic behavior should
manifest itself in dimensions above the maximal dimension of a flat.
Our approach also encompasses the stability properties of maximal quasiflats
that were used in the proofs of the quasi-isometric rigidity of higher rank
symmetric spaces in~\cite{EskF, KleL, Most}.

We show further that a quasi-isometry between two proper $\CAT(0)$ spaces 
of asymptotic rank $n \ge 2$ naturally induces an isomorphism between the 
groups of compactly supported integral $(n-1)$-cycles --- metric integral 
currents in the sense of Ambrosio--Kirchheim~\cite{AmbK} --- in their Tits 
boundaries. We remind the reader that in the (hyperbolic) rank one case, 
the usual visual boundaries are homeomorphic, whereas for $n \ge 2$ this 
can fail, even if the quasi-isometry is equivariant with respect to geometric 
actions of some finitely generated group~\cite{CroK}.
The construction of the above isomorphism involves, on the one hand,
an existence result for area-minimizing $n$-dimensional varieties with
prescribed asymptotics. To our knowledge, this is the first general such
result in a setting of nonpositive (rather than strictly negative) curvature
(compare Section~1 in~\cite{Gro3}).
On the other hand, we show that $n$-dimensional (quasi-)minimizers
with $r^n$ volume growth possess unique tangent cones at infinity,
a phenomenon that occurs rather rarely (compare, for example, the
discussion in~\cite{ColM}).

\subsection{Setup} 
\label{subsect:setup}
\mbox{}
For simplicity, we assume throughout the paper that the underlying metric
space $X = (X,d)$ is proper (that is, bounded closed subsets are compact).
For a first set of results, described in Section~\ref{subsect:slim-morse}
below, we assume that $X$ satisfies the following two conditions
for some $n \ge 1$:
\ben
\item[(CI$_n$)]
(Coning inequalities)\; There is a constant $c$ such that any two points
$x,x'$ in $X$ can be joined by a curve of length $\le c\,d(x,x')$,
and for $k = 1,\dots,n$, every $k$-cycle $R$ in some $r$-ball bounds a
$(k+1)$-chain $S$ with mass
\[
\M(S) \le c\,r\,\M(R).  
\]
Here, for a general proper metric space $X$, we use metric integral currents
(see Section~\ref{sect:prelim}).
However, if $X$ is bi-Lipschitz homeomorphic to a finite-dimensional
simplicial complex with standard metrics on the simplices, then
(by a variant of the Federer--Fleming deformation theorem~\cite{FedF})
one may equivalently take simplicial chains or singular Lipschitz chains
(with integer coefficients).
\item[(AR$_n$)]
(Asymptotic rank $\le n$)\; No asymptotic cone of $X$ contains an isometric
copy of an $(n+1)$-dimensional normed space.
Equivalently, $\asrk(X) \le n$, where $\asrk(X)$ is defined as the supremal
$k$ for which there exist a sequence $r_i \to \infty$ and subsets
$Y_i \sub X$ such that the rescaled sets $(Y_i,r_i^{-1}d)$ converge in the
Gromov--Hausdorff topology to the unit ball in some $k$-dimensional normed
space (see Section~\ref{sect:asympt-rank}). 
\een
Condition~(CI$_n$) is reminiscent of nonpositive curvature:
if $X$ is a $\CAT(0)$ or Busemann space~\cite{Pap},
the required inequality holds for the geodesic cone $S$ from the center
of the $r$-ball over $R$ (see Section~\ref{subsect:homotopies}).
Furthermore, any $n$-connected simplicial complex as above with a
properly discontinuous and cocompact simplicial action of a combable group
satisfies~(CI$_n$); see Section~10.2 in~\cite{Eps+}. Every combable group,
in particular every automatic group, admits such an action.

When $X$ is a cocompact $\CAT(0)$ or Busemann space,
the asymptotic rank $\asrk(X)$ equals the maximal dimension of an
isometrically embedded Euclidean or normed space, respectively~\cite{Kle}.
More generally, for spaces satisfying~(CI$_n$), condition~(AR$_n$) is
equivalent to a {\em sub-Euclidean isoperimetric inequality\/} for
$n$-cycles~\cite{Wen4}; this result, restated in Theorem~\ref{thm:subeucl},
plays a key role in this paper.
If $X$ is a geodesic Gromov hyperbolic space, then every asymptotic cone 
of $X$ is an $\R$-tree, thus $\asrk(X) \le 1$.
Conversely, a space satisfying~(CI$_1$) and~(AR$_1$) is Gromov hyperbolic
(compare Corollary~1.3 in~\cite{Wen4} and the special case $n = 1$
of Theorem~\ref{intro-slim} below).

We remark that the asymptotic rank is a quasi-isometry invariant for metric
spaces~\cite{Wen4}, whereas condition~(CI$_n$) is preserved, for instance,
by quasi-isometries between proper and cocompact, $n$-connected simplicial
complexes with standard metrics on the simplices.

The main results discussed in the second half of the paper,
starting from Section~\ref{sect:conical}, involve actual convexity
properties or the ideal boundary of $X$ (rather than condition~(CI$_n$)).
For the outline of these results in Sections~\ref{subsect:asymp}
and~\ref{subsect:qi}, we will therefore assume that $X$ is $\CAT(0)$.
In the body of the paper, we will work with the weaker sufficient condition
that $X$ admits a {\em convex bicombing} --- this disposes with geodesic
uniqueness but retains Busemann convexity for a distinguished family of
geodesics; see Definition~\ref{def:bicombing} and the comments thereafter.

\subsection{Quasi-minimizers with controlled density}
\label{subsect:qmin}
\mbox{}
We now discuss the objects we use to exhibit higher rank hyperbolic behavior, 
that is, $n$-dimensional replacements for quasi-geodesics.  

One approach would be to study $n$-quasiflats, or more generally, images of 
quasi-isometric embeddings $W \to X$ for suitable subsets $W \sub \R^n$.
(See Section~\ref{subsect:metric} for the standard definitions of
quasi-isometric maps.)
However, since geodesics may be viewed either as isometric embeddings of 
intervals or as length minimizing curves, an alternative approach is to 
consider (relative) $n$-cycles which ``quasi-minimize'' area
(compare~\cite{BanL,Gro3}, for example).  
We follow the latter approach in this paper: it turns out that it is not 
only more general, but it also leads to cleaner and sharper results.  
The quasi-minimality condition will be used in conjunction with a polynomial
growth bound of order $n$. We now provide more details.

We will work with the chain complexes $\bI_{*,\cs}(X)$ and $\bI_{*,\loc}(X)$
of metric {\em integral currents\/} with compact support
and {\em locally integral currents\/} introduced in~\cite{AmbK, Lan3}.
This enables us in particular to pass to limits and to produce
area-minimizers with sharp density and monotonicity properties.
All relevant concepts and results will be reviewed in detail in 
Section~\ref{sect:prelim}. Every singular Lipschitz $n$-chain in $X$ with
integer coefficients may be viewed as an element of $\bI_{n,\cs}(X)$
(and, conversely, every integral current in $\R^N$ admits an
approximation by Lipschitz chains; see Theorem~5.8 in~\cite{FedF}).
Similarly, $\bI_{n,\loc}(X)$ comprises all locally finite Lipschitz $n$-chains.
Associated with every $S \in \bI_{n,\loc}(X)$ is a locally finite Borel 
measure $\|S\|$ on $X$ whose total {\em mass\/} is denoted
$\M(S) := \|S\|(X)$, and the {\em support\/} $\spt(S) \sub X$ is the
smallest closed set supporting $\|S\|$.
We let $\bZ_{n,\cs}(X)$ and $\bZ_{n,\loc}(X)$ denote the respective
cycle groups for $n \ge 1$.

A local cycle $S \in \bZ_{n,\loc}(X)$ will be called (large-scale) 
{\em quasi-minimizing\/} if there exist constants $Q \ge 1$ and $a \ge 0$ 
such that, for every $x \in \spt(S)$ and almost every $r > a$, the restriction
$S \on \B{x}{r} \in \bI_{n,\cs}(X)$ of $S$ to the closed $r$-ball centered
at $x$ satisfies
\[
\M(S \on \B{x}{r}) \le Q\,\M(T)
\]
for all $T \in \bI_{n,\cs}(X)$ with $\d T = \d(S \on \B{x}{r})$;
then~$S$ is {\em $(Q,a)$-quasi-minimizing}. 
A $(1,0)$-quasi-minimizing local cycle is {\em \mbox{(area-)}minimizing}. 
Every quasiflat in $X$ may be viewed as a quasi-minimizer
(see Proposition~\ref{prop:lip-qflats} and Proposition~\ref{prop:qflats} for
two precise statements).

We say that $S \in \bZ_{n,\loc}(X)$ has (large-scale) {\em controlled density\/}
if there exist constants $C > 0$ and $a \ge 0$ such that
\[
\G_{p,r}(S) := \frac{\|S\|(\B{p}{r})}{r^n} \le C
\]
for all $p \in X$ and $r > a$; then $S$ has {\em $(C,a)$-controlled density}.
A generally weaker condition is that the {\em asymptotic density} 
\[
\Gi(S) := \limsup_{r \to \infty}\G_{p,r}(S)
\]
of $S$ be finite; here $p$ is fixed, however the upper limit is independent 
of $p$. Similarly, for $Z \in \bZ_{n,\loc}(X)$ and any $p \in X$, 
we define the {\em asymptotic filling density\/}
\[
\Fi(Z) := \limsup_{r \to \infty}\F_{p,r}(Z),
\]
where $\F_{p,r}(Z)$ denotes the infimum of $\M(V)/r^{n+1}$ over all 
$V \in \bI_{n+1,\cs}(X)$ with $\spt(Z - \d V) \cap \B{p}{r} = \es$ 
(that is, $V$ ``fills $Z$ in $\B{p}{r}$'').
For $S,S' \in \bZ_{n,\loc}(X)$, the relation $\Fi(S - S') = 0$
will serve as an appropriate notion of asymptoticity.

We now discuss the main results in the paper.

\subsection{Slim simplices, Morse Lemma, and asymptote classes} 
\label{subsect:slim-morse}
\mbox{}
We first recall that a geodesic metric space $X$ is 
{\em (Gromov) hyperbolic}\/~\cite{Gro2} if there exists a constant
$\del \ge 0$ such that every geodesic triangle in $X$ is {\em $\del$-slim},
that is, each of its sides lies in the closed $\del$-neighborhood of the
union of the other two. According to the Morse Lemma (which for the real
hyperbolic plane goes back to~\cite{Mor}), every $(L,a)$-quasi-geodesic
segment in $X$ is then at Hausdorff distance at most $b$ from a geodesic
segment connecting its endpoints, where the constant $b$ depends only
on $L,a$ and $\del$.
Thus any triangle composed of three $(L,a)$-quasi-geodesic segments
is still $(\delta + 2b)$-slim. 

We prove the following higher rank analog of this property. 

\begin{theorem}[slim simplices] \label{intro-slim} \mbox{}
Let $X$ be a proper metric space satisfying conditions~{\rm (CI$_n$)}
and\/~{\rm (AR$_n$)} for some $n \ge 1$.
Let $\Del$ be a Euclidean $(n+1)$-simplex, and let $f \colon \d\Del \to X$ 
be a map such that for every facet $W$ of $\Del$, the restriction $f|_W$ 
is an $(L,a)$-quasi-isometric embedding. Then, for every facet $W$, the image
$f(W)$ is contained in the closed $D$-neighborhood of
$f\bigl( \ol{\d\Del \sm W} \bigr)$ for some constant $D = D(X,n,L,a)$.
\end{theorem}

Here $\Del$ is the convex hull of a set of $n+2$ points in $\R^{n+1}$
such that $\Del$ has non-empty interior, and a {\em facet\/} of $\Del$ is the
convex hull of $n+1$ of them. 

The proof of this result depends, on the one hand, on an iterated application 
of the aforementioned sub-Euclidean isoperimetric inequality. For a cycle 
$Z \in \bZ_{n,\cs}(X)$ with controlled density, this provides an arbitrarily 
small upper bound $F_{p,r}(Z) < \eps$ on the filling density in any ball
$\B{p}{r}$ of sufficiently large radius, depending on $\eps$
(Proposition~\ref{prop:part-filling}).
On the other hand, if $Z$ is ``piecewise $(Q,a)$-quasi-minimizing'',
then $F_{x,r}(Z) \ge c = c(X,n,Q) > 0$ for any ball $\B{x}{r}$ with $r > 4a$
centered on one of the pieces and disjoint from the union of the remaining
ones (Lemma~\ref{lem:fill-density}); thus $x$ cannot be too far away
from this union. 
For an appropriately chosen cycle $Z$ approximating the image of
$f \colon \d\Del \to X$, this yields Theorem~\ref{intro-slim}
(see Theorem~\ref{thm:slim}). 

In combination with the existence of area-minimizing integral currents with 
prescribed boundary, a similar argument yields a higher rank 
analog of the Morse Lemma stated above; see Theorem~\ref{thm:morse-2}.
We further establish the following asymptotic version of this result
(see Theorem~\ref{thm:morse-3} for a generalization including boundaries).

\begin{theorem}[asymptotic Morse Lemma] \label{intro-morse} \mbox{}
Let $X$ be a proper metric space satisfying conditions~{\rm (CI$_n$)}
and\/~{\rm (AR$_n$)} for some $n \ge 1$.
Suppose that $S \in \bZ_{n,\loc}(X)$ is $(Q,a)$-quasi-minimizing and
has $(C,a)$-controlled density.
Then there exists an area-minimizing local cycle $\til S \in \bZ_{n,\loc}(X)$ 
such that $\Fi(S - \til S) = 0$, and every such $\til S$
satisfies $\Gi(\til S) \le \Gi(S)$ and
$\Hd(\spt(S),\spt(\til S)) \le b$ for some constant $b = b(X,n,Q,C,a)$.
\end{theorem}

This implies in particular the following analog of Morse's
Theorem~1~\cite{Mor} on the stability of geodesics in the hyperbolic plane.
We remark that for a Riemannian manifold $X$, metric locally integral currents
in $X$ can be identified with the classical ones from~\cite{Fed}.  

\begin{corollary}[persistence of minimizers]
\label{intro-persistence} \mbox{}
Let $X = (X,g)$ be a Hadamard manifold of asymptotic rank $n \ge 1$, and
suppose that $S \in \bZ_{n,\loc}(X)$ is area-minimizing and has controlled
density.
Then for every Riemannian metric $\til g$ on $X$ bi-Lipschitz equivalent to 
$g$ there is an $\til S \in \bZ_{n,\loc}(X)$ that is area-minimizing with 
respect to~$\til g$ and whose support is at finite Hausdorff distance 
from $\spt(S)$.
\end{corollary}

Note that if $\til d$ is the distance function on $X$ induced by $\til g$,
then $X = (X,\til d)$ satisfies the assumptions of Theorem~\ref{intro-morse}, 
and $S$ is quasi-minimizing and has controlled density with respect 
to~$\til d$. Hence, the result follows.
By regularity theory, $\spt(\til S)$ is a smooth 
$n$-dimensional submanifold except for a closed singular set of Hausdorff 
dimension at most $n - 2$ (see~\cite{DeLS} for a guide to the literature).
For example, $S$ could be the current associated to an oriented $n$-flat 
in $(X,g)$ (but see also Theorem~\ref{intro-t-plateau} below). 
The primary instance of Corollary~\ref{intro-persistence} is when $(X,g)$
is the universal covering of a compact manifold of nonpositive sectional
curvature such that $(X,g)$ contains no $(n+1)$-flat, and $\til g$ is
the lift of an arbitrary metric on the quotient. 

Morse's result was generalized in various directions to
surfaces of arbitrary dimension and codimension in spaces of negative 
curvature~\cite{BanL, Gro3, Lab, Lan0, Lan2} and to totally geodesic 
hyperplanes in some product spaces~\cite{Lan1}. 
There is a parallel development based on periodicity 
(rather than hyperbolicity) and limited to codimension one,
starting with the work of Hedlund~\cite{Hed} on the two-dimensional torus 
and including the investigation of laminations of compact Riemannian 
manifolds by minimal hypersurfaces; see~\cite{AueB, CafD, Mos} and 
the references therein. 
Corollary~\ref{intro-persistence} is now the first result in this area
for higher rank and arbitrary codimension.

The tools developed so far enable us further to introduce visual metrics
on sets of asymptote classes of local $n$-cycles, in analogy with the
usual metrization of the visual boundary of a geodesic Gromov hyperbolic
space. Let $X$ be a proper metric space satisfying condition~{\rm (CI$_n$)}
for $n = \asrk(X) \ge 1$. We consider the group
\[
\bZ_{n,\loc}^\infty(X) := \{S \in \bZ_{n,\loc}(X): \Gi(S) < \infty\}
\]
and the quotient space $\cZ X := \bZ_{n,\loc}^\infty(X)/{\sim_\F}$ of
{\em $F$-asymptote classes}, where $S \sim_\F S'$ if and only if
$\Fi(S - S') = 0$. Making use of the existence of area-minimizers
in each class $[S]$, we define an analog of the Gromov product of two
points at infinity and show that for any constants $C > 0$ and $a \ge 0$,
the set $\cZ_{C,a}X$ of all classes represented by some element with
$(C,a)$-controlled density admits an analog of Gromov's $\delta$-inequality
(Proposition~\ref{prop:d-ineq}) and carries a family of visual metrics,
with respect to which $\cZ_{C,a}X$ is compact;
see Theorem~\ref{thm:visual-metrics}.

\subsection{Asymptotic geometry of local cycles}
\label{subsect:asymp}
\mbox{}
For the remainder of the introduction, we will be mainly concerned 
with asymptotic properties of local $n$-cycles in spaces of asymptotic 
rank $n \ge 2$, and relations with the ideal boundary of $X$.
For this reason we assume in Sections~\ref{subsect:asymp} and~\ref{subsect:qi}
that $X$ is a $\CAT(0)$ space, so that we may make use of the boundary at
infinity $\di X$ and the compactification $\olX := X \cup \di X$ --- both
equipped with the cone topology --- as well as the Tits boundary $\bT X$ and
the Tits cone $\CT X$.  As mentioned earlier, all of the results discussed
here hold more generally if $X$ is a proper metric space equipped with a
convex bicombing, and the respective statements will be given in the
body of the paper.  

A point in $\di X$ is an asymptote class of unit speed rays in $X$. 
The Tits cone $\CT X$ may be defined as the set of asymptote classes of 
rays $\rho \colon \R_+ \to X$ of arbitrary (constant) speed 
$s \ge 0$, endowed with the metric $\dT$, where 
\[
\dT([\rho],[\rho']) = \lim_{t \to\infty} \frac1t \,d(\rho(t),\rho'(t))
\]
is the asymptotic slope of the convex function $t \mapsto d(\rho(t),\rho'(t))$.
For every $p \in X$ there is a canonical $1$-Lipschitz map 
\[
\can_p \colon \CT X \to X
\] 
such that $\can_p([\rho]) = \rho(1)$ for every ray $\rho$ with $\rho(0) = p$.
The Tits boundary $\bT X$ is the unit sphere in $\CT X$ and agrees with
$\di X$ as a set, but is endowed with the finer topology induced by $\dT$.
With respect to the (equivalent) angle metric $0 \le \aT \le \pi$
characterized by the relation $2\sin(\aT(u,v)/2) = \dT(u,v)$,
$\bT X$ is a $\CAT(1)$ space, and $\CT X$ agrees with the Euclidean cone 
over $(\bT X,\aT)$ and is thus a $\CAT(0)$ space. If $X$ is a symmetric space
of non-compact type or a thick Euclidean building of rank $n \ge 2$, 
then $(\bT X,\aT)$ has the structure of a thick $(n-1)$-dimensional 
spherical building.

For a local cycle $S \in \bZ_{n,\loc}(X)$, we let 
\[
\Lam(S) \sub \di X
\]
denote the {\em limit set\/} of $\spt(S)$, that is, the set of all points in 
$\di X$ belonging to the closure of $\spt(S)$ in $\olX$. 
We say that $S$ is {\em conical\/} with respect to some point $p \in X$ 
if $S$ is invariant, for every $\lam \in (0,1)$, under the $\lam$-Lipschitz 
map $h_{p,\lam} \colon X \to X$ that takes $x$ to $\si_{px}(\lam)$, where 
$\si_{px} \colon [0,1] \to X$ denotes the geodesic from~$p$ to~$x$.

The following result summarizes Theorem~\ref{thm:conical-repr},
Proposition~\ref{prop:lim-sets}, and Theorem~\ref{thm:f-classes} for the case
when $X$ is $\CAT(0)$. It shows in particular that the group
$\cZ X = \bZ_{n,\loc}^\infty(X)/{\sim_\F}$ of $\F$-asymptote classes
is canonically isomorphic to the group of integral $(n-1)$-cycles in $\bT X$.

\begin{theorem}[Tits boundary] \label{intro-f-classes} \mbox{}
Let $X$ be a proper $\CAT(0)$ space with $\asrk(X) = n \ge 2$. 
If $S \in \bZ_{n,\loc}^\infty(X)$, then for every $p \in X$ there is a unique
representative $S_{p,0} \in [S] \in \cZ X$ that is conical with respect to $p$,
and there is a unique local cycle $\Sig \in \bZ_{n,\loc}(\CT X)$ such that
$\can_{p\#}\Sig = S_{p,0}$ for all $p \in X$; furthermore, $\Sig$ is conical
with respect to the cone vertex $o$, and the spherical slice 
$\d(\Sig \on \B{o}{1})$ defines an element 
$\bT S = \bT[S] \in \bZ_{n-1,\cs}(\bT X)$.
This yields an isomorphism
\[
\bT \colon \cZ X \to \bZ_{n-1,\cs}(\bT X).
\]
For every $p \in X$, $\spt(\bT S) = \Lam(S_{p,0}) \sub \Lam(S)$, 
and if $S$ is quasi-minimizing, then $\Lam(S_{p,0}) = \Lam(S)$.
\end{theorem}

We call $\bT S = \bT[S]$ the {\em Tits boundary\/} of $S$ or $[S]$, 
respectively. Due to the rank assumption, $\bI_{m,\cs}(\bT X) = \{0\}$ for 
$m > n-1$, thus $\bZ_{n-1,\cs}(\bT X)$ agrees with the homology group
$\bH_{n-1,\cs}(\bT X)$ of integral currents, which is in turn  
isomorphic to the usual singular homology group $H_{n-1}(\bT X)$
(see~\cite{RieS}). Hence, $\cZ X$ is isomorphic to $H_{n-1}(\bT X)$.

Regarding the last assertion of Theorem~\ref{intro-f-classes}, we will 
in fact show that every quasi-minimizer $S \in \bZ_{n,\loc}^\infty(X)$ is 
asymptotically conical in that $\spt(S)$ and $\spt(S_{p,0})$ lie within 
``sublinear'' distance from each other, in terms of the distance to $p$;
see (the proof of) Theorem~\ref{thm:visibility} and 
Theorem~\ref{thm:conicality}. The following key result, which is part of 
the first of these two theorems, may be viewed as an analog of the 
{\em visibility axiom} for a Hadamard manifold $X$. This postulates that 
for all $p \in X$ and $\eps > 0$ there is an $r = r(p,\eps)$ such that 
every geodesic segment $[x,y] \sub X$ at distance at least $r$ from $p$ 
subtends an angle $\angle_p(x,y) \le \eps$ at $p$; 
see Definition~4.2 and Remark~4.3 in~\cite{EbeO} 
(compare pp.~294ff and~400 in \cite{BriH} for a discussion in the context 
of $\CAT(0)$ spaces).

\begin{theorem}[visibility property] \label{intro-visibility} \mbox{}
Let $X$ be a proper $\CAT(0)$ space with $\asrk(X) = n \ge 2$. 
Suppose that $S \in \bZ_{n,\loc}(X)$ is $(Q,a)$-quasi-minimizing and satisfies 
$\G_{p,r}(S) \le C$ for some $p \in X$ and for all $r > a$. 
Then for every $\eps > 0$ there exists a constant $r_\eps = r_\eps(X,Q,C,a)$
such that for every $x \in \spt(S)$ with $d(p,x) \ge r_\eps$ there exists a
unit speed ray~$\rho$ emanating from~$p$ and representing a point in $\Lam(S)$
such that\/ $\inf_{t \ge 0}d(x,\rho(t)) < \eps\,d(p,x)$.
\end{theorem}

The next result solves an asymptotic Plateau problem (see also
Theorem~\ref{thm:a-plateau} and Theorem~\ref{thm:t-plateau}).
This may be viewed as a higher rank analog of the property that any pair
of distinct points in the visual boundary $\di X$ can be joined by a
geodesic line in $X$. 

\begin{theorem}[minimizer with prescribed Tits data]
\label{intro-t-plateau}\mbox{}
Let $X$ be a proper $\CAT(0)$ space with $\asrk(X) = n \ge 2$.
Then for every cycle $R \in \bZ_{n-1,\cs}(\bT X)$ there exists an 
area-minimizing local cycle $S \in \bZ_{n,\loc}^\infty(X)$ with $\bT S = R$. 
Every such $S$ satisfies $\Lam(S) = \spt(R)$ and
$\G_{p,r}(S) \le \Gi(S) = \M(R)/n$ for all $p \in X$ and $r > 0$,
in particular $S$ has controlled density,
and $\M(R)/n = \Gi(S) \ge \om_n$ whenever $R \ne 0$. 
\end{theorem}

Here $\om_n$ denotes the Lebesgue measure of the unit ball in $\R^n$.
The equality $\Gi(S) = \om_n$ clearly holds if $S$ is the current
associated with an oriented $n$-flat in $X$. 

For ambient spaces of strictly negative curvature, minimal varieties of 
arbitrary dimension and codimension with prescribed limit sets were first
constructed in~\cite{And1,And2}.  
We refer to~\cite{CasHR,Gab,Gro3,Lan2} and the references therein 
for some generalizations and variations of these results.
In Section~8.3 of~\cite{Gro1}, Gromov raised the question about the 
asymptotic behavior of minimal varieties in spaces of nonpositive 
curvature and symmetric spaces in particular.
Theorem~\ref{intro-t-plateau} addresses this for $n$-currents in spaces 
of rank~$n$. 

Theorem~\ref{intro-f-classes} and Theorem~\ref{intro-t-plateau} show 
in particular that the three classes of conical, minimizing, or 
quasi-minimizing elements of $\bZ_{n,\loc}^\infty(X)$ give rise to the same 
collection of limit sets, which also agrees with 
$\{\spt(R): R \in \bZ_{n-1,\cs}(\bT X)\}$. We denote this canonical 
class of subsets of $\di X$ by $\cL X$.
From Theorem~\ref{intro-visibility} and Theorem~\ref{intro-t-plateau}
we deduce the following result (see Theorem~\ref{thm:dense-orbit},
and~\cite{GurS} for a closely related discussion).

\begin{theorem}[dense orbit] \label{intro-dense-orbit} \mbox{}
Let $X$ be a proper $\CAT(0)$ space of asymptotic rank $n \ge 2$, 
and suppose that\/ $\Gam$ is a cocompact group of isometries of $X$.
Then, for every non-empty set $\Lam \in \cL X$, the orbit of $\Lam$ under
the action of\/ $\Gam$, extended to $\olX = X \cup \di X$, is dense in 
$\di X$ (with respect to the cone topology).
\end{theorem}

\subsection{Applications to quasi-isometries} \label{subsect:qi}
\mbox{}
We recall that every quasi-isometric embedding $f \colon X \to \bar X$ 
between two geodesic Gromov hyperbolic spaces naturally induces a 
topological embedding $\di f \colon \di X \to \di \bar X$ of their visual 
boundaries. In fact, $\di f$ is a power quasi-symmetric
(and hence bi-H\"older)
embedding with respect to any pair of visual metrics on $\di X$ and
$\di \bar X$~\cite{BonS, BuyS}. The proof is based on the Morse Lemma.

We now consider a quasi-isometric embedding $f \colon X \to \bar X$ between 
two proper $\CAT(0)$ spaces of asymptotic rank $n \ge 2$.
Theorem~\ref{intro-f-classes} and Theorem~\ref{intro-t-plateau}
show that every $(n-1)$-cycle in $\bT X$ corresponds to an $\F$-asymptote 
class in $X$ which is represented by a minimizing local $n$-cycle  
with controlled density. Furthermore, for any quasi-minimizer 
$S \in \bZ_{n,\loc}(X)$ with controlled density, 
there exists a Lipschitz map $g \colon X \to \bar X$ such that 
$\sup_{x \in \spt(S)}d(f(x),g(x)) < \infty$, and this map takes $S$ to a local 
cycle $g_\#S \in \bZ_{n,\loc}(\bar X)$ that is again quasi-minimizing and has
controlled density (see Proposition~\ref{prop:qi-inv}). The ambiguity in the 
choice of $g$ disappears on the level of $\F$-asymptote classes. 
In fact, there is a unique monomorphism
\[
\cZ f \colon \cZ X \to \cZ \bar X
\]
such that $\cZ f\,[S] = [g_\#S]$ whenever $S \in \bZ_{n,\loc}^\infty(X)$ 
and $g$ is a Lipschitz map as above; see Theorem~\ref{thm:mapping-as-classes}.
Since classes in $\cZ \bar X$ are, in turn, in bijective correspondence 
with $(n-1)$-cycles in $\bT \bar X$, this provides a canonical map from 
$\bZ_{n-1,\cs}(\bT X)$ into $\bZ_{n-1,\cs}(\bT \bar X)$ induced by $f$.

\begin{theorem}[mapping Tits cycles] \label{intro-tits-cycles} \mbox{}
Let $X,\bar X$ be two proper $\CAT(0)$ spaces of asymptotic rank $n \ge 2$, 
and suppose that $f \colon X \to \bar X$ is a quasi-isometric embedding.
Then there exists a unique monomorphism
\[
f_\T \colon \bZ_{n-1,\cs}(\bT X) \to \bZ_{n-1,\cs}(\bT \bar X)
\]
such that $f_\T(\bT S) = \bT(g_\#S)$ whenever $S \in \bZ_{n,\loc}^\infty(X)$ 
and $g \colon X \to \bar X$ is a Lipschitz map with 
$\sup_{x \in \spt(S)}d(f(x),g(x)) < \infty$. If $f$ is a quasi-isometry, then
$f_\T$ is an isomorphism.
\end{theorem}

In particular, by the remark after Theorem~\ref{intro-f-classes}, if 
$X$ and $\bar X$ are quasi-isometric, then $H_{n-1}(\bT X)$ are 
$H_{n-1}(\bT \bar X)$ are isomorphic.

The next result describes the effect of a quasi-isometry on intersection 
patterns of limit sets. We let $\cP(\cL X)$ denote the set,
partially ordered by inclusion, of all intersections 
$\bigcap_{i=1}^k \Lam_i$ such that $1 \le k < \infty$ and $\Lam_i \in \cL X$.
Recall that $\cL X = \{\spt(R): R \in \bZ_{n-1,\cs}(\bT X)\}$.

\begin{theorem}[mapping limit sets] \label{intro-lim-sets} \mbox{}
Let $f \colon X \to \bar X$ be a quasi-isometry between two 
proper $\CAT(0)$ spaces of asymptotic rank $n \ge 2$.
Then there exists an isomorphism (order preserving bijection)
\[
\cL f \colon \cP(\cL X) \to \cP(\cL \bar X)
\]
such that $\cL f(\spt(R)) = \spt(f_\T(R))$ for all $R \in \bZ_{n-1,\cs}(\bT X)$.
Furthermore, for every $P \in \cP(\cL X)$ and $\bar P := \cL f(P)$ there 
is a pointed $L$-bi-Lipschitz homeomorphism between the cones 
$\R_+P \sub \CT X$ and $\R_+\bar P \sub \CT \bar X$, where $L$ is the 
multiplicative quasi-isometry constant of $f$.
\end{theorem}

This follows from Theorem~\ref{thm:map-lim-sets}.
For a higher rank symmetric space $X$ of non-compact type, 
the partially ordered set $\cP(\cL X)$ contains the simplicial building 
structure of $\bT X$. This structure is pivotal in the proofs of both 
Mostow's rigidity theorem~\cite{Mos} and the general non-equivariant
rigidity theorem~\cite{EskF, KleL} for such spaces. Indeed, the latter may 
be derived relatively quickly from Theorem~\ref{intro-lim-sets} in 
conjunction with Tits' work~\cite{Tit} and the case $k = 1$ of the 
following result.

\begin{theorem}[structure of quasiflats] \label{intro-qflats} \mbox{}
Let $X$ be a proper $\CAT(0)$ space of asymptotic rank $n \ge 2$, and let 
$f \colon \R^n \to X$ be an $(L,a)$-quasi-isometric embedding with limit
set $\Lam := \di(f(\R^n))$. Then the cone $\R_+\Lam \sub \CT X$ 
is $L$-bi-Lipschitz homeomorphic to $\R^n$.
Suppose that $\Lam$ is contained in the union of the limit sets of 
$k$ $n$-flats in $X$ with a common point $p \in X$, and let\/ 
$\C_p(\Lam) \sub X$ denote the geodesic cone from $p$ over~$\Lam$.
Then $f(\R^n)$ is within distance at most $b$ from $\C_p(\Lam)$ for some 
constant~$b$ depending only on $X,L,a,k$. In the case $k = 1$, $f(\R^n)$ is 
at Hausdorff distance at most $b$ from the flat $\C_p(\Lam)$.
\end{theorem}

We refer to Theorem~\ref{thm:struct-qflats} and the comments thereafter 
for a more general statement and some implications.


\section{Preliminaries} \label{sect:prelim}

\subsection{Metric notions} \label{subsect:metric}
\mbox{}
Let $X = (X,d)$ be a metric space. We write 
\[
\B{p}{r} := \{x \in X : d(p,x) \le r\}, \quad
\Sph{p}{r} := \{x \in X : d(p,x) = r\}
\] 
for the closed ball and sphere with radius $r \ge 0$ and center $p \in X$.

A set $N \sub X$ is called {\em $\del$-separated}, for a constant
$\del \ge 0$, if $d(x,y) > \del$ for every pair of distinct points
$x,y \in N$. 
For $A \sub X$, we call a subset $N \sub A$ a {\em $\del$-net in $A$} 
if the family of all balls $\B{x}{\del}$ with $x \in N$ covers $A$. 
Every maximal (with respect to inclusion) $\del$-separated subset of $A$
is a $\del$-net in $A$.

A map $f \colon X \to Y$ into another metric space $Y = (Y,d)$ is
{\em $L$-Lipschitz}, for a constant $L \ge 0$, 
if $d(f(x),f(x')) \le L\,d(x,x')$ for all $x,x' \in X$.
The smallest such $L$ is the {\em Lipschitz constant\/} $\Lip(f)$ of~$f$. 
The map $f \colon X \to Y$ is an {\em $L$-bi-Lipschitz embedding}
if $L^{-1}d(x,x') \le d(f(x),f(x')) \le L\,d(x,x')$ for all $x,x' \in X$.
For an $L$-Lipschitz function $f \colon A \to \R$ defined on a set 
$A \sub X$,
\[
\bar f(x) := \sup\{f(a) - L\,d(a,x): a \in A\} \quad \text{($x \in X$)}
\]
defines an $L$-Lipschitz extension $\bar f \colon X \to \R$ of $f$.
Every $L$-Lipschitz map $f \colon A \to \R^n$, $A \sub X$, admits
a $\sqrt{n}L$-Lipschitz extension $\bar f \colon X \to \R^n$.

A map $f \colon X \to Y$ between two metric spaces is called an
{\em $(L,a)$-quasi-isometric embedding}, for constants
$L \ge 1$ and $a \ge 0$, if
\[
L^{-1} d(x,x') - a \le d(f(x),f(x')) \le L\,d(x,x') + a
\]
for all $x,x' \in X$.
A {\em quasi-isometry} $f \colon X \to Y$ has the additional property
that $Y$ is within finite distance of the image of $f$.
An {\em $(L,a)$-quasi-geodesic segment\/} in $X$ is the image of an 
$(L,a)$-quasi-isometric embedding of some compact interval. 
An $n$-dimensional {\em quasiflat\/} in $X$ is the image of a quasi-isometric
embedding of $\R^n$.

\subsection{Currents in metric spaces} \label{subsect:currents}
\mbox{}
Currents of finite mass in complete metric spaces were introduced by Ambrosio 
and Kirchheim in~\cite{AmbK}. Here we will mainly work with the localized 
variant of this theory for locally compact metric spaces, as described 
in~\cite{Lan3}. However, to avoid certain technicalities, we will
assume throughout that the underlying metric space $X$ is proper,
hence complete and separable.

For every integer $n \ge 0$, let $\cD^n(X)$ denote the set of all 
$(n+1)$-tuples $(\pi_0,\ldots,\pi_n)$ of real valued functions on $X$ such 
that $\pi_0$ is Lipschitz with compact support $\spt(\pi_0)$ and 
$\pi_1,\dots,\pi_n$ are locally Lipschitz. 
(In the case that $X = \R^N$ and the entries of $(\pi_0,\ldots,\pi_n)$ are 
smooth, this tuple should be thought of as representing the compactly supported 
differential $n$-form $\pi_0\,d\pi_1 \wedge \ldots \wedge d\pi_n$.)
An {\em $n$-dimensional current\/} $S$ in $X$ is a function 
$S \colon \cD^n(X) \to \R$ satisfying the following three conditions:
\ben
\item
$S$ is $(n+1)$-linear;
\item 
$S(\pi_{0,k},\ldots,\pi_{n,k}) \to S(\pi_0,\ldots,\pi_n)$
whenever $\pi_{i,k} \to \pi_i$ pointwise on $X$ with 
$\sup_k\Lip(\pi_{i,k}|_K) < \infty$ for every compact set $K \sub X$ 
($i = 0,\dots,n$) and with $\bigcup_k\spt(\pi_{0,k}) \sub K$ for some such set; 
\item
$S(\pi_0,\ldots,\pi_n) = 0$ whenever one of the functions
$\pi_1,\ldots,\pi_n$ is constant on a neighborhood of $\spt(\pi_0)$.
\een
We write $\cD_n(X)$ for the vector space of all $n$-dimensional currents 
in $X$. The defining conditions already imply that every $S \in \cD_n(X)$ is 
alternating in the last $n$ arguments and satisfies a product derivation 
rule in each of these. The definition is further motivated by the fact that
every function $w \in L^1_\loc(\R^n)$ induces a current 
$\bb{w} \in \cD_n(\R^n)$ defined by
\[
\bb{w}(\pi_0,\dots,\pi_n) 
:= \int w \pi_0\det\bigl[\d_j\pi_i\bigr]_{i,j = 1}^n \,dx
\]
for all $(\pi_0,\dots,\pi_n) \in \cD^n(\R^n)$, where the partial derivatives 
$\d_j\pi_i$ exist almost every\-where according to Rademacher's theorem. 
Note that this just corresponds to the integration of the differential form
$\pi_0\,d\pi_1 \wedge \ldots \wedge d\pi_n$ over $\R^n$, weighted by $w$.
For the characteristic function $\chi_W$ of a Borel set $W \sub \R^n$,
we put $\bb{W} := \bb{\chi_W}$.
(See Section~2 in~\cite{Lan3} for details.) 

\subsection{Support, push-forward,  and boundary}
\mbox{}
For every $S \in \cD_n(X)$ there exists a smallest closed subset of $X$,
the {\em support\/} $\spt(S)$ of $S$, such that the value 
$S(\pi_0,\ldots,\pi_n)$ depends only on the restrictions of 
$\pi_0,\dots,\pi_n$ to this set.
For a proper Lipschitz map $f \colon X \to Y$ into another proper
metric space $Y$, the {\em push-forward\/} $f_\#S \in \cD_n(Y)$ is defined
simply by
\[
(f_\#S)(\pi_0,\ldots,\pi_n) := S(\pi_0 \circ f,\ldots,\pi_n \circ f)
\]
for all $(\pi_0,\ldots,\pi_n) \in \cD^n(Y)$. This definition can be
extended to proper Lipschitz maps $f \colon \spt(S) \to Y$ via
appropriate extensions of the functions $\pi_i \circ f$ to $X$. 
In either case, $\spt(f_\#S) \sub f(\spt(S))$.
For $n \ge 1$, the {\em boundary} $\d S \in \cD_{n-1}(X)$ of 
$S \in \cD_n(X)$ is defined by
\[
(\d S)(\pi_0,\dots,\pi_{n-1}) := S(\tau,\pi_0,\dots,\pi_{n-1})
\]
for all $(\pi_0,\ldots,\pi_{n-1}) \in \cD^{n-1}(X)$ and for any compactly 
supported Lipschitz function $\tau$ that is identically $1$ on some 
neighborhood of $\spt(\pi_0)$. If $\til\tau$ is another such function, 
then $\pi_0$ vanishes on a neighborhood of $\spt(\tau - \til\tau)$ and 
$\d S$ is thus well-defined by~(1) and~(3).
Similarly one can check that $\d \circ \d = 0$. The inclusion
$\spt(\d S) \sub \spt(S)$ holds, and $f_\#(\d S) = \d(f_\# S)$ for
$f \colon \spt(S) \to Y$ as above.
(See Section~3 in \cite{Lan3}.)

\subsection{Mass} \label{subsect:mass}
\mbox{}
Let $S \in \cD_n(X)$. A tuple $(\pi_0,\ldots,\pi_n) \in \cD^n(X)$
will be called {\em normalized\/} if the restrictions of
$\pi_1,\ldots,\pi_n$ to the compact set $\spt(\pi_0)$ are $1$-Lipschitz.
For an open set $U \sub X$, 
the {\em mass} $\|S\|(U) \in [0,\infty]$ of $S$ in $U$ is then defined as
the supremum of $\sum_j S(\pi_{0,j},\ldots,\pi_{n,j})$
over all finite families of normalized tuples
$(\pi_{0,j},\ldots,\pi_{n,j}) \in \cD^n(X)$ such that
$\bigcup_j \spt(\pi_{0,j}) \sub U$ and $\sum_j |\pi_{0,j}| \le 1$.
Note that $\|S\|(U) > 0$ if and only if $U \cap \spt(S) \ne \es$.
This induces a regular Borel measure $\|S\|$ on $X$, whose
{\em total mass} $\|S\|(X)$ is denoted by $\M(S)$. For Borel sets
$W,A \sub \R^n$, $\|\bb{W}\|(A)$ equals the Lebesgue measure of $W \cap A$.
If $T \in \cD_n(X)$ is another $n$-current in $X$, then clearly
\[
\|S + T\| \le \|S\| + \|T\|.
\]
We will now assume that the measure $\|S\|$ is locally finite (and hence
finite on bounded sets, as $X$ is proper). Then it can be shown that
\[
|S(\pi_0,\ldots,\pi_n)| \le \int_X |\pi_0|\,d\|S\|
\]
for every normalized tuple $(\pi_0,\ldots,\pi_n) \in \cD^n(X)$.
This inequality allows to define the {\em restriction} $S \on A \in \cD_n(X)$
of $S$ to a Borel set $A \sub X$ by
\[
(S \on A)(\pi_0,\dots,\pi_n) := \lim_{k \to \infty} S(\tau_k,\pi_1,\ldots,\pi_n)
\]
for any sequence of compactly supported Lipschitz functions $\tau_k$ converging
in $L^1(\|S\|)$ to $\chi_A \pi_0$. The measure $\|S \on A\|$ equals
the restriction $\|S\| \on A$ of $\|S\|$.
If $f \colon \spt(S) \to Y$ is a proper $L$-Lipschitz map into a proper 
metric space $Y$ and $B \sub Y$ is a Borel set, then
$(f_\#S) \on B = f_\#(S \on f^{-1}(B))$ and
\[
\|f_\#S\|(B) \le L^n\,\|S\|(f^{-1}(B)).
\]
(See Section~4 in~\cite{Lan3}.)

\subsection{Integral currents}
\mbox{}
A current $S \in \cD_n(X)$ is called {\em locally integer rectifiable\/} 
if the measure $\|S\|$ is locally finite and concentrated on the union of
countably many Lipschitz images of compact subsets of $\R^n$,
and the following integer multiplicity condition holds: for every
Borel set $A \sub X$ with compact closure and every Lipschitz map
$\phi \colon X \to \R^n$, the current $\phi_\#(S \on A) \in \cD_n(\R^n)$
is of the form $\bb{w}$ for some {\em integer valued\/}
$w = w_{A,\phi} \in L^1(\R^n)$.
Then $\|S\|$ turns out to be absolutely continuous with 
respect to $n$-dimensional Hausdorff measure. 
Furthermore, push-forwards and restrictions to Borel sets
of locally integer rectifiable currents are again locally integer rectifiable.

A current $S \in \cD_n(X)$ is called a {\em locally integral current\/} 
if $S$ is locally integer rectifiable and, for $n \ge 1$, $\d S$ satisfies 
the same condition.
(Remarkably, this is the case already when $\|\d S\|$ is locally finite,
provided $S$ is locally integer rectifiable; see Theorem~8.7 in~\cite{Lan3}.)
This yields a chain complex of abelian groups $\bI_{n,\loc}(X)$.
We write $\bI_{n,\cs}(X)$ for the respective subgroups of
{\em integral currents\/} with compact support.
For example, if $\Del \sub \R^n$ is an $n$-simplex and $f \colon \Del \to X$
is a Lipschitz map, then $f_\#\bb{\Del} \in \bI_{n,\cs}(X)$.
Thus every singular Lipschitz chain in $X$ with integer coefficients
defines an element of $\bI_{n,\cs}(X)$.
For $X = \R^N$, there is a canonical chain isomorphism from $\bI_{*,\cs}(\R^N)$
to the chain complex of ``classical'' integral currents in $\R^N$
originating from~\cite{FedF}.

For $n \ge 1$, we let $\bZ_{n,\loc}(X) \sub \bI_{n,\loc}(X)$ and 
$\bZ_{n,\cs}(X) \sub \bI_{n,\cs}(X)$ denote the subgroups of currents
with boundary zero. An element of $\bI_{0,\cs}(X)$ is an integral
linear combination of currents of the form $\bb{x}$, where
$\bb{x}(\pi_0) = \pi_0(x)$ for all $\pi_0 \in \cD^0(X)$. We let
$\bZ_{0,\cs}(X) \sub \bI_{0,\cs}(X)$ denote the subgroup of linear combinations
whose coefficients sum up to zero. The boundary of a current in $\bI_{1,\cs}(X)$
belongs to $\bZ_{0,\cs}(X)$.
Given $Z \in \bZ_{n,\cs}(X)$, for $n \ge 0$, we will call
$V \in \bI_{n+1,\cs}(X)$ a {\em filling} of $Z$ if $\d V = Z$. 

\subsection{Slicing} \label{subsect:slicing}
\mbox{}
Let $S \in \bI_{n,\loc}(X)$ be a locally integral current of dimension
$n \ge 1$. Note that both $\|S\|$ and $\|\d S\|$ are locally finite
(that is, $S$ is {\em locally normal}, see Section~5 in~\cite{Lan3}).
Let $\rho \colon X \to \R$ be a Lipschitz function,
and let $B_s := \{\rho \le s\}$ denote the closed
sublevel set for $s \in \R$.
The corresponding {\em slice} of $S$ is the $(n-1)$-dimensional current
\[
\la S,\rho,s \ra := \d(S \on B_s) - (\d S) \on B_s
\]
with support in $\{\rho = s\} \cap \spt(S)$.
We will use this construction exclusively in the case that
$B_s \cap \spt(S)$ is compact for all $s$ (typically $\rho$ will be the
distance function to a point in $X$). Then, for almost every $s$,
$\la S,\rho,s \ra \in \bI_{n-1,\cs}(X)$ and hence $S \on B_s \in \bI_{n,\cs}(X)$.
Furthermore, for $a < b$, the coarea inequality
\[
\int_a^b \M(\la S,\rho,s \ra) \,ds \le \Lip(\rho)\,\|S\|(\{a < \rho < b\})
\]
holds. In particular, for every $c \in (0,b-a]$, the set of all 
$s \in (a,b)$ such that
\[
\M(\la S,\rho,s \ra) \le c^{-1} \Lip(\rho)\,\|S\|(\{a < \rho < b\})
\]
has measure $> b - a - c$. (See Section~6 and Theorem~8.5 in~\cite{Lan3}.)

\subsection{Homotopies, cones, and isoperimetric inequality} 
\label{subsect:homotopies}
\mbox{}
Let $\bb{0,1} \in \bI_{1,\cs}([0,1])$ denote the current defined by
\[
\bb{0,1}(\pi_0,\pi_1) := \int_0^1 \pi_0(t) \pi'_1(t) \,dt.
\]
Note that $\d\bb{0,1} = \bb{1} - \bb{0}$.
We endow $[0,1] \times X$ with the usual $l_2$~product metric.
There exists a canonical product construction
\[
S \in \bI_{n,\cs}(X) \leadsto 
\bb{0,1} \times S \in \bI_{n+1,\cs}([0,1] \times X)
\]
for all $n \ge 0$.
Suppose now that $Y$ is another proper metric space,
$h \colon [0,1] \times X \to Y$ is a Lipschitz homotopy from 
$f = h(0,\cdot)$ to $g = h(1,\cdot)$, and $S \in \bI_{n,\cs}(X)$. 
Then $h_\#(\bb{0,1} \times S)$ is an element of $\bI_{n+1,\cs}(Y)$ 
with boundary
\[
\d\,h_\#(\bb{0,1} \times S) = g_\# S - f_\# S - h_\#(\bb{0,1} \times \d S)
\]
(for $n = 0$ the last term is zero.)
If $h(t,\cdot)$ is $L$-Lipschitz for every $t$, and $h(\cdot,x)$ is a
geodesic of length at most $D$ for every $x \in \spt(S)$, then
\[
\M(h_\#(\bb{0,1} \times S)) \le (n+1) L^n D\,\M(S).  
\]
(See Section~2.3 in~\cite{Wen1}.) An important special case of this is
when $R \in \bZ_{n,\cs}(X)$ and $h(\cdot,x) = \si_{px}$ is a geodesic from
some fixed point $p \in X$ to $x$ for every $x \in \spt(R)$.   
Then $h_\#(\bb{0,1} \times R) \in \bI_{n+1,\cs}(X)$ is the 
{\em cone from $p$ over $R$} determined by this family of geodesics,
whose boundary is $R$. 
If the family of geodesics satisfies the convexity condition
\[
d(h(t,x),h(t,x')) = d(\si_{px}(t),\si_{px'}(t)) \le t\, d(x,x')
\]
for all $x,x' \in \spt(R)$ and $t \in [0,1]$, and if 
$\spt(R) \sub \B{p}{r}$, then
\[
\M(h_\#(\bb{0,1} \times R)) \le r\,\M(R).  
\]
Finally, if $X$ is a $\CAT(0)$ space, then this inequality
holds with $r/(n+1)$ in place of $r$ (see Theorem~4.1 in~\cite{Wen3}).

\begin{definition}[coning inequalities] \label{def:cone-ineq}
For $n \ge 0$, we say that $X$ satisfies {\em condition~{\rm (CI$_n$)}} if 
for every $k \in \{0,\ldots,n\}$ there is a constant $c_k$ such that every
$R \in \bZ_{k,\cs}(X)$ with support in some $r$-ball possesses a filling
$S \in \bI_{k+1,\cs}(X)$ with mass 
\[
\M(S) \le c_k r\,\M(R).
\]
\end{definition}

Condition~{\rm (CI$_0$)} is satisfied if and only if $X$ is
{\em quasi-convex}, that is, there is a constant $c'_0$ such that
every pair of points $x,x'$ in $X$ can be joined by a curve of length less
than or equal to $c'_0\,d(x,x')$.

Coning inequalities are instrumental for isoperimetric filling inequalities.

\begin{theorem}[isoperimetric inequality] \label{thm:isop-ineq}
Let $n \ge 2$, and let $X$ be a proper metric space satisfying
condition~{\rm (CI$_{n-1}$)}. Then every cycle 
$R \in \bZ_{n-1,\cs}(X)$ possesses a filling $T \in \bI_{n,\cs}(X)$ 
with mass 
\[
\M(T) \le \gam\,\M(R)^{n/(n-1)}
\]
for some constant $\gam > 0$ depending only on the constants 
$c_1,\dots,c_{n-1}$ from Definition~\ref{def:cone-ineq}.
\end{theorem}

(Here the condition~{\rm (CI$_0$)} is actually not needed.)
This was shown in more general form for Ambrosio--Kirchheim currents 
in complete metric spaces in~\cite{Wen1}; 
see~Theorem~1.2 and the remark thereafter regarding compact supports. 
For earlier results of this type, see Remark~6.2 in~\cite{FedF}
and the comments after Corollary~3.4.C in~\cite{Gro1}.

\subsection{Convergence, compactness, and Plateau problem}
\label{subsect:convergence}
\mbox{}
A sequence $(S_i)$ in $\bI_{n,\loc}(X)$ is said to converge {\em weakly}
to a current $S \in \bI_{n,\loc}(X)$ if $S_i \to S$ pointwise as functionals
on $\cD^n(X)$. Then
\[
\|S\|(U) \le \liminf_{i \to \infty} \|S_i\|(U)
\]
for every open set $U \sub X$. Furthermore, weak convergence commutes
with the boundary operator and with push-forwards.

A more geometric notion of convergence, with analogous properties,
is given as follows. A sequence $(S_i)$ in $\bI_{n,\loc}(X)$ converges 
in the {\em local flat topology} to a current $S \in \bI_{n,\loc}(X)$
if for every compact set $K \sub X$ there exists a sequence
$(V_i)$ in $\bI_{n+1,\loc}(X)$ such that
\[
(\|S - S_i - \d V_i\| + \|V_i\|)(K) \to 0.
\]
This implies that $S_i \to S$ weakly.
The {\em flat distance} between two elements $S,S' \in \bI_{n,\cs}(X)$
is defined by
\[
\cF(S - S') := \inf\{ \M(S - S' - \d V) + \M(V): V \in \bI_{n+1,\cs}(X) \};
\]
this yields a metric on $\bI_{n,\cs}(X)$.

We now state the compactness theorem for locally integral currents and
minimizing locally integral currents. An element
$S \in \bI_{n,\loc}(X)$ is {\em \mbox{(area-)}minimizing\/} if
\[
\M(S \on B) \le \M(T)
\]
whenever $B \sub X$ is a Borel set such that $S \on B \in \bI_{n,\cs}(X)$
and $T \in \bI_{n,\cs}(X)$ satisfies $\d T = \d(S \on B)$. 

\begin{theorem}[compactness] \label{thm:cptness}
Let $X$ be a proper metric space, and let $n \ge 1$.
Suppose that $(S_i)$ is a sequence in $\bI_{n,\loc}(X)$ such that
\[
\sup_i (\|S_i\| + \|\d S_i\|)(K) < \infty
\]
for every compact set $K \sub X$.
\ben
\item
There is a subsequence $(S_{i_j})$ that converges weakly to 
a current $S \in \bI_{n,\loc}(X)$.
\item
Suppose, in addition, that $X$ satisfies condition~{\rm (CI$_n$)}. 
Then there is a subsequence $(S_{i_j})$ that converges in the local
flat topology to a current $S \in \bI_{n,\loc}(X)$. 
If each $S_i$ is area-minimizing, then so is $S$.
\een
\end{theorem}

For~(2), a uniformly local version of condition~{\rm (CI$_n$)} suffices;
compare the assumptions in~\cite{Wen2}.  

\begin{proof}
For~(1), see Theorem~8.10 in~\cite{Lan3}.

For the proof of~(2), pick a base point $p \in X$. By passing to a further 
subsequence, denoted again by $(S_{i_j})$, one can arrange that there exists 
a sequence of radii $0 < r_k \uparrow \infty$ such that for every
$B_k := \B{p}{r_k}$, the restrictions $S_{i_j} \on B_k$ and $S \on B_k$
are in $\bI_{n,\cs}(X)$,
\[
\sup_j \bigl(\M(S_{i_j} \on B_k) 
+ \M(\d(S_{i_j} \on B_k))\bigr) < \infty,
\]
and $S_{i_j} \on B_k \to S \on B_k$ weakly, as $j \to \infty$
(see the proof of Proposition~6.6 in~\cite{Lan3}). Now, to show that 
$S_{i_j} \to S$ in the local flat topology, fix an index $k$.
Since $X$ satisfies condition~{\rm (CI$_n$)}, it follows from
Theorem~1.4 in~\cite{Wen2} that $\cF((S - S_{i_j}) \on B_k) \to 0$.
Hence, there exists a sequence $(V_j)$ in $\bI_{n+1,\cs}(X)$ such that,
for $T_j := (S - S_{i_j}) \on B_k - \d V_j$, 
\[
\M(T_j) + \M(V_j) \to 0.
\]
Since $\|S - S_{i_j} - \d V_j\|(B_k) 
\le \|T_j\|(B_k) + \|(S - S_{i_j}) \on (X \sm B_k)\|(B_k) \le \M(T_j)$
and $\|V_j\|(B_k) \le \M(V_j)$, this gives the result.

Suppose now that each $S_i$ is minimizing. 
To prove that $S$ is minimizing, it suffices to show that for every fixed $k$, 
\[
\M(S \on B_k) \le \M(T)
\] 
for all $T \in \bI_{n,\cs}(X)$ with $\d T = \d(S \on B_k)$.
Let $V_j$ and $T_j$ be given as above, and note that then
$\d(T - T_j) = \d(S_{i_j} \on B_k)$. By the minimality of $S_{i_j}$,
\[
\M(S_{i_j} \on B_k) \le \M(T - T_j) \le \M(T) + \M(T_j). 
\]
Since $S_{i_j} \on B_k \to S \on B_k$ weakly and $\M(T_j) \to 0$, 
it follows that
\[
\M(S \on B_k) \le \liminf_{j \to \infty} \M(S_{i_j} \on B_k) \le \M(T),
\]
as desired.
\end{proof}

From Theorem~\ref{thm:isop-ineq} and the first part of
Theorem~\ref{thm:cptness} one obtains a solution of the Plateau problem
in spaces with coning inequalities (compare also Theorem~10.6 in~\cite{AmbK}
and Theorem~1.6 in~\cite{Wen1}).

\begin{theorem}[minimizing filling] \label{thm:plateau}
Let $n \ge 1$, and let $X$ be a proper metric space satisfying 
condition~{\rm (CI$_{n-1}$)}. Then for every $R \in \bZ_{n-1,\cs}(X)$
there exists a filling $S \in \bI_{n,\cs}(X)$ of $R$ with mass
\[
\M(S) = \inf\{\M(S'): S' \in \bI_{n,\loc}(X),\, \d S' = R\}.
\]
Furthermore, $\spt(S)$ is within distance at most $(\M(S)/\del)^{1/n}$
from $\spt(R)$ for some constant $\del > 0$ depending only on $n$ and the
constants $c_1,\dots,c_{n-1}$ from Definition~\ref{def:cone-ineq}.
\end{theorem}

\begin{proof}
Let $\cS$ denote the set of all $S' \in \bI_{n,\loc}(X)$ with $\d S' = R$.
By condition~(CI$_{n-1}$), $\cS$ is non-empty.
Choose a sequence $(S_i)$ in $\cS$ such that 
\[
\M(S_i) \to \mu := \inf\{\M(S'): S' \in \cS\} \quad \text{for $i \to \infty$.} 
\]
By Theorem~\ref{thm:cptness}, some subsequence $(S_{i_j})$ converges weakly 
to a current $S \in \cS$, and $\M(S) \le \liminf_{j\to\infty}\M(S_{i_j})$, 
thus $\M(S) = \mu$. It is well-known that an isoperimetric inequality of
Euclidean type as in Theorem~\ref{thm:isop-ineq} leads to a lower
density bound for minimizing $n$-currents: if $x \in \spt(S)$ and $r > 0$ are
such that $\B{x}{r} \cap \spt(\d S) = \es$, then
$\|S\|(\B{x}{r}) \ge \del r^n$, where $\del := n^{-n}\gam^{1-n}$ for $n \ge 2$
and $\del := 2$ for $n = 1$ (see Theorem~9.13 in~\cite{FedF} and 
the special case $(Q,a) = (1,0)$ of Lemma~\ref{lem:density} below).
This gives the desired distance bound and shows in particular that
$\spt(S)$ is compact.
\end{proof}


\section{Quasi-minimizers} \label{sect:qmin}

We now introduce the main objects of study and discuss some basic properties 
and examples.

\begin{definition}[quasi-minimizer] \label{def:qmin}
Suppose that $X$ is a proper metric space, $n \ge 1$, and 
$Q \ge 1$, $a \ge 0$ are constants. For a closed set $Y \sub X$, a local cycle 
\[
S \in \bZ_{n,\loc}(X,Y) := \{Z \in \bI_{n,\loc}(X): \spt(\d S) \sub Y\}
\]
relative to $Y$ will be called {\em $(Q,a)$-quasi-minimizing mod\/ $Y$} if, 
for all $x \in \spt(S)$ and almost all $r > a$ such that 
$\B{x}{r} \cap Y = \es$, the inequality
\[
\M(S \on \B{x}{r}) \le Q\,\M(T)
\] 
holds whenever $T \in \bI_{n,\cs}(X)$ and $\d T = \d(S \on \B{x}{r})$
(recall that $S \on \B{x}{r} \in \bI_{n,\cs}(X)$ for almost all $r > 0$,
see Section~\ref{subsect:slicing}). A current $S \in \bI_{n,\loc}(X)$ 
is {\em $(Q,a)$-quasi-minimizing} or a {\em $(Q,a)$-quasi-minimizer} 
if $S$ is $(Q,a)$-quasi-minimizing mod $\spt(\d S)$, and we say that $S$ is 
{\em quasi-minimizing} or a {\em quasi-minimizer} if this holds for some
$Q \ge 1$ and $a \ge 0$.
\end{definition}

Obviously every minimizing $S \in \bI_{n,\loc}(X)$ is $(1,0)$-quasi-minimizing.

\begin{definition}[density/filling density] \label{def:g-f}
Suppose that $X$ is a proper metric space, $n \ge 1$, and 
$S \in \bI_{n,\loc}(X)$. For $p \in X$ and $r > 0$, put
\begin{align*}
\G_{p,r}(S) &:= \frac{1}{r^n} \|S\|(\B{p}{r}), \\
\F_{p,r}(S) &:= \frac{1}{r^{n+1}}
\inf\{\M(V) : V \in \bI_{n+1,\cs}(X),\,\spt(S-\d V) \cap \B{p}{r} = \es\}
\end{align*}  
(where $\inf \es := \infty$). Furthermore, for any $p \in X$, put 
\begin{align*}
\Gi(S) &:= \limsup_{r\to\infty} \G_{p,r}(S), \\
\Fi(S) &:= \limsup_{r\to\infty} \F_{p,r}(S);
\end{align*}
the upper limits are clearly independent of the choice of $p \in X$.
For constants $C > 0$ and $a \ge 0$, we say that $S$ has 
{\em $(C,a)$-controlled density} if $\G_{p,r}(S) \le C$ for all $p \in X$
and $r > a$, and $S$ has {\em controlled density} if this holds for 
some such constants.
\end{definition}

Note that if $\spt(\d S) \cap \B{p}{r} \ne \es$, then there is no 
$V \in \bI_{n+1,\cs}(X)$ with $\spt(S-\d V) \cap \B{p}{r} = \es$, thus 
$\F_{p,r}(S) = \infty$. Note also that if $S,S' \in \bI_{n,\loc}(X)$,
then 
\[
\G_{p,r}(S+S') \le \G_{p,r}(S) + \G_{p,r}(S')
\]
for all $p \in X$ and $r > 0$, hence $\Gi(S+S') \le \Gi(S) + \Gi(S')$.
Likewise, $\F_{p,r}$ and $\Fi$ satisfy the triangle inequality.

If $S \in \bI_{n,\loc}(X)$ has $(C,a)$-controlled density, 
then obviously $\Gi(S) \le C$. 
However, an $S \in \bZ_{n,\loc}(X)$ with $\Gi(S) < \infty$ 
need not have controlled density. For example, it is not difficult to see
that there exists a complete Riemannian metric on $\R^2$ with bounded 
curvature $|K| \le 1$ and with arbitrarily large disks of constant
curvature $-1$ such that the associated current 
$S = \bb{\R^2} \in \bZ_{2,\loc}(\R^2)$ is of this type.

\begin{lemma}[density] \label{lem:density}
Let $n \ge 1$, let $X$ be a proper metric space satisfying 
condition~{\rm (CI$_{n-1}$)}, and let $Y \sub X$ be a closed set. 
If $S \in \bZ_{n,\loc}(X,Y)$ is $(Q,a)$-quasi-minimizing mod~$Y$,
and if $x \in \spt(S)$ and $r > 2a$ are such that $\B{x}{r} \cap Y = \es$,
then
\[
\G_{x,r}(S) \ge \del
\] 
for some constant $\del > 0$ depending only on $n$, the constants
$c_1,\dots,c_{n-1}$ from Definition~\ref{def:cone-ineq}, and $Q$. 
\end{lemma}

\begin{proof} 
Let first $n \ge 2$.
Define $\mu \colon (0,r] \to \R$ by $\mu(s) := \|S\|(\B{x}{s})$.
Note that $\mu$ is non-decreasing, and $\mu > 0$ since $x \in \spt(S)$. 
For almost every $s \in (0,r)$, the derivative $\mu'(s)$ exists, and
the slice $R_s := \d(S \on \B{x}{s})$ is in $\bZ_{n-1,\cs}(X)$ and satisfies
$\M(R_s) \le \mu'(s)$. It follows from the quasi-minimality of $S$ 
and Theorem~\ref{thm:isop-ineq} (isoperimetric inequality) that for almost 
every $s \in (a,r)$, there is a filling $T_s \in \bI_{n,\cs}(X)$ of $R_s$ 
such that 
\begin{align*}
\mu(s) &= \M(S \on \B{x}{s}) \le Q\,\M(T_s) \le Q\gam\,\M(R_s)^{n/(n-1)} \\
&\le Q\gam\,\mu'(s)^{n/(n-1)}
\end{align*}
and hence $\mu'(s) \mu(s)^{(1-n)/n} \ge (Q\gam)^{(1-n)/n}$. Now integration 
from $a$ to $r$ yields $\mu(r) \ge n^{-n}(Q\gam)^{1-n}(r - a)^n$. 
Since $r - a > r/2$, this gives the result.

In the case $n = 1$, since $S$ is $(Q,a)$-quasi-minimizing mod~$Y$ and
$x \in \spt(S)$, the $0$-dimensional slice $R_s = \d(S \on \B{x}{s})$
is a non-zero integral boundary for almost every $s \in (a,r)$,
so in fact $\M(R_s) \ge 2$, and the coarea inequality gives
$\|S\|(\B{x}{r}) \ge 2(r-a) > r$.
\end{proof}

We show two direct consequences of this lemma.

\begin{lemma}[filling density] \label{lem:fill-density}
Let $n \ge 1$, let $X$ be a proper metric space satisfying
condition~{\rm (CI$_{n-1}$)}, and let $Y \sub X$ be a closed set. 
If $S \in \bZ_{n,\loc}(X,Y)$ is $(Q,a)$-quasi-minimizing mod $Y$,
and if $x \in \spt(S)$ and $r > 4a$ are such that $\B{x}{r} \cap Y = \es$, 
then 
\[
\F_{x,r}(S) \ge c
\]
for some constant $c > 0$ depending only on $n$, the constant $\del$ from 
Lemma~\ref{lem:density}, and $Q$.
\end{lemma}

\begin{proof}
Let $V \in \bI_{n+1,\cs}(X)$ be such that $\spt(S - \d V) \cap \B{x}{r} = \es$. 
For almost every $s \in (0,r)$, the slice 
$T_s := \d(V \on \B{x}{s}) - (\d V) \on \B{x}{s}$ is in $\bI_{n,\cs}(X)$,
and $\d T_s = - \d(S \on \B{x}{s})$ because
$(\d V) \on \B{x}{s} = S \on \B{x}{s}$.
By the quasi-minimality of $S$ and Lemma~\ref{lem:density}, 
for almost every $s \in (2a,r)$,
\[
Q\,\M(T_s) \ge \M(S \on \B{x}{s}) = \|S\|(\B{x}{s}) \ge \del s^n.
\]
Since $\M(V) \ge \int_{2a}^r \M(T_s) \,ds$ and $2a < r/2$, the result follows.
\end{proof}  

Recall that a subset $A$ of a metric space $X$ is {\em doubling} if there 
is a constant $M \ge 1$ such that every bounded subset $B \sub A$ can be 
covered by at most $M$ sets of diameter less than or equal to $\diam(B)/2$. 
The {\em Assouad dimension} of a set $A \sub X$ is the infimum of all 
$\alpha > 0$ for which there exists $L \ge 1$ such that for all
$\lam \in (0,1)$, every bounded set $B \sub A$ can be covered by no more
than $L \lam^{-\alpha}$ sets of diameter $\le \lam \diam(B)$.
The set $A$ has finite Assouad dimension if and only if it is doubling.
(See~\cite{Hei}.) 

\begin{lemma}[doubling] \label{lem:doubling}
Let $n \ge 1$, and let $X$ be a proper metric space
satisfying condition~{\rm (CI$_{n-1}$)}. Suppose that $S \in \bZ_{n,\loc}(X)$ is
a $(Q,a)$-quasi-minimizer with $(C,a)$-controlled density.
Then every $s$-separated subset of $\spt(S)$ with $s \ge 4a$
has Assouad dimension at most~$n$ and is thus doubling. The doubling constant
depends only on~$n$, the constant $\del$ from Lemma~\ref{lem:density},
and~$C$.
\end{lemma}  

Note that if $a = 0$, then $\spt(S)$ itself has Assouad dimension at most $n$.

\begin{proof}
Let $A \sub \spt(S)$ be an $s$-separated set, where $s \ge 4a$.
Suppose that $B \sub A$ is a bounded set with $D := \diam(B) > 0$.
Let $\lam \in (0,1)$, and let $N \sub B$ be a $(\lam D/2)$-separated 
$(\lam D/2)$-net in $B$. Put $r := \max\{s/2,\lam D/4\}$.
The balls $\B{x}{r}$ in $X$ with $x \in N$ are pairwise disjoint,
and the corresponding sets $\B{x}{2r} \cap B$ have diameter at most $\lam D$
and cover $B$. Since $r \ge 2a$, Lemma~\ref{lem:density} shows that
$\|S\|(\B{x}{r}) \ge \del r^n$ for all these balls, and their union $U$
is contained in $\B{p}{D + r}$ for any $p \in N$. Note that $D \le 4r/\lam$,
thus $\|S\|(U) \le C(5r/\lam)^n$. It follows that the covering has
cardinality $|N| \le 5^n C\del^{-1}\lam^{-n}$.
\end{proof}

The doubling property will be used in Section~\ref{sect:qi} in order to
approximate quasi-isometric embeddings by Lipschitz maps.
We will then show that if $S$ is a quasi-minimizing local $n$-cycle with
controlled density in a proper $\CAT(0)$ space (or a space
with a convex bicombing) of asymptotic rank $n \ge 2$, and if
$g$ is a Lipschitz quasi-isometric embedding of $\spt(S)$ into
another proper metric space, then $g_\#S$ is again quasi-minimizing and has
controlled density (see Proposition~\ref{prop:qi-inv}).

Here we first prove a simpler result for Lipschitz quasiflats,
which also allows for boundaries.

\begin{proposition}[Lipschitz quasiflats] \label{prop:lip-qflats}
For all $n,L \ge 1$ and $a_0 \ge 0$ there exist $Q \ge 1$, $C > 0$, and 
$a \ge 0$ such that the following holds. Let $W \sub \R^n$ be any closed set
such that the associated current $E := \bb{W}$ is in $\bZ_{n,\loc}(\R^n,\d W)$.
Suppose that $g \colon W \to X$ is a map into a proper metric space~$X$ 
such that for all $x,y \in W$,
\[
L^{-1}d(x,y) - a_0 \le d(g(x),g(y)) \le L\,d(x,y).
\]
Then $S := g_\#E \in \bZ_{n,\loc}(X,g(\d W))$ is 
$(Q,a)$-quasi-minimizing mod $g(\d W)$ and has $(C,a)$-controlled density,
furthermore $d(g(x),\spt(S)) \le a$ for all $x \in W$ with $d(x,\d W) \ge a$.
\end{proposition}

The condition on $W$ is satisfied if and only if $W$ has locally finite 
perimeter (that is, $\chi_W$ has locally bounded variation;
see Theorem~7.2 in~\cite{Lan3}). We will use this result only for 
$W$ equal to $\R^n$ or an $n$-simplex in $\R^n$.
 
\begin{proof}
If $p \in X$ and $r > a_0$, and if $B := \B{p}{r}$ and $x,y \in g^{-1}(B)$, 
then $d(x,y) \le L(d(g(x),g(y)) + a_0) \le 3Lr$, thus 
\[
\|S\|(B) = \|g_\#E\|(B) \le L^n\,\|E\|(g^{-1}(B)) \le C_0r^n
\]
for some constant $C_0$ depending only on $n$ and $L$.
Hence $S$ has $(C_0,a_0)$-controlled density.

Let $N \sub W$ be a $2La_0$-separated $2La_0$-net in $W$. If $x,y \in N$ are 
distinct, then $d(g(x),g(y)) \ge L^{-1}d(x,y) - a_0 \ge (2L)^{-1}d(x,y)$,
thus $g|_N$ is injective, and $(g|_N)^{-1} \colon g(N) \to N$ 
admits an $\bar L$-Lipschitz extension $\bar g \colon X \to \R^n$, 
where $\bar L := 2\sqrt{n}L$. 
Put $h := \bar g \circ g$. For every $x \in W$ there is a $y \in N$ such that
$d(x,y) \le 2La_0$. Then $h(y) = y$, thus
\[
d(h(x),x) \le d(h(x),h(y)) + d(y,x) \le (\bar L L + 1)\,d(x,y) \le b
\]
for $b := 2(\bar L L + 1)La_0$.

Suppose now that $x \in W$ and $r > 2Lb$ are such that
$\B{\bar x}{r} \cap g(\d W) = \es$, where $\bar x := g(x)$. For almost
every such $r$, both $S' := S \on \B{\bar x}{r}$ and
$E' := E \on g^{-1}(\B{\bar x}{r})$ are integral currents,
and $g_\#(E') = S'$. Since $g^{-1}(\B{\bar x}{r}) \cap \spt(\d E) = \es$,
the support of $\d E'$ is in $g^{-1}(\Sph{\bar x}{r})$ and thus at 
distance at least $r/L$ from $x$. Note that $\d(\bar g_\#S') = h_\#(\d E')$.
Using the geodesic homotopy from $\id_W$ to $h$, we get a 
current $R \in \bI_{n,\cs}(\R^n)$ with boundary $\d R = \d(\bar g_\#S') - \d E'$
and support within distance $b$ from $\spt(\d E')$;
in fact $R = \bar g_\#S' - E'$, because $\bZ_{n,\cs}(\R^n) = \{0\}$.
Since $r/L - b > r/(2L)$, the support of $R$ lies outside $\B{x}{r/(2L)}$.
It follows that
\[
\M(\bar g_\#S') = \M(E' + R) \ge \|E\|(\B{x}{r/(2L)}) \ge \eps r^n
\]
for some constant $\eps > 0$ depending only on $n$ and $L$. 

Now if $T \in \bI_{n,\cs}(X)$ is such that $\d T = \d S'$, then 
$\bar g_\#T = \bar g_\#S'$, and
\[
\M(S') \le C_0r^n \le C_0\eps^{-1} \M(\bar g_\#T) \le Q\,\M(T)
\]
for $Q := C_0\eps^{-1}\bar L^n$. Since $\spt(S) \sub g(W)$, this shows that
$S$ is $(Q,2Lb)$-quasi-minimizing mod $g(\d W)$.

To prove the last assertion, choose any $a > L(2Lb + a_0)$ and let $x \in W$ be
a point with $d(x,\d W) \ge a$. Then $d(g(x),g(\d W)) \ge L^{-1}a - a_0 > 2Lb$.
For a suitable $r \in (2Lb,a]$, the above argument then shows that 
$\M(\bar g_\#S') > 0$, thus $S' = S \on \B{g(x)}{r} \ne 0$,
and this implies that $d(g(x),\spt(S)) \le r \le a$. 
\end{proof}

The following variant of the above result applies to situations where
quasiflats can possibly not be approximated by Lipschitz ones
(see the proof of Theorem~\ref{thm:slim}.)
Here we call a compact set $W \sub \R^n$ {\em triangulated\/} if $W$ has the
structure of a finite simplicial complex all of whose maximal cells are
Euclidean $n$-simplices (thus $W$ is polyhedral). We denote by $W^0$ and
$(\d W)^0$ the set of vertices and boundary vertices of the triangulation,
respectively.

\begin{proposition}[triangulated quasiflats] \label{prop:qflats}
Let $n \ge 2$, and let $X$ be a proper metric space satisfying 
condition~{\rm (CI$_{n-1}$)}. Then for all $C_0,D_0 > 0$ and 
$L,a_0$ there exist $Q,C,a$ such that the following holds.
Suppose that $W \sub \R^n$ is a compact triangulated set with simplices of 
diameter at most $D_0$, and such that every $r$-ball in\/ $\R^n$ with  
$r \ge D_0$ meets at most $C_0r^n$ $n$-simplices. 
Let $\cP_*(W)$ denote the corresponding chain complex of simplicial integral 
currents. 
If $f \colon W \to X$ is an $(L,a_0)$-quasi-isometric embedding, then there 
exists a chain map $\iota \colon \cP_*(W) \to \bI_{*,\cs}(X)$ such that 
\ben
\item
$\iota$ maps every vertex $\bb{x_0} \in \cP_0(W)$ to $\bb{f(x_0)}$ and, for 
$1 \le k \le n$, every basic oriented simplex $\bb{x_0,\dots,x_k} \in \cP_k(W)$ 
to a current with support in $N_a(f(\{x_0,\dots,x_k\}))$; 
\item 
$S := \iota\bb{W} \in \bI_{n,\cs}(X)$ is $(Q,a)$-quasi-minimizing mod
$N_a(f((\d W)^0))$ and has $(C,a)$-controlled density;
\item 
$d(f(x),\spt(S)) \le a$ for all $x \in W$ with $d(x,\d W) \ge a$. 
\een
\end{proposition}

Here $N_a(\cdot)$ stands for the closed $a$-neighborhood of a set.
Note that by~(1), $\spt(S) \sub N_a(f(W^0))$ and 
$\spt(\d S) \sub N_{a}(f((\d W)^0))$.
An analogous result for closed sets with locally finite triangulations
and locally integral currents also holds.

\begin{proof}
Put $\cS_* := \bigcup_{k=0}^n \cS_k$, where $\cS_k$ denotes the set of all 
basic simplices $s = \bb{x_0,\dots,x_k} \in \cP_k(W)$ (compare p.~365
in~\cite{Fed} for the notation). Using
Theorem~\ref{thm:plateau} (minimizing filling), we can inductively 
build a map $\iota \colon \cS_* \to \bI_{*,\cs}(X)$ as follows.
For every $\bb{x_0} \in \cS_0$, $\iota\bb{x_0} := f_\#\bb{x_0} = \bb{f(x_0)}$. 
Now let $k \ge 1$, and suppose that $\iota$ is defined on $\cS_{k-1}$.
For every $k$-cell of $W$, we pick
an orientation $s = \bb{x_0,\dots,x_k} \in \cS_k$, then we let 
$\iota(s) \in \bI_{k,\cs}(X)$ be a minimizing filling of
\[
\sum_{i=0}^k (-1)^i\,\iota\bb{x_0,\dots,x_{i-1},x_{i+1},\dots,x_k}
\in \bZ_{k-1,\cs}(X),
\]
and we put $\iota(-s) := -\iota(s)$.
The resulting map on $\cS_*$ readily extends to a chain map 
$\iota \colon \cP_*(W) \to \bI_{*,\cs}(X)$.
It follows inductively from condition~{\rm (CI$_{k-1}$)} and the distance
bound in Theorem~\ref{thm:plateau} for $k = 1,\dots,n$ that for all
$s = \bb{x_0,\dots,x_k} \in \cS_k$,
\[
\M(\iota(s)) \le M
\]
and $\spt(\iota(s)) \sub N_{a'}(f(\{x_0,\dots,x_k\}))$ for some 
constants $M,a'$ depending on $D_0,L,a_0$ and the constants $c_0,\dots,c_{n-1}$ 
implicit in condition~{\rm (CI$_{n-1}$)}.
In the following we assume that $a' \ge a_0$.

Let now $\cS_n^+ \sub \cS_n$ be the set of all positively oriented 
$n$-simplices, whose sum is $\bb{W}$. Put $S := \iota\bb{W}$. 
To show that $S$ has controlled density, 
let $p \in X$ and $r > a'$, and consider the set of all $s \in \cS_n^+$ for 
which $\spt(\iota(s)) \cap \B{p}{r} \ne \es$. Every such $s$ has a 
vertex $x^s$ with $f(x^s) \in \B{p}{r+a'}$, thus the set of all
$x^s$ has diameter at most $L(2(r+a')+a_0) \le 5Lr$. It follows that there 
are at most $C_0(5Lr)^n$ such simplices and that 
$\G_{p,r}(S) \le C := C_0(5L)^nM$ for $p \in X$ and $r > a'$.  

Similarly as in the proof of Proposition~\ref{prop:lip-qflats}, 
there exists an $\bar L$-Lipschitz map $\bar f \colon X \to \R^n$ such that 
$h := \bar f \circ f$ satisfies $d(h(x),x) \le b'$ for all $x \in W$,
where $\bar L := 2\sqrt{n}L$ and $b'$ depends on $n,L,a_0$. Then
\[
\bar\iota := \bar f_\# \circ \iota \colon \cP_*(W) \to \bI_{*,\cs}(\R^n)
\]
is a chain map that sends every $\bb{x_0} \in \cS_0$ to $\bb{h(x_0)}$ and
every $\bb{x_0,\dots,x_k} \in \cS_k$ to a current with support in
$N_{\bar L a' + b'}(\{x_0,\dots,x_k\})$. 
Note that $\bar\iota(\d\bb{W}) = \d(\bar f_\#S)$.
Similarly as above, using geodesic cone fillings of cycles in $\R^n$, 
we can inductively construct a chain homotopy between 
$\id_\#,\bar\iota \colon \cP_*(\d W) \to \bI_{*,\cs}(\R^n)$. 
This yields an $R \in \bI_{n,\cs}(\R^n)$ such that 
$\d R = \d(\bar f_\#S) - \d\bb{W}$ and $\spt(R) \sub N_b((\d W)^0)$ 
for some constant $b$ depending on $n,D_0,L,a'$. 
Since $\bZ_{n,\cs}(\R^n) = \{0\}$, in fact $R = \bar f_\#S - \bb{W}$.

Suppose now that $x \in W$ and $r > 2Lb$ are such that
\[
\B{\bar x}{r} \cap N_{a'}(f((\d W)^0)) = \es
\] 
and $S' := S \on \B{\bar x}{r} \in \bI_{n,\cs}(X)$, where $\bar x := f(x)$. 
For all $y \in (\d W)^0$,
$d(x,y) \ge (d(\bar x,f(y)) - a_0)/L > r/L > r/(2L) + b$; hence 
\[
\spt(R) \cap \B{x}{r/(2L)} = \es.
\]
For every $\bar y \in \spt(S - S') \sub \spt(S)$ there exists a $y \in W^0$ 
such that $d(f(y),\bar y) \le a'$ and 
$r \le d(\bar x,\bar y) \le d(\bar x,f(y)) + a' \le L\,d(x,y) + 2a'$, thus
\begin{align*}
d(x,\bar f(\bar y)) 
&\ge d(x,y) - d(y,h(y)) - d(\bar f(f(y)),\bar f(\bar y)) \\
&\ge L^{-1}(r - 2a') - b' - \bar L a'.
\end{align*}
By increasing $b$ if necessary, so that $r$ is large enough, we arrange 
that this last expression is bigger than $r/(2L)$. This then shows that 
\[
\spt(\bar f_\#(S - S')) \cap \B{x}{r/(2L)} = \es.
\]
Since $\bar f_\#S' = \bb{W} + R - \bar f_\#(S-S')$, it follows that 
\[
\M(\bar f_\#S') \ge \|\bb{W}\|(\B{x}{r/(2L)}) \ge \eps r^n
\]
for some $\eps > 0$ depending on $n$ and $L$. The proof may now be 
completed as for Proposition~\ref{prop:lip-qflats}. For assertion~(3),
choose $a > L(2Lb + a' + a_0)$.
\end{proof}


\section{Asymptotic rank} \label{sect:asympt-rank}

In this section we will first discuss the notion of asymptotic rank and
the sub-Euclidean isoperimetric inequality from~\cite{Wen4}.
Then we will derive a localized version of this result as well as various
characterizations of quasi-minimizing local $n$-cycles in spaces of
asymptotic rank at most~$n$.

In~\cite{Gro4}, Section~6.B$_2$, Gromov defined a number of different
large-scale notions of rank for spaces of nonpositive curvature. 
Many of the ensuing questions were then answered in~\cite{Kle} 
(see the discussion in Section~9 therein).
Theorem~D in that paper shows in particular the following.

\begin{theorem}[rank conditions] \label{thm:rank-conditions}
Let $X$ be a proper Busemann space with cocompact isometry group.
Then for every $n \ge 1$ the following are equivalent:  
\ben
\item
$X$ contains an isometric copy of some $n$-dimensional normed space;
\item
there exists a quasi-isometric embedding of $\R^n$ into $X$;
\item
there exist a sequence of subsets $Y_i \sub X$ and a sequence
$0 < r_i \to \infty$ such that the rescaled sets $(Y_i,r_i^{-1}d)$ converge 
in the Gromov--Hausdorff topology to the closed unit ball in some
$n$-dimensional normed space.
\een
\end{theorem}

Stronger conclusions hold if $X$ is a proper and cocompact $\CAT(0)$ space.
Then any normed space isometrically embedded in $X$ is necessarily Euclidean;
furthermore, the {\em Euclidean rank\/} of $X$, the
maximal $n$ for which $X$ contains an $n$-flat, is equal to the
{\em geometric dimension} or the {\em compact topological dimension}
(that is, the supremum of the topological dimensions of compact subsets)
of the Tits cone $\CT X$ or of any asymptotic cone $X_\om$
and also agrees with the maximal $n$ for which $H_{n-1}(\bT X) \ne \{0\}$, 
where $\bT X$ denotes the Tits boundary.
See Theorems~A and~C in~\cite{Kle}.

Property~(3) above suggests the following notion of asymptotic rank that was
investigated in~\cite{Wen4}.

\begin{definition}[asymptotic subset, asymptotic rank] \label{def:as-rank}
Let $X = (X,d)$ be a metric space. Any compact metric space $(Y,d_Y)$ 
that can be obtained as the Gromov--Hausdorff limit of a sequence
$(Y_i,r_i^{-1}d)$ as in~(3) above will be called an {\em asymptotic subset\/}
of $X$. The {\em asymptotic rank\/} $\asrk(X)$ of $X$ is the supremum of all
integers $n \ge 0$ such that there exists an asymptotic subset of $X$
isometric to the unit ball in some $n$-dimensional normed space.
\end{definition}

\begin{remark} \label{rem:as-rank}
Alternatively, $\asrk(X)$ may be defined as the supremum of all $n$ such 
that $X$ admits an asymptotic subset bi-Lipschitz homeomorphic to a compact 
subset of $\R^n$ with positive Lebesgue measure. The equivalence of the two 
definitions is shown by means of a metric differentiation argument
(see~\cite{Wen4}).
\end{remark}

We remark that every asymptotic subset of $X$ embeds isometrically
into some asymptotic cone of $X$. Conversely, every compact subset
$Y \sub X_\om$ of an asymptotic cone is an asymptotic subset of $X$,
and the respective sets $Y_i \sub X$ may be chosen to be finite. 
If $f \colon X \to \bar X$ is a quasi-isometric embedding into another 
metric space $\bar X$, then
\[
\asrk(X) \le \asrk(\bar X);
\]
thus $\asrk$ is a quasi-isometry invariant for metric spaces
(see Corollary~3.3 in \cite{Wen4}).   

In a nonpositively curved symmetric space $X$, every
$n$-cycle $Z$ with $n$ greater than or equal to the rank of $X$
admits a filling $V$ with mass
\[
\M(V) \le \const \cdot \M(Z)
\]
(see p.~105 in \cite{Gro4}, and~\cite{Leu}), whereas in smaller dimensions,
the optimal isoperimetric inequalities in $X$ are of Euclidean type,
as in Theorem~\ref{thm:isop-ineq}. 
It is not known whether the linear inequalities for $n$-cycles remain valid,
for example, in cocompact Hadamard manifolds containing no $(n+1)$-flat. 
However, the following key result due to Stefan Wenger provides
a substitute for spaces of asymptotic rank at most $n$.

\begin{theorem}[sub-Euclidean isoperimetric inequality] \label{thm:subeucl}
Let $X$ be a proper metric space satisfying condition~{\rm (CI$_n$)}
for some $n \ge 1$, and suppose that\/ $\asrk(X) \le n$.
Then for all $C,\eps > 0$ there is a constant $a_\eps \ge 0$
(depending on $X,n,C,\eps$) such that if $r > a_\eps$, then every cycle
$Z \in \bZ_{n,\cs}(X)$ with $\M(Z) \le C r^n$ and $\spt(Z) \sub \B{p}{r}$
for some $p \in X$ possesses a filling $V \in \bI_{n+1,\cs}(X)$ with mass 
\[
\M(V) < \eps r^{n+1}.
\]
\end{theorem}

This is shown in a more general form, for complete metric spaces, and 
without restrictions on $\spt(Z)$, in Theorem~1.2 in~\cite{Wen4}. 
The stated version suffices for our purposes, and the proof could be 
slightly simplified under these assumptions.

The following result may be viewed as a localized version of 
Theorem~\ref{thm:subeucl} and will be used repeatedly throughout the paper.
The main content is that if a cycle $Z \in \bZ_{n,\cs}(X)$ satisfies 
$\|Z\|(\B{p}{r}) \le Cr^n$ for some $p \in X$ and for all $r > a \ge 0$,
then for every $\eps > 0$ and every sufficiently large $r > 0$ 
there exists a ``partial filling'' $V \in \bI_{n+1,\cs}(X)$ such that  
$\spt(Z - \d V) \cap \B{p}{r} = \es$ and $\M(V) < \eps r^{n+1}$; that is,
$\F_{p,r}(Z) < \eps$.
We formulate this more generally for local cycles of the form $Z = S - S'$ 
with $\Fi(Z) < \infty$, where only $S \in \bI_{n,\loc}(X)$ is required
to satisfy a density bound with respect to $p$ and $S' \in \bI_{n,\loc}(X)$ 
(possibly zero) is area-minimizing.

\begin{proposition}[partial filling] \label{prop:part-filling}
Let $X$ be a proper metric space satisfying condition~{\rm (CI$_n$)}
for some $n \ge 1$, and suppose that\/ $\asrk(X) \le n$.
Then for all $C,\eps > 0$ and $a \ge 0$ there is a constant $a'_\eps \ge 0$
such that the following holds. 
Suppose that $S \in \bI_{n,\loc}(X)$ satisfies $\G_{p,r}(S) \le C$ for some
point $p \in X$ and for all $r > a$, and $S' \in \bI_{n,\loc}(X)$ is
minimizing with $\d S' = \d S$ and $\Fi(S - S') < \infty$. Then
\[
\G_{p,r}(S') < C + \eps \quad \text{and} \quad \F_{p,r}(S - S') < \eps
\]
for all $r > a'_\eps$, in particular $\Gi(S') \le C$ and $\Fi(S - S') = 0$.
\end{proposition}

This shows in particular the following dichotomoy:
if $Z \in \bZ_{n,\loc}(X)$ and $\Gi(Z) < \infty$, then $\Fi(Z)$
is either $0$ or $\infty$. 
 
\begin{proof}
We write $B_r := \B{p}{r}$ for $r > 0$.
Choose a constant $D > \Fi(S - S')$. Then, for every sufficiently large 
$r_0 > 0$, there is a $V_0 \in \bI_{n+1,\cs}(X)$ such that
\[
\M(V_0) < D r_0^{\,n+1}
\]
and $\spt(S - S' - \d V_0) \cap B_{r_0} = \es$. We fix such $r_0$ and $V_0$
for the moment, and we put $r_i := \eta^ir_0$ for some fixed $\eta \in (0,1)$
and every integer $i \ge 1$.

There exists an $s \in (r_1,r_0)$ such that both $S' \on B_s$ and the slice
\[
T_s := \d(V_0 \on B_s) - (\d V_0) \on B_s = \d(V_0 \on B_s) - (S - S') \on B_s
\]
belong to $\bI_{n,\cs}(X)$, and
$\M(T_s) \le \M(V_0)/(r_0 - r_1) \le (1-\eta)^{-1}D r_0^{\,n}$.
Note that $\d(S' \on B_s) = \d(S \on B_s + T_s)$.
Using the minimality of $S'$ and assuming that $r_1 > a$,
so that $\G_{p,s}(S) \le C$, we infer that 
\[
\M(S' \on B_s) \le \M(S \on B_s) + \M(T_s)
\le Cs^n + (1-\eta)^{-1}D r_0^{\,n} \le \bar C r_1^{\,n}
\]
for $\bar C := \eta^{-n}(C + (1-\eta)^{-1}D)$. Thus $\G_{p,r_1}(S') \le \bar C$, 
and the cycle $Z_s := (S - S') \on B_s + T_s$ satisfies 
$\M(Z_s) \le 2\bar C r_1^{\,n}$ and $\spt(Z_s) \sub B_{2r_1}$.
Let $\del > 0$. By Theorem~\ref{thm:subeucl} there exists a constant
$\bar a_\del \ge a$, depending only on $n,X,\bar C,a,\del$, such that if
$r_1 > \bar a_\del$, then $Z_s$ possesses a filling $V_1 \in \bI_{n+1,\cs}(X)$
with mass
\[
\M(V_1) < \del r_1^{\,n+1}.
\]
Note that the support of $S - S' - \d V_1 = S - S' - Z_s$ lies outside
$B_{r_1}$, thus $\F_{p,r_1}(S - S') < \del$.
Note further that for $\del \le D$, $V_1$ replicates the properties of $V_0$
at the next smaller scale $r_1$.

Now, given any $\del \in (0,D]$ and $r > \bar a_\del$, we can choose
$r_0$ initially such that $r = r_k = \eta^kr_0$ for some $k \ge 1$.  
In the case that $k \ge 2$, we repeat the slicing and filling procedure 
described in the preceding paragraph successively for $i = 1,2,\dots,k-1$,
with $(r_{i+1},r_i)$ and $V_i$ in place of $(r_1,r_0)$ and $V_0$.
This produces a sequence of partial fillings
$V_1,\dots,V_k \in \bI_{n+1,\cs}(X)$ of $S-S'$, with
$\spt(S - S' - \d V_i) \cap B_{r_i} = \es$, 
such that $\G_{p,r_i}(S') \le \bar C$ and $\M(V_i) < \del r_i^{\,n+1}$.
For $i = k$, this shows that
\[
\G_{p,r}(S') \le \bar C = \eta^{-n}(C + (1-\eta)^{-1}D) \quad \text{and} \quad
\F_{p,r}(S - S') < \del
\]
whenever $0 < \del \le D$ and $r > \bar a_\del$.

In particular, $\Fi(S - S') = 0$, and we may thus repeat the above
argument for arbitrarily small $D > 0$. Let $\eps > 0$.
Choosing $\eta \in (0,1)$ and $D$ such that $\bar C < C + \eps$,
and putting $\del := \min\{\eps,D\}$,
we conclude that $\G_{p,r}(S') < C + \eps$ and $\F_{p,r}(S - S') < \eps$
whenever $r > a'_\eps := \bar a_\del$. Note that $a'_\eps$ depends only on
$n,X,C,a,\eps$.
\end{proof}
 
The following result is included mainly for illustration. It shows that
for local $n$-cycles with controlled density in spaces satisfying the
assumptions of Theorem~\ref{thm:subeucl}, quasi-minimality is equivalent
to several other conditions, among them the lower bound on the filling density
obtained in Lemma~\ref{lem:fill-density}.

\begin{proposition}[characterizing quasi-minimizers]
Let $X$ be a proper metric space satisfying condition~{\rm (CI$_n$)}
for some $n \ge 1$, and suppose that\/ $\asrk(X) \le n$.
For an $S \in \bZ_{n,\loc}(X)$ with $(C,a)$-controlled density,
the following are equivalent:
\ben
\item
There exist $Q \ge 1$ and $a_1 \ge 0$ such that $S$ is 
$(Q,a_1)$-quasi-minimizing.
\item
There exist $c_2 > 0$ and $a_2 \ge 0$ such that if $x \in \spt(S)$,
then $\M(T) \ge c_2r^n$
for almost every $r > a_2$ and every
$T \in \bI_{n,\cs}(X)$ with $\d T = \d(S \on \B{x}{r})$.
\item
There exist $c_3 > 0$ and $a_3 \ge 0$ such that if $x \in \spt(S)$,
then $\M(T) \ge c_3r^n$ for almost every $r > a_3$ and every
$T \in \bI_{n,\cs}(X)$ with $\d T = \d(S \on \B{x}{r})$ and 
$\spt(T) \sub \Sph{x}{r}$.
\item
There exist $c_4 > 0$ and $a_4 \ge 0$ such that $\F_{x,r}(S) \ge c_4$ for all 
$x \in \spt(S)$ and $r > a_4$.
\een
\end{proposition}

Notice that~(3) is a divergence condition for $S$; compare, for example, 
the definition of the divergence of a geodesic line in Section~3 
of~\cite{KapL2}.

\begin{proof}
The implications (1)~$\Rightarrow$~(2) and (2)~$\Rightarrow$~(1) follow 
easily from Lemma~\ref{lem:density} (density) and the fact that $S$ has 
controlled density, respectively. 
The implication (2)~$\Rightarrow$~(3) holds trivially, and 
(3)~$\Rightarrow$~(4) is shown by a simple integration as in the proof of 
Lemma~\ref{lem:fill-density} (filling density).

To prove that (4)~$\Rightarrow$~(2), let $x \in \spt(S)$, $r > a$, and
$T \in \bI_{n,\cs}(X)$ be such that $S' := S \on \B{x}{r} \in \bI_{n,\cs}(X)$
and $\d T = \d S'$. By Theorem~\ref{thm:plateau} (minimizing filling),
there is no loss of generality in assuming that $T$ is minimizing.
Then the cycle $Z := S' - T \in \bZ_{n,\cs}(X)$ has mass at most $2Cr^n$, and
$\spt(T)$ is within distance $(\M(T)/\del)^{1/n}$ from $\spt(\d S')$. 
Hence, if $\M(T) < \del(r/2)^n$, say, then 
$\spt(Z) \sub \B{x}{3r/2}$ and $\spt(S - Z) \cap \B{x}{r/2} = \es$, 
and it follows from Theorem~\ref{thm:subeucl} that $\F_{x,r/2}(S) < c_4$,
provided $r$ is sufficiently large. In view of~(4), we conclude that
$\M(T) \ge \del(r/2)^n$ for large enough $r$.
\end{proof}


\section{Morse Lemmas} \label{sect:morse-lemmas}

In this section we will prove some higher rank analogs of the Morse Lemma, 
replacing quasi-geodesics with $n$-dimensional quasi-minimizers 
with controlled density.
Here we will also establish Theorem~\ref{intro-slim}.

A first result follows very quickly from 
Lemma~\ref{lem:fill-density} (filling density) and
Proposition~\ref{prop:part-filling} (partial filling).
 
\begin{theorem}[Morse Lemma I] \label{thm:morse-1}
Let $X$ be a proper metric space satisfying condition~{\rm (CI$_n$)}
for some $n \ge 1$, and suppose that\/ $\asrk(X) \le n$.
Then for all $Q \ge 1$, $C > 0$, and $a \ge 0$ there is a constant $b \ge 0$
such that the following holds. Suppose that $Z \in \bZ_{n,\loc}(X)$ 
has $(C,a)$-controlled density and satisfies $\Fi(Z) < \infty$.
If $Y \sub X$ is a closed set such that $Z$ is $(Q,a)$-quasi-minimizing 
mod~$Y$, then $\spt(Z)$ lies within distance at most $b$ from~$Y$.
\end{theorem}

Note that if $S,S' \in \bI_{n,\loc}(X)$ are two $(Q,a)$-quasi-minimizers with
$(C/2,a)$-controlled density and $\d S = \d S'$, then 
$Z := S - S' \in \bZ_{n,\loc}(X)$ is $(Q,a)$-quasi-minimizing mod~$\spt(S)$ as 
well as mod~$\spt(S')$ and has $(C,a)$-controlled density.
Theorem~\ref{thm:morse-1} then shows that the Hausdorff distance 
$\Hd(\spt(S),\spt(S'))$ is at most $b$, provided $\Fi(Z) < \infty$ 
(which holds trivially if $S,S' \in \bI_{n,\cs}(X)$).

\begin{proof}
Since $Z$ is $(Q,a)$-quasi-minimizing mod $Y$, 
Lemma~\ref{lem:fill-density} shows that
\[
\F_{x,r}(Z) \ge c = c(n,X,Q) > 0
\]
whenever $x \in \spt(Z)$, $r > 4a$, and $\B{x}{r} \cap Y = \es$.
On the other hand, since $Z$ has $(C,a)$-controlled density and satisfies
$\Fi(Z) < \infty$, we may apply Proposition~\ref{prop:part-filling} 
with $p = x$ and $S = Z$, $S' = 0$. Taking $\eps = c$, we infer that there 
is a constant $b \ge 4a$, depending only on $n,X,Q,C,a$, such that 
\[
\F_{x,r}(Z) < c
\]
for $r > b$. This shows that $r \le b$ (in particular $Y \ne \es$).
\end{proof}

As a first application, we deduce Theorem~\ref{intro-slim},
which we restate for convenience.

\begin{theorem}[slim simplices] \label{thm:slim}
Let $X$ be a proper metric space satisfying condition~{\rm (CI$_n$)}
for some $n \ge 1$, and suppose that\/ $\asrk(X) \le n$.
Let $\Del$ be a Euclidean $(n+1)$-simplex, and let $f \colon \d\Del \to X$ 
be a map such that for every facet $W$ of $\Del$, the restriction $f|_W$ is an 
$(L,a_0)$-quasi-isometric embedding. Then, for every facet $W$, the image
$f(W)$ is contained in the closed $D$-neighborhood of
$f\bigl( \ol{\d\Del \sm W} \bigr)$ for some constant $D \ge 0$ 
depending only on $X,n,L,a_0$.
\end{theorem}

\begin{proof}
Let $W_0,\dots,W_{n+1}$ be an enumeration of the facets of $\Del$,
and write the cycle $\d\bb{\Del} \in \bZ_{n,\cs}(\R^{n+1})$ as
$\sum_{i=0}^{n+1}E_i$ for $E_i := (\d\bb{\Del}) \on W_i \in \bI_{n,\cs}(\R^{n+1})$.
Choose a triangulation of $\d\Del$ with simplices of diameter
at most $D_0$ such that every $r$-ball in $\R^{n+1}$ with $r \ge D_0$
meets at most $C_0r^n$ $n$-simplices in each $W_i$,
for some constants $C_0,D_0 > 0$ depending only on $n$.
Consider the corresponding chain complex
$\cP_*(\d\Del)$ of simplicial integral currents and proceed as in the
proof of Proposition~\ref{prop:qflats} (triangulated quasiflats) 
to get a chain map $\iota \colon \cP_*(\d\Del) \to \bI_{*,\cs}(X)$
such that the following properties hold for every 
$S_i := \iota(E_i) \in \bI_{n,\cs}(X)$ and for some constants $Q,C,a$ 
depending only on $X,n,L,a_0$:
\ben
\item
$\spt(S_i) \sub N_a(f(W_i))$ and $\spt(\d S_i) \sub N_a(f(\d W_i))$;
\item
$S_i$ is $(Q,a)$-quasi-minimizing mod $N_a(f(\d W_i))$ and 
has $(C,a)$-controlled density;
\item
$d(f(x),\spt(S_i)) \le a$ for all $x \in W_i$ with $d(x,\d W_i) \ge a$.
\een
(Here $N_a$ stands again for the closed $a$-neighborhood, and
$\d W_i$ denotes the relative boundary of $W_i$.)
Let $M_i$ denote the union of all $W_j$ with $j \ne i$.
The cycle $Z := \iota(\d\bb{\Del}) = \sum_{i=0}^{n+1} S_i$ is
$(Q,a)$-quasi-minimizing mod~$N_a(f(M_i))$ for every $i$ and 
has $((n+2)C,a)$-controlled density. It then follows from
Theorem~\ref{thm:morse-1} that the set 
$\spt(S_i) \sm N_a(f(M_i)) = \spt(Z) \sm N_a(f(M_i))$ lies within distance
at most $b$ from $N_a(f(M_i))$ for some constant $b$ depending only on
$X,n,L,a_0$. Hence, for $x \in W_i$, it follows from~(3) that
$d(f(x),f(M_i))$ is less than or equal to $2a + b$ 
if $d(x,\d W_i) \ge a$ and less than $La + a_0$ otherwise.
\end{proof}

\begin{remark}
If, for $f \colon \d\Del \to X$ as above, there
exists a map $g \colon \d\Del \to X$ such that $g|_W$ is
$L'$-Lipschitz for every facet $W$ and $d(f(x),g(x)) \le b'$
for all $x \in \d\Del$, for some constants $L',b'$ depending on $n,L,a_0$,
then one may use Proposition~\ref{prop:lip-qflats} (Lipschitz quasiflats)
instead of Proposition~\ref{prop:qflats} in the above argument.
Such a map $g$ exists if $X$ is Lipschitz $(n-1)$-connected
(compare Corollary~1.7 in~\cite{LanS}), in particular if $X$ is $\CAT(0)$
or a space with a convex bicombing.
\end{remark}

We now prove an analog of the Morse Lemma for quasi-geodesic segments.

\begin{theorem}[Morse Lemma II] \label{thm:morse-2}
Let $X$ be a proper metric space satisfying condition~{\rm (CI$_n$)} for some
$n \ge 1$, and suppose that\/ $\asrk(X) \le n$.
Then for all $Q \ge 1$, $C > 0$, and $a \ge 0$ there is a constant $b \ge 0$ 
such that the following holds.
If $S \in \bI_{n,\cs}(X)$ is a $(Q,a)$-quasi-minimizer with $(C,a)$-controlled 
density, then there exists a minimizing $\til S \in \bI_{n,\cs}(X)$ such that 
$\d S = \d \til S$, and every such $\til S$ satisfies
\[
\Hd(\spt(S),\spt(\til S)) \le b.
\]
\end{theorem}

\begin{proof}
A minimizing $\til S \in \bI_{n,\cs}(X)$ with $\d \til S = \d S$ exists by 
Theorem~\ref{thm:plateau} (minimizing filling).
Since $S$ has $(C,a)$-controlled density, it follows from 
Proposition~\ref{prop:part-filling} (partial filling) that every
such $\til S$ has $(\til C,\til a)$-controlled density
for some constants $\til C \ge C$ and $\til a \ge a$ depending only on 
$n,X,C,a$. Then the cycle $S - \til S$ has $(2\til C,\til a)$-controlled
density and is $(Q,a)$-quasi-minimizing mod~$\spt(\til S)$ as well as
mod~$\spt(S)$; the result thus follows from Theorem~\ref{thm:morse-1}.
\end{proof}

Our next goal is to extend this last result to local chains.
We state an auxiliary lemma.

\begin{lemma}[$\F$-convergence] \label{lem:f-conv}
Let $X$ be a proper metric space satisfying condition~{\rm (CI$_n$)} for some
$n \ge 1$. Then a sequence $(Z_j)$ in $\bZ_{n,\loc}(X)$ converges in the local 
flat topology to $0$ if and only if\/ $\lim_{j \to \infty} \F_{p,r}(Z_j) = 0$
for all $p \in X$ and $r > 0$.
\end{lemma}

\begin{proof}
Suppose that $Z_j \to 0$ in the local flat topology, and let $p \in X$ and
$r > 0$. There is a sequence $(V_j)$ in $\bI_{n+1,\loc}(X)$ such that
\[
(\|Z_j - \d V_j\| + \|V_j\|)(\B{p}{2r}) \to 0.
\]
Note that $Z_j - \d V_j \in \bZ_{n,\loc}(X)$.
Pick $s \in (r,2r)$ such that, for $K := \B{p}{s}$, the slice
$\d((Z_j - \d V_j) \on K)$ is in $\bZ_{n-1,\cs}(X)$ for all $j$, and furthermore
$V_j \on K \in \bI_{n+1,\cs}(X)$ for all $j$.
By Theorem~\ref{thm:plateau} (minimizing filling), there exists a minimizing
current $T_j \in \bI_{n,\cs}(X)$ with $\d T_j = \d((Z_j - \d V_j) \on K)$,
and since $\spt(\d T_j) \sub \Sph{p}{s}$ and
$\M(T_j) \le \M((Z_j - \d V_j) \on K) \to 0$, it follows that
$\spt(T_j) \sub \B{p}{2r} \sm \B{p}{r}$ for $j$ sufficiently large. 
The cycles $(Z_j - \d V_j) \on K - T_j$ converge to zero in mass,
by condition~{\rm (CI$_n$)} they thus possess fillings
$W_j \in \bI_{n+1,\cs}(X)$ such that also $\M(W_j) \to 0$. 
Now define $V'_j := V_j \on K + W_j \in \bI_{n+1,\cs}(X)$. 
Then $\M(V'_j) \le \M(V_j \on K) + \M(W_j) \to 0$, and the support of 
\begin{align*}
Z_j - \d V'_j 
&= Z_j - \d(V_j \on K) - (Z_j - \d V_j) \on K + T_j \\
&= Z_j \on (X \sm K) - \bigl( \d(V_j \on K) - (\d V_j) \on K \bigr) + T_j 
\end{align*}
is disjoint from $\B{p}{r}$ for all sufficiently large $j$. 
This shows that $\F_{p,r}(Z_j) \le \M(V'_j)/r^{n+1} \to 0$.

The reverse implication is clear.
\end{proof}

We now establish a basic existence theorem for minimizing 
local $n$-chains in spaces of asymptotic rank at most $n$.

\begin{theorem}[constructing minimizers] \label{thm:constr-minimizers}
Let $X$ be a proper metric space satisfying condition~{\rm (CI$_n$)} for some
$n \ge 1$, and suppose that\/ $\asrk(X) \le n$.
Then for every $S \in \bI_{n,\loc}(X)$ with $\Gi(S) < \infty$ there
exists a minimizing $\til S \in \bI_{n,\loc}(X)$ such that $\d \til S = \d S$ 
and $\Fi(S - \til S) = 0$, and every such $\til S$ 
satisfies $\Gi(\til S) \le \Gi(S)$.
\end{theorem}

Note that $\d S$ may well be zero; the assertion $\Fi(S - \til S) = 0$
then guarantees that $\til S$ is non-zero, provided $\Fi(S) \ne 0$.
Conversely, if $\Fi(S) = 0$, then $\d S = 0$, and it follows from
Lemma~\ref{lem:fill-density} (filling density) that there is no
minimizer $\til S \ne 0$ with $\d\til S = 0$ and $\Fi(\til S) = 0$.

\begin{proof}
Fix a base point $p \in X$, and choose a sequence $0 < r_i \uparrow \infty$
such that $S_i := S \on \B{p}{r_i} \in \bI_{n,\cs}(X)$ for all $i$.
Theorem~\ref{thm:plateau} (minimizing filling) provides
a corresponding sequence of minimizing currents $\til S_i \in \bI_{n,\cs}(X)$
with $\d \til S_i = \d S_i$. Since $\Gi(S) < \infty$, there exist $C > 0$ and
$a \ge 0$ such that $\G_{p,r}(S_i) \le \G_{p,r}(S) \le C$ for all $r > a$.
Proposition~\ref{prop:part-filling} (partial filling) shows that
for every $\eps > 0$ there is a constant $\til a_\eps \ge 0$
such that, for all $i$ and $r > \til a_\eps$,
\[
\|\til S_i\|(\B{p}{r}) < (C+\eps)r^n \quad \text{and} \quad
\F_{p,r}(S_i - \til S_i) < \eps.
\]
Note also that if $K \sub X$ is a compact set,
then $\|\d \til S_i\|(K) = \|\d S_i\|(K) = \|\d S\|(K)$
for all but finitely many indices $i$.
By Theorem~\ref{thm:cptness} (compactness), some subsequence 
$(\til S_{i_j})$ converges in the local flat topology to a
minimizing current $\til S \in \bI_{n,\loc}(X)$ with $\d \til S = \d S$.

To show that $\Fi(S - \til S) = 0$, put 
$Z_j := S - S_{i_j} - (\til S - \til S_{i_j}) \in \bZ_{n,\loc}(X)$ and note that 
$Z_j \to 0$ in the local flat topology.
If $\eps > 0$ and $r > \til a_\eps$, then
\[
\F_{p,r}(S - \til S) \le \F_{p,r}(Z_j) + \F_{p,r}(S_{i_j} - \til S_{i_j})
< \F_{p,r}(Z_j) + \eps.
\]
Hence, $\F_{p,r}(S - \til S) \le \eps$ by Lemma~\ref{lem:f-conv}.

Finally, a simple slicing argument shows that $\Gi(\til S) \le \Gi(S)$
for every minimizing $\til S \in \bI_{n,\loc}(X)$ with $\d \til S = \d S$ and
$\Fi(S - \til S) = 0$.
(This also follows from Proposition~\ref{prop:part-filling}.)
\end{proof}

The next result generalizes Theorem~\ref{thm:morse-2} to local currents.
Theorem~\ref{intro-morse} in the introduction corresponds to the case 
$\d S = 0$.

\begin{theorem}[Morse Lemma III] \label{thm:morse-3}
Let $X$ be a proper metric space satisfying condition~{\rm (CI$_n$)} for some
$n \ge 1$, and suppose that\/ $\asrk(X) \le n$.
Then for all $Q \ge 1$, $C > 0$, and $a \ge 0$ there is a constant $b \ge 0$
such that the following holds.
If $S \in \bI_{n,\loc}(X)$ is a $(Q,a)$-quasi-minimizer with $(C,a)$-controlled 
density, then there exists a minimizing $\til S \in \bI_{n,\loc}(X)$ such that
$\d S = \d \til S$ and $\Fi(S - \til S) = 0$, and every such $\til S$
satisfies $\Gi(\til S) \le \Gi(S)$ and 
\[
\Hd(\spt(S),\spt(\til S)) \le b.
\]
\end{theorem}

\begin{proof}
Since $\Gi(S) < \infty$, Theorem~\ref{thm:constr-minimizers} shows that there
exists a minimizing $\til S \in \bI_{n,\loc}(X)$ with $\d S = \d \til S$ and
$\Fi(S - \til S) = 0$, and every such $\til S$
satisfies $\Gi(\til S) \le \Gi(S)$.
The rest of the proof is the same as for Theorem~\ref{thm:morse-2}.
\end{proof}


\section{Asymptote classes and visual metrics}
\label{sect:asymp-classes}

We now consider asymptote classes of local $n$-cycles in spaces of
asymptotic rank $n$.

\begin{definition}[asymptote classes] \label{def:f-asymptotic}
Let $X$ be a proper metric space that satisfies
condition~{\rm (CI$_n$)} for $n = \asrk(X) \ge 1$.
We put 
\[
\bZ_{n,\loc}^\infty(X) := \{S \in \bZ_{n,\loc}(X): \Gi(S) < \infty\}
\] 
and call two elements $S,S'$ of this group {\em $\F$-asymptotic}
if $\Fi(S - S') = 0$ (or, equivalently, $\Fi(S - S') < \infty$; see
Proposition~\ref{prop:part-filling} (partial filling)).
This defines an equivalence relation $\sim_\F$ on 
$\bZ_{n,\loc}^\infty(X)$. We denote the quotient space by 
\[
\cZ X := \bZ_{n,\loc}^\infty(X)/{\sim_\F}
\] 
(note that $n = \asrk(X)$ is implicit in $X$) 
and the equivalence class of $S$ by $[S] \in \cZ X$.
The addition $[S] + [S'] := [S + S']$ 
is clearly well-defined, thus $\cZ X$ is an abelian group.
\end{definition}
 
As stated in Theorem~\ref{intro-f-classes}, when $X$ is a $\CAT(0)$ space,
$\cZ X$ turns out to be canonically isomorphic to the group
$\bZ_{n-1,\cs}(\bT X)$ of integral $(n-1)$-cycles in the Tits boundary of~$X$. 
This will be discussed in Section~\ref{sect:cycles-infty}.

Theorem~\ref{thm:constr-minimizers} (constructing minimizers) shows that
every class $[S] \in \cZ X$ contains an area-minimizing
$\til S \in \bZ_{n,\loc}^\infty(X)$, and furthermore every such $\til S$ has
minimal asymptotic density among all members of $[S]$.
We will now show that for any $C > 0$ and $a \ge 0$, the set
\[
\cZ_{C,a} X := \{[S] \in \cZ X: \text{$S$ has $(C,a)$-controlled density}\}
\]
carries a family of metrics analogous to the visual metrics on $\di X$ in
the hyperbolic case.
With the present hypotheses ($X$ satisfies condition~{\rm (CI$_n$)}
for $n = \asrk(X) \ge 1$), a class in $\cZ X$ need not contain a
representative with controlled density; however, under the stronger
assumptions of the subsequent sections, in particular when $X$ is $\CAT(0)$,
every minimizer $\til S \in \bZ_{n,\loc}^\infty(X)$ has controlled density
(see Proposition~\ref{prop:contr-density} and Remark~\ref{rem:monotonicity}).
Note also that every quasiflat $f \colon \R^n \to X$ yields an $S$ with
controlled density (compare Proposition~\ref{prop:qflats}).

First, for any reference point $p \in X$ and any $[S] \in \cZ X$,
we put
\[
\gp{p}{[S]} := \inf\{d(p,\spt(\til S)): \text{$\til S \in [S]$ is minimizing}\}.
\]
Note that $\gp{p}{[S]} = \infty$ if and only if $[S] = [0]$, that is,
$\Fi(S) = 0$ (see the remark after Theorem~\ref{thm:constr-minimizers}).
Clearly $\gp{p}{[-S]} = \gp{p}{[S]}$ and
\[
\bigl| \gp{p}{[S]} - \gp{q}{[S]} \bigr| \le d(p,q).
\]  
If $X$ is a geodesic $\del$-hyperbolic space ($n = 1$) and
$S$ corresponds to a geodesic $\gam \colon \R \to X$ connecting two points
$u,v \in \di X$, then $\gp{p}{[S]}$ agrees, up to
a bounded additive error, with the Gromov product $(u\,|\,v)_p$.
The following result mimics the $\del$-inequality for the Gromov product
of points at infinity (see p.~89 in~\cite{Gro2} and Sect.~2.2 in~\cite{BuyS}). 

\begin{proposition}[$D$-inequality] \label{prop:d-ineq}
Let $X$ be a proper metric space that satisfies
condition~{\rm (CI$_n$)} for $n = \asrk(X) \ge 1$.
Then for all $C > 0$ and $a \ge 0$ there exists $D \ge 0$ such that
\[
\gp{p}{[S+S']} \ge \min\{\gp{p}{[S]},\gp{p}{[S']}\} - D
\]
for all $p \in X$ and $[S],[S'] \in \cZ_{C,a}X$.
\end{proposition}  

\begin{proof}
Let $\eps > 0$.
Pick minimizers $\til S \in [S]$, $\til S' \in [S']$,
and $\hat S \in [S+S']$ such that
\[
d(p,\spt(\hat S)) < \gp{p}{[S+S']} + \eps.
\]
Note that $[S+S'] \in \cZ_{2C,a}X$.
Applying Proposition~\ref{prop:part-filling} (partial filling)
to each of $\til S,\til S',\hat S$, we infer that
$Z := \hat S - (\til S + \til S')$ has $(\til C,\til a)$-controlled density
for some constants $\til C,\til a$ depending only on $X,C,a$. 
Note that $\Fi(Z) = 0$. Since $Z$ is minimizing mod
$Y := \spt(\til S) \cup \spt(\til S')$,
Theorem~\ref{thm:morse-1} (Morse Lemma I) shows that $\spt(\hat S)$ is within
distance at most $D$ from $Y$ for some constant $D$ depending only on $X,C,a$.
Hence,
\[
d(p,Y) \le d(p,\spt(\hat S)) + D < \gp{p}{[S+S']} + \eps + D.
\]
Since $\min\{\gp{p}{[S]},\gp{p}{[S']}\} \le d(p,Y)$, this gives the result.
\end{proof}  

We call a metric $\nu$ on $\cZ_{C,a}X$ {\em visual\/} if
there are $p \in X$, $b > 1$ and $c \ge 1$ such that
\[
c^{-1}b^{-\gp{p}{[S - S']}} \le \nu([S],[S']) \le c\,b^{-\gp{p}{[S-S']}}
\]
for all $[S],[S'] \in \cZ_{C,a}X$. It is easily seen that
any two metrics that are visual with respect to the same parameter $b$
but different base points are bi-Lipschitz equivalent, whereas
any two visual metrics are snowflake equivalent
(compare Theorem~3.2.4 in~\cite{MacT}). In particular, all visual
metrics induce the same topology on $\cZ_{C,a}X$.

\begin{theorem}[visual metrics] \label{thm:visual-metrics}
Let $X$ be a proper metric space that satisfies condition~{\rm (CI$_n$)}
for $n = \asrk(X) \ge 1$, and let $C > 0$ and $a \ge 0$.
Then for every $p \in X$ and every sufficiently small $b > 1$ there exists
a metric $\nu$ on $\cZ_{C,a}X$ that is visual with respect to $p$ and $b$.
Furthermore, $\cZ_{C,a}X$ is compact with respect to any visual
metric.
\end{theorem}  

\begin{proof}
Let $p \in X$ and $b > 1$, and put $\til\nu([S],[S']) := b^{-\gp{p}{[S-S']}}$;
then
\[
\til\nu([S],[S'']) \le \kap \max\{\til\nu([S],[S']),\til\nu([S'],[S''])\}   
\]
for all $[S],[S'],[S''] \in \cZ_{C,a}X$, where $\kap = b^D$ and $D$ is
the constant from Proposition~\ref{prop:d-ineq} associated with
the parameters $2C$ and $a$.
Note that $\til\nu([S],[S']) = 0$ if and only if $[S] = [S']$. 
If $\kap \le 2$, then a standard chain construction yields a metric $\nu$
on $\cZ_{C,a}X$ such that
\[
\frac{1}{2\kap}\,\til\nu([S],[S']) \le \nu([S],[S']) \le \til\nu([S],[S'])    
\]
(see Lemma~2.2.5 in~\cite{BuyS}). Thus $\nu$ is visual with respect to
$p$ and $b$.

To prove the compactness assertion, let $(S_i)$ be a sequence in
$\bZ_{n,\loc}^\infty(X)$ such that each $S_i$ has $(C,a)$-controlled density.
By Theorem~\ref{thm:cptness} (compactness), some
subsequence $(S_{i_j})$ converges in the local flat topology,
hence also weakly, to an $S \in \bZ_{n,\loc}(X)$. For all $p \in X$
and $s > r > a$,
\[
\|S\|(\B{p}{r}) \le \liminf_{j \to \infty}\|S_{i_j}\|(\B{p}{s}) \le C s^n  
\]
by the lower semicontinuity of mass on open sets; thus $S$ has
$(C,a)$-controlled density. Suppose now that $\nu$ is a visual metric
on $\cZ_{C,a}X$ with respect to $p \in X$, and note that
$\nu([S_{i_j}],[S]) \to 0$ if and only if $\gp{p}{[Z_j]} \to \infty$,
where $Z_j := S_{i_j} - S$.
Consider a sequence of minimizers $\til Z_j \in [Z_j]$, and let $\eps > 0$.
Since $\G_{p,r}(Z_j) \le 2C$ for all $r > a$,
Proposition~\ref{prop:part-filling} (partial filling) shows that 
if $r$ is sufficiently large, then $\F_{p,r}(\til Z_j - Z_j) < \eps/2$ for
all $j$. Furthermore, it follows from Lemma~\ref{lem:f-conv}
($\F$-convergence) that $\F_{p,r}(Z_j) \to 0$ for every $r > 0$.
Hence, for every sufficiently large $r > 0$, there is an index $j_0$ such that
\[
\F_{p,r}(\til Z_j) \le \F_{p,r}(\til Z_j - Z_j) + \F_{p,r}(Z_j) < \eps 
\]
for all $j \ge j_0$. Let $V_j \in \bI_{n+1,\cs}(X)$ be such that 
$\spt(\til Z_j - \d V) \cap \B{p}{r} = \es$ and $\M(V_j) < \eps r^{n+1}$.
For a point $x \in \spt(\til Z_j) \cap \B{p}{r/2}$,
Lemma~\ref{lem:fill-density} (filling density) then shows that
$\F_{x,r/2}(\til Z_j) \ge c = c(X) > 0$, thus $\M(V_j) \ge c(r/2)^{n+1}$.
Choosing $\eps = c/2^{n+1}$ we conclude that for every sufficiently large
$r$ there is a $j_0$ such that $d(p,\spt(\til Z_j)) > r/2$ for all
$j \ge j_0$. This shows that $\gp{p}{[Z_j]} \to \infty$ as desired.
\end{proof}

Visual metrics will be discussed further in Remark~\ref{rem:visual-metrics}
and Remark~\ref{rem:hoelder}.


\section{Conical representatives} \label{sect:conical}

Our next goal is to relate $\F$-asymptote classes to geodesic cones
and to cycles at infinity. For this purpose, we now impose a convexity
condition on the metric space~$X$.

A curve $\rho \colon I \to X$ defined on some interval $I \sub \R$ is a
{\em geodesic} if there is a constant $s \ge 0$, the {\em speed}
of $\rho$, such that $d(\rho(t),\rho(t')) = s |t - t'|$ for all $t,t' \in I$.
A geodesic defined on $I = \R_+ := [0,\infty)$ is called a {\em ray}.

\begin{definition}[convex bicombing] \label{def:bicombing}
By a {\em convex bicombing} $\si$ on a metric space $X$ 
we mean a map $\si \colon X \times X \times [0,1] \to X$ such that
\ben
\item   
$\si_{xy} := \si(x,y,\cdot) \colon [0,1] \to X$ is a geodesic from
$x$ to $y$ for all $x,y \in X$;
\item
$t \mapsto d(\si_{xy}(t),\si_{x'y'}(t))$ is convex on $[0,1]$
for all $x,y,x',y' \in X$;
\item
$\im(\si_{pq}) \sub \im(\si_{xy})$ whenever $x,y \in X$ and
$p,q \in \im(\si_{xy})$.
\een
A geodesic $\rho \colon I \to X$ is then called a
{\em $\si$-geodesic} if $\im(\si_{xy}) \sub \im(\rho)$ whenever
$x,y \in \im(\rho)$.
A convex bicombing $\si$ on $X$ is {\em equivariant\/} if
$\gam \circ \si_{xy} = \si_{\gam(x)\gam(y)}$ for every
isometry $\gam$ of $X$ and for all $x,y \in X$.
\end{definition}

Note that in~(3), we do not specify the order of $p$ and $q$ with respect to
the parameter of $\si_{xy}$, in particular $\si_{yx}(t) = \si_{xy}(1-t)$.
In the terminology of~\cite{DesL1}, $\si$ is a {\em reversible} and
{\em consistent convex geodesic bicombing} on $X$.
In Section~10.1 of~\cite{Kle}, metric spaces with such a structure $\si$
are called {\em often convex}. 
This class of spaces includes all $\CAT(0)$ and Busemann spaces 
as well as (linearly) convex subsets of normed spaces; at the same time, 
it is closed under various limit and product constructions such as 
ultralimits, (complete) Gromov--Hausdorff limits, and $l_p$ products 
for $p \in [1,\infty]$. 

A large part of the theory of spaces of nonpositive curvature
extends to this more general setting, see~\cite{Bas,DesL1,DesL2}.
Furthermore, as was shown in~\cite{DesL1,Lan4}, every word hyperbolic
group acts geometrically on a proper metric space of finite topological 
dimension with an equivariant convex bicombing $\si$.
In the recent paper~\cite{Des} it is shown that 
Theorem~\ref{thm:rank-conditions} (rank conditions)
still holds for every proper and cocompact metric space $X$ with a convex 
bicombing.
In fact, Theorem~1.1 in that paper shows that
if the unit ball of some $n$-dimensional normed space $V$ is an asymptotic
subset of $X$, then $V$ itself embeds isometrically into $X$.

Let now $X$ be a proper metric space with a convex bicombing $\si$.
It follows from Section~\ref{subsect:homotopies} that $X$ satisfies 
condition~(CI$_n$) for every $n \ge 1$, 
thus all the preceding results are still at our disposal.
The boundary at infinity of $(X,\si)$ is defined in the usual way, as for
$\CAT(0)$ spaces, except that only $\si$-rays are taken into account.
Specifically, we let $\Ray^\si X$ and $\Ray^\si_1 X$ denote the sets of
all $\si$-rays and $\si$-rays of speed one, respectively, in $X$.
For every pair of rays $\rho,\rho' \in \Ray^\si X$, the function
$t \mapsto d(\rho(t),\rho'(t))$ is convex, and $\rho$ and $\rho'$
are called {\em asymptotic} if this function is bounded.
This defines an equivalence relation $\sim$ on $\Ray^\si X$ as well as
on $\Ray^\si_1 X$.
The {\em boundary at infinity} or {\em visual boundary} of $(X,\si)$ is 
the set
\[
\di X := \Ray^\si_1 X/{\sim}
\]
(whereas $\Ray^\si X/{\sim}$ is the set underlying the Tits cone of $X$,
see the end of Section~\ref{sect:visibility}).
Given $\rho \in \Ray^\si_1 X$ and $p \in X$, there is a unique ray 
$\rho_p \in \Ray^\si_1 X$ asymptotic to $\rho$ with $\rho_p(0) = p$. The set
\[
\olX := X \cup \di X
\]
carries a natural metrizable topology,
analogous to the cone topology for $\CAT(0)$ spaces. With this topology,
$\olX$ is a compact absolute retract, and $\di X$ is a $Z$-set
in $\olX$. See Section~5 in~\cite{DesL1} for details.
For a subset $A \sub X$, the limit set $\di(A)$ is defined as the set of
all points in $\di X$ that belong to the closure of $A$ in~$\olX$.
For a point $p \in X$ we define the geodesic homotopy
\[
h_p \colon [0,1] \times X \to X
\]
by $h_p(\lam,x) := h_{p,\lam}(x) := \si_{px}(\lam)$. Note that
the map $h_{p,\lam} \colon X \to X$ is $\lam$-Lipschitz. 
For a set $A \sub X$,
\[
\C_p(A) := h_p([0,1] \times A)
\]
denotes the geodesic cone from $p \in X$ over $A$, and $\ol\C_p(A)$
denotes its closure in $X$. Similarly, if $\Lam \sub \di X$, then 
$\C_p(\Lam) \sub X$ denotes the union of the traces of the rays emanating 
from $p$ and representing points of $\Lam$.

Let now $S \in \bZ_{n,\loc}(X)$. We write 
\[
\Lam(S) := \di(\spt(S)) \sub \di X
\]
for the {\em limit set\/} of (the support of) $S$, and we call
$S$ \emph{conical\/} if there is a point $p \in X$ such that
\[
h_{p,\lam\#} S = S
\] 
for all $\lam \in (0,1)$. The following lemma collects a number of basic 
properties.

\begin{lemma}[conical] \label{lem:conical}
Let $X$ be a proper metric space with a convex bicombing~$\si$, and suppose 
that $S \in \bZ_{n,\loc}(X)$ is conical with respect to some point $p \in X$. 
Then
\ben
\item
$S \on \B{p}{r} \in \bI_{n,\cs}(X)$ and 
$h_{p,\lam\#}(S \on \B{p}{r}) = S \on \B{p}{\lam r}$ for all $r > 0$ and
$\lam \in (0,1)$;
\item
the functions $r \mapsto \G_{p,r}(S)$ and 
$r \mapsto \F_{p,r}(S)$ are non-decreasing on $(0,\infty)$;
\item
if $S \ne 0$, then $\Fi(S) > 0$ (possibly $\Fi(S) = \infty$);
\item
$\spt(S) \sub \C_p(\Lam(S))$.
\een
\end{lemma}

\begin{proof}
Put $B_r := \B{p}{r}$ for all $r > 0$. 
To see that $S \on B_r \in \bI_{n,\cs}(X)$ for {\em every} $r > 0$, 
note that $S \on B_s \in \bI_{n,\cs}(X)$ for almost every $s > 0$, 
pick such an $s > r$, and put $\lam := r/s$. 
Now $h_{p,\lam\#}(S \on B_s) \in \bI_{n,\cs}(X)$, and since 
$B_s = h_{p,\lam}^{\,-1}(B_r)$ and $h_{p,\lam\#}S = S$, it follows that
\[
h_{p,\lam\#}(S \on B_s) = (h_{p,\lam\#}S) \on B_r = S \on B_r.
\]
From this, the second assertion of~(1) is also clear.

We show~(2). For any $s > r > 0$,
\[
\|S\|(B_r) = \|h_{p,r/s\#}S\|(B_r) \le (r/s)^n\|S\|(B_s),
\]
thus $\G_{p,r}(S)$ is non-decreasing in $r$. Similarly, if there exists 
$V \in \bI_{n+1,\cs}(X)$ such that $\spt(S - \d V) \cap B_s = \es$,
then $\M(h_{p,r/s\#}V) \le (r/s)^{n+1}\M(V)$, and the support of 
$S - \d(h_{p,r/s\#}V) = h_{p,r/s\#}(S - \d V)$ is disjoint from $B_r$,
thus $\F_{p,r}(S) \le \M(V)/s^{n+1}$. Taking the infimum over all such $V$,
we get that $\F_{p,r}(S) \le \F_{p,s}(S)$ (where $\F_{p,s}(S) = \infty$ if no
such $V$ exists).

As for~(3), note that if $S \ne 0$, there is an $s > 0$ such that 
$\spt(S) \cap B_s \ne \es$. Then any $V$ as above must be non-zero,
thus $\F_{p,s}(S) \in (0,\infty]$, and $\Fi(S) \ge \F_{p,s}(S)$ 
by monotonicity. 

Finally, observe that
$\spt(S) = \spt(h_{p,\lam\#}S) \sub h_{p,\lam}(\spt(S))$ for all
$\lam \in (0,1]$. Hence, for every $x_1 \in \spt(S)$ there exist
$x_2,x_3,\ldots \in \spt(S)$ such that $h_{p,1/k}(x_k) = x_1$, and~(4) follows.  
\end{proof}

We now prove a first part of Theorem~\ref{intro-f-classes} stated in the
introduction.

\begin{theorem}[conical representative] \label{thm:conical-repr}
Let $X$ be a proper metric space with a convex bicombing~$\si$ and with
$\asrk(X) = n \ge 2$. Suppose that $S \in \bZ_{n,\loc}^\infty(X)$
and $p \in X$. Then there exists a unique local cycle 
$S_{p,0} \in \bZ_{n,\loc}^\infty(X)$ that is conical with respect to $p$ and\/ 
$\F$-asymptotic to $S$. Furthermore, $\Gi(S_{p,0}) \le \Gi(S)$,
$\Lam(S_{p,0}) \sub \Lam(S)$, and $\spt(S_{p,0}) \sub \C_p(\Lam(S))$.
\end{theorem}

Note that by uniqueness, $S_{p,0} = 0$ if and only if $\Fi(S) = 0$.
For the proof of Theorem~\ref{thm:conical-repr}, we consider the family 
of all
\[
S_{p,\lam} := h_{p,\lam\#}S \in \bZ_{n,\loc}(X)
\]
for $\lam \in (0,1]$. We show that, as $\lam \to 0$, this family 
converges in the local flat topology to the desired local cycle $S_{p,0}$. 

\begin{proof}
Pick any $C > \Gi(S)$. Then there exists an $a \ge 0$ such that
$\G_{p,r}(S) \le C$ for all $r > a$. We write again $B_r := \B{p}{r}$. 
Since $h_{p,\lam}$ is $\lam$-Lipschitz,
\[
\|S_{p,\lam}\|(B_r) \le \lam^n \|S\|(B_{\lam^{-1}r}) \le Cr^n
\]
for all $r > \lam a$ (see Section~\ref{subsect:mass}), 
thus $\G_{p,r}(S_{p,\lam}) \le C$ for all $r > \lam a$.

First we construct partial fillings of $S_{p,\lam} - S$ for a fixed 
$\lam \in (0,1)$. Let $R' > a/2$. Then $\|S\|(B_{2R'}) \le C(2R')^n$,
hence there exists an $R \in (R',2R')$ such 
that $S \on B_R \in \bI_{n,\cs}(X)$,
\[
\M(\d(S \on B_R)) \le 2^nC (R')^{n-1} \le 2^nC R^{n-1},
\]
and $h_{p,\lam\#}(S \on B_R) = S_{p,\lam} \on B_{\lam R} \in \bI_{n,\cs}(X)$.
The truncated geodesic cone 
$T := h_{p\#} \bigl( \bb{\lam,1} \times \d(S \on B_R) \bigr) \in \bI_{n,\cs}(X)$ 
with boundary
\[
\d T = \d(S \on B_R) - h_{p,\lam\#}\d(S \on B_R)
\]
satisfies $\M(T) \le R\,\M(\d(S \on B_R)) \le 2^nC R^n$
(see Section~\ref{subsect:homotopies}). 
It follows that $\|T\|(B_r) \le 2^nC r^n$ for all $r > 0$. 
Hence,
\[
Z := S_{p,\lam} \on B_{\lam R} - S \on B_R + T
\] 
is a cycle satisfying $\G_{p,r}(Z) \le C'$ for all $r > a$ and for some
constant $C' = C'(C,n)$. 
Proposition~\ref{prop:part-filling} (partial filling) shows that
for every $\eps > 0$ there is an $a'_\eps = a'_\eps(X,C',a) \ge 0$
such that if $r > a'_\eps$, there exists $V \in \bI_{n+1,\cs}(X)$ with 
$\spt(Z - \d V) \cap B_r = \es$ and $\M(V) < \eps r^{n+1}$.
If we choose $R'$ sufficiently large, so that $\lam R > r$, then
$\spt(T) \cap B_r = \es$ and $\spt(S_{p,\lam} - S - \d V) \cap B_r = \es$. 
This shows that $\F_{p,r}(S_{p,\lam} - S) < \eps$ 
whenever $\eps > 0$ and $r > a'_\eps$.

Next, suppose that $0 < \lam' < \lam \le 1$. 
Let $\eps > 0$ and $r > \lam a'_\eps$. 
Since $\lam'/\lam < 1$ and $r/\lam > a'_\eps$,
the above result yields that $\F_{p,r/\lam}(S_{p,\lam'/\lam} - S) < \eps$.
Since $h_{p,\lam}$ is $\lam$-Lipschitz, it follows that
$\F_{p,r}(S_{p,\lam'} - S_{p,\lam}) < \eps$.

We can now conclude the proof. Since $\G_{p,r}(S_{p,\lam}) \le C$ for
$r > \lam a$, Theorem~\ref{thm:cptness} (compactness) shows that for some
sequence $\lam_i \downarrow 0$, the respective $S_{p,\lam_i}$ converge in the
local flat topology to a limit $S_{p,0} \in \bZ_{n,\loc}(X)$.
By Lemma~\ref{lem:f-conv} ($\F$-convergence),
$\F_{p,r}(S_{p,0} - S_{p,\lam_i}) \to 0$ for every fixed $r > 0$. 
Using the inequality
\[
\F_{p,r}(S_{p,0} - S_{p,\lam}) 
\le \F_{p,r}(S_{p,0} - S_{p,\lam_i}) + \F_{p,r}(S_{p,\lam_i} - S_{p,\lam}),
\]
we infer that $\F_{p,r}(S_{p,0} - S_{p,\lam}) \le \eps$ whenever
$\lam \in (0,1]$, $\eps > 0$, and $r > \lam a'_\eps$.
This shows at once that $\Fi(S_{p,0} - S) = 0$ and 
that $S_{p,\lam} \to S_{p,0}$ in the local flat topology, as $\lam \to 0$.
To see that $S_{p,0}$ is conical with respect to $p$, note that for any 
$\mu \in (0,1)$, $h_{p,\mu\#}S_{p,0}$ is the weak limit of 
$h_{p,\mu\#}S_{p,\lam} = S_{p,\mu\lam}$ for $\lam \to 0$, which is again $S_{p,0}$. 

Next we show that $\Gi(S_{p,0}) \le \Gi(S)$. For all pairs $s > r > 0$,
\[
\|S_{p,0}\|(\B{p}{r}) \le \liminf_{\lam \to 0}\|S_{p,\lam}\|(\B{p}{s}) \le Cs^n
\]
by the lower semicontinuity of mass with respect to weak convergence 
and since $\G_{p,s}(S_{p,\lam}) \le C$ for $s > \lam a$. 
As $C > \Gi(S)$ was arbitrary, this gives the result. In particular
$S_{p,0} \in \bZ_{n,\loc}^\infty(X)$.

If $S' \in \bZ_{n,\loc}^\infty(X)$ is another local cycle that is 
conical with respect to $p$ and $\F$-asymptotic to $S$, 
then $S' \sim_\F S_{p,0}$ and so $S' = S_{p,0}$ by Lemma~\ref{lem:conical}.

By construction, $\spt(S_{p,\lam}) \sub \C_p(\spt(S))$ for all
$\lam \in (0,1)$. Therefore $\spt(S_{p,0}) \sub \ol\C_p(\spt(S))$ and thus
\[
\Lam(S_{p,0}) \sub \di(\ol\C_p(\spt(S))) = \Lam(S).
\]
Hence, by Lemma~\ref{lem:conical}, $\spt(S_{p,0}) \sub \C_p(\Lam(S_{p,0}))
\sub \C_p(\Lam(S))$.
\end{proof}

A consequence of Theorem~\ref{thm:conical-repr}
(and Proposition~\ref{prop:part-filling}) is the following uniform
density bound.

\begin{proposition}[controlled density] \label{prop:contr-density}
Let $X$ be a proper metric space with a convex bicombing~$\si$ and 
with $\asrk(X) = n \ge 2$. Then for all $C,\eps > 0$ there is a constant
$a \ge 0$ such that every minimizing $S \in \bZ_{n,\loc}^\infty(X)$
with $\Gi(S) \le C$ has $(C+\eps,a)$-controlled density.
\end{proposition}

\begin{proof}
Let $p \in X$. Since $\Fi(S - S_{p,0}) = 0$ and 
$\G_{p,r}(S_{p,0}) \le \Gi(S_{p,0}) \le \Gi(S) \le C$ for all $r > 0$,
Proposition~\ref{prop:part-filling} (partial filling)
shows that for every $\eps > 0$ there is an $a = a(X,C,\eps) \ge 0$
such that $\G_{p,r}(S) \le C + \eps$ for all $r > a$.
As $p$ was arbitrary, this yields the result.
\end{proof}

\begin{remark} \label{rem:monotonicity}
When $X$ is a proper $\CAT(0)$ space, it follows more directly that
every minimizing $S \in \bZ_{n,\loc}^\infty(X)$ has $(C,0)$-controlled 
density for $C := \Gi(S)$, regardless of the asymptotic rank of $X$. 
In fact, for every fixed $p \in X$, the function $r \mapsto \G_{p,r}(S)$ 
is non-decreasing on $(0,\infty)$.
This monotonicity property is shown by an argument very similar 
to the proof of Lemma~\ref{lem:density} (density), using the
sharp cone inequality $\M(T_s) \le (s/n)\,\M(R_s)$
instead of the Euclidean isoperimetric inequality
$\M(T_s) \le \gam\,\M(R_s)^{n/(n-1)}$ (compare Corollary~4.4 in~\cite{Wen3}).
\end{remark}


\section{Visibility and applications} \label{sect:visibility}

Theorem~\ref{thm:conical-repr} (conical representative) shows in particular
that for every $S \in \bZ_{n,\loc}^\infty(X)$ and $p \in X$ there is an 
$S_{p,0} \in [S]$ with support in the geodesic cone $\C_p(\Lam(S))$. 
We now assume in addition that $S \in \bZ_{n,\loc}^\infty(X)$ 
is quasi-minimizing. If both $S$ and $S_{p,0}$ had controlled density,
we could conclude directly from Theorem~\ref{thm:morse-1} (Morse Lemma I)  
that the support of $S$ is within uniformly bounded distance from
$\spt(S_{p,0})$ and hence from $\C_p(\Lam(S))$. The following result,
which subsumes Theorem~\ref{intro-visibility}, provides
a sublinear bound for the general case. As indicated in the introduction,
this may be viewed as an analog of the visibility axiom from~\cite{EbeO}.

\begin{theorem}[visibility property] \label{thm:visibility} 
Let $X$ be a proper metric space with a convex bicombing~$\si$ and with
$\asrk(X) = n \ge 2$.
Then for all $Q \ge 1$, $C > 0$, $a \ge 0$, and $\eps > 0$ there 
exists $r_\eps > 0$ such that the following holds. 
Suppose that $S \in \bZ_{n,\loc}(X)$ is $(Q,a)$-quasi-minimizing and satisfies
$\G_{p,r}(S) \le C$ for some $p \in X$ and for all $r > a$.
If $x \in \spt(S)$ is a point with $d(p,x) \ge r_\eps$, then
\ben
\item
for every $\lam \in (0,1)$ there is an
$x_\lam \in \spt(S)$ such that $d(x,h_{p,\lam}(x_\lam)) < \eps\,d(p,x)$;
\item
there exists a ray $\rho \in \Ray^\si_1 X$ with $\rho(0) = p$ and 
$[\rho] \in \Lam(S)$ such that $d(x,\im(\rho)) < \eps\,d(p,x)$.
\een  
\end{theorem}

We prove~(1) and~(2) in a unified way by bounding the distance
of $x$ from $\spt(S_{p,\lam}) = \spt(h_{p,\lam\#}S)$ for $\lam \in (0,1)$ and
from $\spt(S_{p,0})$, respectively.

\begin{proof}
Let $\lam \in [0,1)$. We know from
Theorem~\ref{thm:conical-repr} (conical representative) and its proof that 
$\Fi(S - S_{p,\lam}) = 0$ and $\G_{p,r}(S_{p,\lam}) \le C$ for all 
$r > \lam a$. In particular, $\G_{p,r}(S - S_{p,\lam}) \le 2C$ for all $r > a$.
Suppose now that $x \in \spt(S)$ and $s > 0$ are such that 
$\B{x}{s} \cap \spt(S_{p,\lam}) = \es$, and put $r_x := d(p,x)$.
Proposition~\ref{prop:part-filling} (partial filling) shows that for every
$\del > 0$ there is a constant $a'_\del = a'_\del(X,C,a) \ge 0$ such that if
$r_x + s > a'_\del$, there exists $V \in \bI_{n+1,\cs}(X)$ with 
$\spt(S - S_{p,\lam} - \d V) \cap \B{p}{r_x + s} = \es$ and
\[
\M(V) < \del (r_x + s)^{n+1}.
\]
Since $\B{x}{s}$ is disjoint from $\spt(S_{p,\lam})$ and contained in
$\B{p}{r_x + s}$, it follows that $\spt(S - \d V) \cap \B{x}{s} = \es$.
Now Lemma~\ref{lem:fill-density} (filling density) shows that 
if $s > 4a$, then
\[
\M(V) \ge c s^{n+1}
\]
for some constant $c > 0$ depending only on $X$ and $Q$. Hence,
\[
s < \bigl( c^{-1}\del \bigr)^{1/(n+1)} (r_x + s)
\]
whenever $r_x = d(p,x) > a'_\del$ and $4a < s < d(x,\spt(S_{p,\lam}))$. 
By choosing $\del$ sufficiently small, in dependence of $n,c,a$ and 
$\eps > 0$, we infer that  
\[
d(x,\spt(S_{p,\lam})) < \eps\,d(p,x)
\]
for all $x \in \spt(S)$ with $d(p,x) > a'_\del$.

From this, (1) and~(2) follow easily.
Note first that if $\lam \in (0,1)$, then
$\spt(S_{p,\lam}) = \spt(h_{p,\lam\#}S) \sub h_{p,\lam}(\spt(S))$;
it thus follows that there is a point
$x_\lam \in \spt(S)$ such that $d(x,h_{p,\lam}(x_\lam)) < \eps\,d(p,x)$.
Similarly, if $\lam = 0$, then $\spt(S_{p,0}) \sub \C_p(\Lam(S))$
by Theorem~\ref{thm:conical-repr} (conical representative), thus there exists
a ray $\rho \in \Ray^\si_1 X$ emanating from $p$ such that 
$[\rho] \in \Lam(S)$ and $d(x,\im(\rho)) < \eps\,d(p,x)$.
\end{proof}

As a by-product of this argument we obtain the following supplement to
Theorem~\ref{thm:conical-repr} (conical representative).

\begin{proposition}[equal limit sets] \label{prop:lim-sets}
Let $X$ be a proper metric space with a convex bicombing~$\si$ and with
$\asrk(X) = n \ge 2$.
If $S \in \bZ_{n,\loc}^\infty(X)$ is quasi-minimizing, then
$\Lam(S_{p,0}) = \Lam(S)$ for every $p \in X$.
\end{proposition}

\begin{proof}
Let $p \in X$. We already know that $\Lam(S_{p,0}) \sub \Lam(S)$. 
On the other hand, given $v \in \Lam(S)$, it follows from the proof
of Theorem~\ref{thm:visibility} that there exist sequences of
points $x_i \in \spt(S)$ and $y_i \in \spt(S_{p,0})$ such that $x_i \to v$
and $d(x_i,y_i) < (1/i) \,d(p,x_i)$. This implies that $y_i \to v$, 
thus $v \in \Lam(S_{p,0})$.
\end{proof}

We now consider an asymptotic Plateau problem.

\begin{theorem}[minimizer with prescribed asymptotics] 
\label{thm:a-plateau}
Let $X$ be a proper metric space with a convex bicombing~$\si$ and with
$\asrk(X) = n \ge 2$. Suppose that $S_0 \in \bZ_{n,\loc}^\infty(X)$ is conical 
with respect to some point $p \in X$. Then there exists a minimizing
$S \in \bZ_{n,\loc}^\infty(X)$ that is $\F$-asymptotic to $S_0$;
thus $S_{p,0} = S_0$. 
Every such $S$ satisfies $\Gi(S) = \Gi(S_0)$ and $\Lam(S) = \Lam(S_0)$. 
Furthermore, if $S' \in \bZ_{n,\loc}^\infty(X)$ is another minimizer 
$\F$-asymptotic to $S_0$, then $\Hd(\spt(S),\spt(S')) \le b$ for some 
constant $b \ge 0$ depending only on $X$ and $\Gi(S_0)$.
\end{theorem}

\begin{proof}  
By Theorem~\ref{thm:constr-minimizers} (constructing minimizers)
there exists a minimizing $S \in \bZ_{n,\loc}(X)$ with $\Fi(S - S_0) = 0$,
and every such $S$ satisfies $\Gi(S) \le \Gi(S_0)$.
Then $S_{p,0} = S_0$ by the uniqueness assertion of
Theorem~\ref{thm:conical-repr} (conical representative), and since
$\Gi(S_{p,0}) \le \Gi(S)$, it follows that $\Gi(S) = \Gi(S_0)$.
By Proposition~\ref{prop:lim-sets}, $\Lam(S) = \Lam(S_{p,0}) = \Lam(S_0)$.

If $S' \in \bZ_{n,\loc}^\infty(X)$ is another minimizer $\F$-asymptotic to $S_0$,
then $\Fi(S - S') = 0$, and by 
Proposition~\ref{prop:contr-density} (controlled density)
$S - S'$ has $(2C,a)$-controlled density for some constants $C,a$
depending only on $X$ and $\Gi(S_0)$. 
Since $S - S'$ is $(1,0)$-quasi-minimizing mod~$\spt(S)$ as well as 
mod~$\spt(S')$, it follows from Theorem~\ref{thm:morse-1} (Morse Lemma I) 
that $\Hd(\spt(S),\spt(S')) \le b$ for some constant $b$ as claimed.
\end{proof}

Proposition~\ref{prop:lim-sets} and Theorem~\ref{thm:a-plateau} 
show in particular that the following three classes of compact subsets
of $\di X$ agree.

\begin{definition}[canonical class of limit sets] \label{def:can-lim-sets}
Let $X$ be a proper metric space with a convex bicombing $\si$ and with
$\asrk(X) = n \ge 2$. We put
\begin{align*} 
\cL X
&:= \{\Lam(S):
\text{$S \in \bZ_{n,\loc}^\infty(X)$ is conical} \} \\
&\phantom{:}= \{\Lam(S):
\text{$S \in \bZ_{n,\loc}^\infty(X)$ is minimizing} \} \\
&\phantom{:}= \{\Lam(S):
\text{$S \in \bZ_{n,\loc}^\infty(X)$ is quasi-minimizing} \}.
\end{align*}
\end{definition}

We now prove Theorem~\ref{intro-dense-orbit}, reformulated 
for spaces with a convex bicombing.

\begin{theorem}[dense orbit] \label{thm:dense-orbit}
Let $X$ be a proper metric space with a convex bicombing~$\si$ and with
$\asrk(X) = n \ge 2$,
and suppose that\/ $\Gam$ is a cocompact group of isometries of $X$.
Then, for every non-empty set $\Lam \in \cL X$, the orbit of $\Lam$ under
the action of\/ $\Gam$, extended to $\olX = X \cup \di X$, is dense in 
$\di X$ (with respect to the cone topology).
\end{theorem}

\begin{proof}
Suppose that $\Lam = \Lam(S)$, where $S \in \bZ_{n,\loc}^\infty(X)$
is minimizing.
By Proposition~\ref{prop:contr-density} (controlled density), $S$ has 
$(C,a)$-controlled density for some constants $C,a$ depending only 
on $X$ and $\Gi(S)$.
Let $p \in X$, and let $\rho_0 \in \Ray^\si_1 X$ be a ray
emanating from $p$. Since $\Gam$ acts cocompactly, there is a constant $b > 0$
such that for every $t \ge 0$ there exist an isometry $\gam_t \in \Gam$ and
a point $x_t \in \gam_t(\spt(S)) = \spt(\gam_{t\#}S)$
such that $d(\rho_0(t),x_t) \le b$. 
Note that $\gam_{t\#}S$ is minimizing, and 
$\G_{p,r}(\gam_{t\#}S) = \G_{\gam_t^{-1}(p),r}(S) \le C$ for all $r > a$.
Hence, given $\eps > 0$, if $t$ is sufficiently large, then by
Theorem~\ref{thm:visibility} there is a ray $\rho \in \Ray^\si_1 X$ 
with $\rho(0) = p$ such that $[\rho] \in \Lam(\gam_{t\#}S)$ and
$d(x_t,\im(\rho)) < \eps\,d(p,x_t) \le \eps(t + b)$. Then
\[
d(\rho_0(t),\im(\rho)) < b + \eps(t + b).
\]
Note that $[\rho] \in \di(\gam_t(\spt(S))) = \ol\gam_t(\Lam)$
for the extension $\ol\gam_t$ of $\gam_t$ to $\olX$.
Since $\eps > 0$ was arbitrary, this shows that every neighborhood of
$[\rho_0]$ in $\di X$ contains a point of the orbit of $\Lam$.
\end{proof}

Theorem~\ref{thm:visibility} shows that the support of a quasi-minimizer
$S \in \bZ_{n,\loc}^\infty(X)$ lies within sublinear distance 
from $\C_p(\Lam(S))$, in terms of the distance to~$p$. 
Next we show that, conversely, the entire cone $\ol\C_p(\spt(S))$ is within 
sublinear distance from $\spt(S)$; however, the estimate now depends on $S$ 
and $p$ rather than just on the data of~$S$. The proof relies on 
Theorem~\ref{thm:visibility} and a ball packing argument.

\begin{theorem}[asymptotic conicality] \label{thm:conicality}
Let $X$ be a proper metric space with a convex bicombing $\si$ and with
$\asrk(X) = n \ge 2$.
Suppose that $S \in \bZ_{n,\loc}^\infty(X)$ is quasi-minimizing,
and $p \in X$. Then for all $\eps > 0$ there exists $r > 0$ such that
\[
d(y,\spt(S)) < \eps \,d(p,y)
\]
whenever $y \in \ol\C_p(\spt(S))$ and $d(p,y) \ge r$.
\end{theorem}

\begin{proof}
We consider the family of the compact sets 
\[
K_s := \spt(S) \cap \B{p}{s}
\]
for $s > 0$.
Let $\mu > 0$. It follows from Theorem~\ref{thm:visibility} that there
exists an $r > 0$ such that for all $s \ge r$, $x \in K_s$, and
$\lam \in (0,1]$, there is a point $x' \in \spt(S)$ such that
\[
d(x,h_{p,\lam}(x')) < \mu s.
\]
Then $\lam \,d(p,x') = d(p,h_{p,\lam}(x')) \le d(p,x) + \mu s \le (1+\mu)s$.
Hence, given any $t \ge s$, by choosing $\lam := \min\{1, (1+\mu)s/t\}$ and
$x' := x$ in the case that $\lam = 1$, we get that $d(p,x') \le t$. Then
\[
d(h_{p,\lam}(x'),h_{p,s/t}(x')) = \left( \lam - \frac{s}{t} \right) d(p,x')
\le \lam t - s \le \mu s. 
\]
We conclude that for every $x \in K_s$ and every $t \ge s$ there exists
an $x' \in K_t$ such that $d(x,h_{p,s/t}(x')) < 2\mu s$.
Furthermore, if $(y,y') \in K_s \times K_t$ is another such pair with
$d(y,h_{p,s/t}(y')) < 2\mu s$, then
\[
\frac{1}{t}\,d(x',y') \ge \frac{1}{s}\,d(h_{p,s/t}(x'),h_{p,s/t}(y')) >
\frac{1}{s}\,d(x,y) - 4\mu
\]
by convexity.

Let now $\eps > 0$.
For $\del > 0$, denote by $\cN_{\del,s}$ the maximal possible cardinality 
of a $\del s$-separated set $N \sub K_s$ 
(see Section~\ref{subsect:metric}).
By the assumptions on $S$ and Lemma~\ref{lem:density} (density), 
$\bar\cN_\del := \limsup_{s \to \infty} \cN_{\del,s}$ is finite.
Using the monotonicity of $\bar\cN_\del \in \Z$ in $\del$ we now fix 
$\del,\mu > 0$ such that $\bar\cN_{\del + 8\mu} = \bar\cN_\del$ and 
$\del + 2\mu \le \eps$. Then we choose $r > 0$ so large that the result of 
the first part of the proof holds, $\cN_{\del + 8\mu,r} = \bar\cN_{\del + 8\mu}$,
and $\cN_{\del,t} \le \bar\cN_\del$ for all $t \ge r$. Let $N_r \sub K_r$
be a $(\del + 8\mu)r$-separated set with maximal cardinality 
$|N_r| = \cN_{\del + 8\mu,r} = \bar\cN_{\del + 8\mu} = \bar\cN_\del$.
For all $t \ge s \ge r$, it follows from the first part of the proof that
there exists a bijection $f$ from $N_r$ to a $(\del + 4\mu)s$-separated set
$N_s \sub K_s$ as well as a bijection $g$ from $N_s$ to a $\del t$-separated
set $N_t \sub K_t$ such that $d(x,h_{p,s/t}(x')) < 2\mu s$ for all $x \in N_s$
and $x' := g(x) \in N_t$. 
Now $|N_t| \le \cN_{\del,t} \le \bar\cN_\del = |N_r| = |N_t|$, thus $N_t$ is in 
fact maximal and forms a $\del t$-net in $K_t$.

Finally, suppose that $y$ is a point in $\C_p(\spt(S))$ with 
$s := d(p,y) \ge r$. Then $y = h_{p,s/t}(y')$ for some $y' \in \spt(S)$, 
where $t := d(p,y') \ge s$. As we have just shown, there exist $x \in K_s$ 
and $x' \in K_t$ such that
$d(y',x') \le \del t$ and $d(h_{p,s/t}(x'),x) < 2\mu s$,
thus $d(y,h_{p,s/t}(x')) \le (s/t) \,d(y',x') \le \del s$ and 
\[
d(y,x) \le d(y,h_{p,s/t}(x')) + d(h_{p,s/t}(x'),x) < (\del + 2\mu)s \le \eps s.
\]
Hence, $d(y,\spt(S)) < \eps\, d(p,y)$. This yields the result.
\end{proof}

We now turn to the Tits geometry. As a first application of 
Theorem~\ref{thm:conicality} we will show that the limit sets
$\Lam \in \cL X$ are compact with respect to the Tits topology.
 
For a proper metric space $X$ with a convex bicombing $\si$, 
the \emph{Tits cone} of $(X,\si)$ is defined as the set 
\[
\CT X := \Ray^\si X/{\sim}
\]
(see Section~\ref{sect:conical}), 
equipped with the metric given by 
\[
\dT([\rho],[\rho']) := \lim_{t \to \infty} \frac{1}{t}\,d(\rho(t),\rho'(t)).
\]
Note that $t \mapsto d(\rho(t),\rho'(t))$ is convex,
thus $t \mapsto d(\rho(t),\rho'(t))/t$ is non-decreasing if $\rho,\rho'$
are chosen such that $\rho(0) = \rho'(0)$. From this it is easily 
seen that $\CT X$ is complete.
On $\CT X$, multiplication by a scalar $\lam \in \R_+$ is 
defined by $\lam[\rho(\,\cdot\,)] := [\rho(\lam\,\cdot\,)]$.
This yields a homothety
\[
h_\lam \colon \CT X \to \CT X,
\]
thus $h_\lam(v) = \lam v$ and $\dT(h_\lam(v),h_\lam(v')) = \lam \,\dT(v,v')$. 
The {\em cone vertex} $o$ of $\CT X$ is the class of the constant rays.
For every base point $p \in X$ there exists a canonical $1$-Lipschitz map 
\[
\can_p \colon \CT X \to X
\] 
such that $\can_p([\rho]) = \rho(1)$ for all $\rho \in \Ray^\si X$ 
with $\rho(0) = p$. The {\em Tits boundary} of $(X,\si)$ is the unit sphere
\[
\bT X := \Sph{o}{1} = \Ray^\si_1 X/{\sim}
\]
in $\CT X$, endowed with the topology induced by $\dT$. This topology
is finer than the cone topology on the visual boundary $\di X$, 
which agrees with $\bT X$ as a set. However, the following holds.

\begin{proposition}[compact limit sets] \label{prop:cpt-lim-sets}
Let $X$ be a proper metric space with a convex bicombing $\si$ and with
$\asrk(X) = n \ge 2$. Then every $\Lam \in \cL X$ is still compact
when viewed as a subset of $\bT X$.
\end{proposition}

\begin{proof}
Suppose that $\Lam = \Lam(S)$, where $S \in \bZ_{n,\loc}^\infty(X)$ is
quasi-minimizing. Fix $p \in X$, and let $\eps > 0$.
Let $N \sub \Lam$ be a finite $3\eps$-separated set; thus $\dT(u,u') > 3\eps$
for distinct $u,u' \in N$. For $r > 0$ sufficiently large, 
$\can_p(rN)$ is $3\eps r$-separated, and every point in this set is at 
distance less than $\eps r$ from $\spt(S)$ by Theorem~\ref{thm:conicality}. 
This yields an $\eps r$-separated subset of $\spt(S) \cap \B{p}{r+\eps r}$ 
of the same cardinality as $N$. For $r$ sufficiently large, it then follows 
from Lemma~\ref{lem:density} (density) that the cardinality of such sets is 
bounded from above by a constant depending on $\eps$ but not on~$r$.
We conclude that $\Lam$ is totally bounded.
Since $\CT X$ is complete and $\Lam$ is closed in the Tits topology,
this gives the result.
\end{proof}


\section{Cycles at infinity} \label{sect:cycles-infty}

In this section we show that if $S \in \bZ_{n,\loc}^\infty(X)$ is 
conical with respect to $p \in X$, then the cone $\R_+\Lam \sub \CT X$ over 
the limit set $\Lam = \Lam(S) \in \cL X$ is the support of a unique local 
$n$-cycle $\Sig$ in $\CT X$ satisfying $\can_{p\#}\Sig = S$. 
We then complete the proofs of Theorem~\ref{intro-f-classes} and 
Theorem~\ref{intro-t-plateau}.

In general, Tits cones are not locally compact, therefore the theory of local 
currents from~\cite{Lan3}, which depends on the supply of compactly supported 
Lipschitz functions, is not directly applicable to $\CT X$. 
However, by Proposition~\ref{prop:cpt-lim-sets} above, 
$\R_+\Lam$ is proper, and $\Sig$ will be constructed as a 
current in its own support $\spt(\Sig) = \R_+\Lam$.
Thus, we \mbox{(re-)define} $\bZ_{n,\loc}(\CT X)$
as the collection of all local cycles $\Sig \in \bZ_{n,\loc}(K_\Sig)$
such that $K_\Sig \sub \CT X$ is proper and $\spt(\Sig) = K_\Sig$
(compare the discussion after Proposition~3.3 in~\cite{Lan3}). 
The sum of two elements $\Sig,\Sig' \in \bZ_{n,\loc}(\CT X)$
may be formed by viewing them temporarily as currents 
in $K_{\Sig} \cup K_{\Sig'}$; thus $\bZ_{n,\loc}(\CT X)$ is an abelian group.
The complexes $\bI_{*,\loc}(\CT X)$, $\bI_{*,\cs}(\CT X)$ and $\bI_{*,\cs}(\bT X)$ 
are understood similarly.

We start with a basic fact.

\begin{lemma}[uniform convergence] \label{lem:unif-conv}
Let $X$ be a proper metric space with a convex bicombing $\si$. 
Suppose that $K$ is a compact subset of $\CT X$, and $p \in X$.
Then for every $\eps > 0$ there is an $r_\eps > 0$ such that
\[
\dT(u,v) - \eps \le r^{-1} d(\can_p(ru),\can_p(rv)) \le \dT(u,v)
\]
for all $r \ge r_\eps$ and $u,v \in K$. In particular, $K$ is an asymptotic 
subset of $X$ as defined in Definition~\ref{def:as-rank}. 
\end{lemma}

\begin{proof}
For every $r > 0$, the map $u \mapsto \can_p(ru)$ is $r$-Lipschitz
on $\CT X$. It follows that the function
$\rho_r \colon (u,v) \mapsto r^{-1} d(\can_p(ru),\can_p(rv))$ is $1$-Lipschitz 
with respect to the $l_1$ product metric on $\CT X \times \CT X$. 
Moreover, as $r \to \infty$, $\rho_r \to \dT$ pointwise on 
$\CT X \times \CT X$ by the definition of $\dT$. 
Hence the convergence is uniform on $K \times K$ for every compact 
set $K \sub \CT X$. 

In particular, the rescaled sets $(\can_p(rK),r^{-1}d)$ 
converge in the Gromov--Hausdorff topology to $K$.
\end{proof}

\begin{remark} \label{rem:dimension}
It follows from Lemma~\ref{lem:unif-conv} and Remark~\ref{rem:as-rank} that 
if $\asrk(X) = n$ and $m > n$, then $\CT X$ contains no set bi-Lipschitz
homeomorphic to a compact subset of $\R^m$ with positive Lebesgue measure, 
and this implies that $\bI_{m,\loc}(\CT X) = \{0\}$ and 
$\bI_{m-1,\cs}(\bT X) = \{0\}$.

As a consequence, if $\asrk(X) = n$, then {\em every\/} local cycle 
$\Sig \in \bZ_{n,\loc}(\CT X)$ is conical with respect to the cone 
vertex $o$, that is, $h_{\lam\#}\Sig = \Sig$ for all $\lam > 0$. To see this, 
consider the radial homotopy $H \colon (t,v) \mapsto (1-t+\lam t)v$ of 
$\CT X$; then $h_{\lam\#}\Sig - \Sig$  equals the boundary of 
$H_\#(\bb{0,1} \times \Sig) = 0 \in \bI_{n+1,\loc}(\CT X)$.
\end{remark}

We now prove the following general result, which is independent of the 
asymptotic rank. However, the assumption $\asrk(X) = n$ will guarantee that 
$\Lam(S) \sub \bT X$ is compact.

\begin{theorem}[lifting cones] \label{thm:lift-cones}
Let $X$ be a proper metric space with a convex bicombing $\si$. 
Suppose that $S \in \bZ_{n,\loc}^\infty(X)$ is conical with respect to
some point $p \in X$, and $\Lam := \Lam(S)$ is compact in the Tits topology.
Then there is a unique local cycle $\Sig \in \bZ_{n,\loc}(\CT X)$ such 
that $\can_{p\#} \Sig = S$. Moreover, $\Sig \in \bZ_{n,\loc}(\CT X)$ is conical
with respect to $o$, $\M(\Sig \on \B{o}{1}) = \Gi(S)$, 
$\spt(\Sig) = \R_+\Lam$, and $\spt(\d(\Sig \on \B{o}{1})) = \Lam$.
\end{theorem}

Note that since $\Sig \in \bZ_{n,\loc}(\CT X)$ is conical, 
$\Sig \on \B{o}{\lam} \in \bI_{n,\cs}(\CT X)$ for all $\lam > 0$
(compare Lemma~\ref{lem:conical} (conical)). 

To construct $\Sig$ we will consider the family of all 
$S_r := S \on \B{p}{r} \in \bI_{n,\cs}(X)$ for $r > 0$. 
First we embed each $S_r$ by a map that dilates
all distances by the factor $1/r$ into a fixed compact metric space $Y$.
The embedded family converges, as $r \to \infty$,
to an integral current in $Y$ with support in an
isometric copy of the cone $K := [0,1]\Lam \sub \CT X$, and this yields
$\Sig \on \B{o}{1}$. As regards $Y$, we will use the following general fact.
Given any compact metric space $(K,d_K)$ with diameter $D$, the set $Y$
of all $1$-Lipschitz functions $y \colon K \to [0,D]$, endowed 
with the metric defined by
\[
d_Y(y,y') := \sup_{v \in K}|y(v) - y'(v)|,
\]
is a compact convex subspace of $l_\infty(K)$, and the map
$u \mapsto d_K(u,\cdot)$ is an isometric embedding of $K$ into~$Y$.
Furthermore, $Y$ is an injective metric space; that is, every $1$-Lipschitz
map $\rho \colon A \to Y$ defined on a subset $A$ of a metric space $B$
extends to a $1$-Lipschitz map $\bar\rho \colon B \to Y$. In fact,
such an extension is given by
\[
\bar\rho(b)(v) := \sup_{a \in A} \max\{\rho(a)(v) - d(a,b),0\}
\]
for all $b \in B$ and $v \in K$. 

\begin{proof}
For $s > r > 0$, we put $\pi_r := \can_p \circ h_r \colon \CT X \to X$
and $\pi_{s,r} := h_{p,r/s} \colon X \to X$. Note that $\pi_r$ is $r$-Lipschitz,
$\pi_{s,r}$ is $(r/s)$-Lipschitz, and $\pi_r = \pi_{s,r} \circ \pi_s$.

Let first $K \sub \CT X$ be an arbitrary compact set, and put 
$K_r := \pi_r(K)$. Let $(Y,d_Y)$ be the compact convex subspace of 
$l_\infty(K)$ as described before the proof, and let
\[
f \colon K \to Y, \quad f(u) := \dT(u,\cdot),
\]
denote the canonical isometric embedding of $K$ into $Y$.
Similarly, since $\pi_r$ is $r$-Lipschitz, there is a map 
\[
f_r \colon K_r \to Y, \quad f_r(x) := r^{-1}d(x,\pi_r(\cdot)),
\]
and since $\pi_r$ maps $K$ onto $K_r$, it follows that
\[
d_Y(f_r(x),f_r(x')) = r^{-1} \sup_{v \in K}
\bigl| d(x,\pi_r(v)) - d(x',\pi_r(v)) \bigr| = r^{-1}d(x,x')
\]
for all $x,x' \in K_r$. Note also that $f_r(p) = f(o) =: y_0 \in Y$.

Let $\eps > 0$. By Lemma~\ref{lem:unif-conv} there is an $r_\eps > 0$
such that if $s > r \ge r_\eps$, then 
$s^{-1}d(\pi_s(u),\pi_s(v)) \le \dT(u,v) \le r^{-1}d(\pi_r(u),\pi_r(v)) + \eps$
for all $u,v \in K$. We infer that
\begin{align*}
&d_Y(f_s(\pi_s(u)),f_r(\pi_r(u))) \\
&\qquad = \sup_{v \in K} \bigl| s^{-1}d(\pi_s(u),\pi_s(v))
- r^{-1}d(\pi_r(u),\pi_r(v)) \bigr| \le \eps
\end{align*}
for all $u \in K$ and hence
\[
d_Y(f_s(x),(f_r \circ \pi_{s,r})(x)) \le \eps
\] 
for all $x \in K_s$. Similarly, 
\[
d_Y(f(u),(f_r \circ \pi_r)(u)) \le \eps
\]
for all $u \in K$. Thus $f_r(K_r)$ lies within distance $\eps$ of $f(K)$.

We now apply this construction for the cone $K := [0,1]\Lam \sub \CT X$. 
Let $C := \Gi(S)$. By Lemma~\ref{lem:conical} (conical), for all $s > r > 0$,
$S_r := S \on \B{p}{r} \in \bI_{n,\cs}(X)$, $\pi_{s,r\#}S_s = S_r$,
$\M(S_r) \le Cr^n$, and $\spt(S) \sub \C_p(\Lam)$; 
thus $\spt(S_r) \sub K_r = \pi_r(K)$. Since $\pi_{s,r}$ is $r/s$-Lipschitz
and $\pi_{s,r\#}(\d S_s) = \d S_r$, it follows that 
$s^{n-1}\M(\d S_r) \le r^{n-1}\M(\d S_s)$, and an integration
over $s$ yields
\[
\frac{R^n - r^n}{n}\M(\d S_r) \le r^{n-1}\M(S_R) \le Cr^{n-1}R^n
\]
for all $R > r$; thus $\M(\d S_r) \le nCr^{n-1}$. Since $f_r \colon K_r \to Y$
is $(1/r)$-Lipschitz, we get the uniform bounds
\[
\M(f_{r\#}S_r) \le C, \quad \M(\d(f_{r\#}S_r)) = \M(f_{r\#}(\d S_r)) \le nC.
\]
For $\eps > 0$ and $s > r \ge r_\eps$, let $H \colon [0,1] \times K_s \to Y$ 
be the affine homotopy from $f_s$ to $f_r \circ \pi_{s,r}$ in 
$Y \sub l_\infty(K)$. Then $H(t,\cdot)$ is $(1/s)$-Lipschitz for every 
$t \in [0,1]$, and $H(\cdot,x)$ is a segment of length at most $\eps$
for every $x \in K_s$. It follows that the family $(f_{r\#}S_r)_{r > 0}$ 
is Cauchy with respect to the flat distance $\cF$ on $\bI_{n,\cs}(Y)$
(see Sections~\ref{subsect:homotopies} and~\ref{subsect:convergence}),
and by Theorem~\ref{thm:cptness} (compactness) there exists
a current $\bar\Sig_1 \in \bI_{n,\cs}(Y)$ such that 
\[
\lim_{r \to \infty}\cF(f_{r\#}S_r - \bar\Sig_1) = 0.
\]
Note that $\M(\bar\Sig_1) \le C$ and $\spt(\bar\Sig_1) \sub f(K)$,
furthermore $\spt(\d\bar\Sig_1) \sub \Sph{y_0}{1}$ because
$\spt(\d(f_{r\#}S_r)) \sub f_r(\spt(\d S_r)) \sub \Sph{y_0}{1}$ for all $r > 0$.
Via the isometric embedding $f^{-1} \colon f(K) \to \CT X$ 
we get a current $\Sig_1 := (f^{-1})_\#\bar\Sig_1 \in \bI_{n,\cs}(\CT X)$ 
with $\spt(\Sig_1) \sub K$, $\spt(\d\Sig_1) \sub \Lam \sub \Sph{o}{1}$,
and $\M(\Sig_1) \le C = \Gi(S)$.

Next we show that for each $r > 0$, $\pi_{r\#} \Sig_1 = S_r$.
We know that if $\eps > 0$ and $s \ge r_\eps$, then 
$d_Y(f(u),(f_s \circ \pi_s)(u)) \le \eps$ for all $u \in K$. 
Since $f_\#\Sig_1 = \bar\Sig_1$, this yields
\[
\lim_{s \to \infty} \cF(f_{s\#}(\pi_{s\#}\Sig_1) - \bar\Sig_1) = 0.
\]
Putting $T_s := f_{s\#}(\pi_{s\#}\Sig_1 - S_s)$, we get that 
$\lim_{s \to \infty}\cF(T_s) = 0$. For every $s > r$, the $1$-Lipschitz 
map $\rho_s := f_r \circ \pi_{s,r} \circ f_s^{\,-1} \colon f_s(K_s) \to f_r(K_r)$
satisfies $\rho_{s\#}T_s = T_r$ and possesses a $1$-Lipschitz extension
$\bar\rho_s \colon Y \to Y$. It follows that $\cF(T_r) \le \cF(T_s)$ for all 
$s > r$, thus $\cF(T_r) = 0$ and therefore $T_r = 0$. 
Hence, $\pi_{r\#}\Sig_1 - S_r = (f_r^{\,-1})_\#T_r = 0$, as claimed.

As a consequence, $\M(S_r) \le r^n\M(\Sig_1)$ and 
$\spt(\d S_r) \sub \pi_r(\spt(\d\Sig_1))$ for all $r > 0$, thus
$\Gi(S) \le \M(\Sig_1)$ and $\Lam \sub \spt(\d\Sig_1)$. Hence, in view of the
relations shown above, $\M(\Sig_1) = \Gi(S)$ and $\spt(\d\Sig_1) = \Lam$.

Finally, consider the family $\{\Sig_\lam\}_{\lam > 0}$ in $\bI_{n,\cs}(\CT X)$
such that $\Sig_\lam = h_{\lam\#}\Sig_1$ for every $\lam > 0$. Then 
$\pi_{r\#}\Sig_\lam = S_{\lam r}$ for all $r > 0$, and we claim that 
$\Sig_\lam$ is the unique element of $\bI_{n,\cs}(\CT X)$ with this property.
Let $\Sig'$ be any non-zero element of $\bI_{n,\cs}(\CT X)$. It suffices to 
show that $\pi_{r\#}\Sig' \ne 0$ for some $r > 0$. 
Put $K := \spt(\Sig')$, $K_r := \pi_r(K)$, and define $Y$, $f$, and $f_r$ 
as above. Then it follows that 
$\lim_{r \to \infty}\cF(f_{r\#}(\pi_{r\#}\Sig') - f_\#\Sig') = 0$.
Since $f_\#\Sig' \ne 0$, this implies the claim.
Now if $0 < \lam < \lam'$, then $\B{o}{\lam} = \pi_r^{-1}(\B{p}{\lam r})$ 
and hence
\[
\pi_{r\#}(\Sig_{\lam'} \on \B{o}{\lam}) 
= (\pi_{r\#}\Sig_{\lam'}) \on \B{p}{\lam r}
= S_{\lam'r} \on \B{p}{\lam r} = S_{\lam r}
\]
for all $r > 0$; therefore $\Sig_{\lam'} \on \B{o}{\lam} = \Sig_\lam$ by 
uniqueness. It follows that the family $\{\Sig_\lam\}_{\lam > 0}$ determines 
a local cycle $\Sig \in \bZ_{n,\loc}(\CT X)$ such that $\Sig \on \B{o}{\lam} = 
\Sig_\lam$ for all $\lam > 0$, and it is easily verified that $\Sig$ has 
the desired properties.
Note that $\lam\Lam = \spt(\d\Sig_\lam) \sub \spt(\Sig) \sub \R_+\Lam$ 
for all $\lam > 0$, thus $\spt(\Sig) = \R_+\Lam$.
\end{proof}

From Theorem~\ref{thm:lift-cones} we obtain the following result which,
in conjunction with Theorem~\ref{thm:conical-repr} (conical representative) 
and Proposition~\ref{prop:lim-sets} (equal limit sets), 
establishes Theorem~\ref{intro-f-classes} stated in the introduction.

\begin{theorem}[Tits boundary] \label{thm:f-classes}
Let $X$ be a proper metric space with a convex bicombing $\si$ and with
$\asrk(X) = n \ge 2$. Then for every $S \in \bZ_{n,\loc}^\infty(X)$ there exists
a unique local cycle $\Sig \in \bZ_{n,\loc}(\CT X)$ such that $\can_{p\#}\Sig = 
S_{p,0}$ for all $p \in X$; furthermore $\Sig$ is conical with respect to $o$, 
and the slice $\d(\Sig \on \B{o}{1})$ defines an element 
$\bT S = \bT[S] \in \bZ_{n-1,\cs}(\bT X)$ with $\spt(\bT S) = \Lam(S_{p,0})$ 
for all $p \in X$. This yields an isomorphism 
\[
\bT \colon \cZ X \to \bZ_{n-1,\cs}(\bT X). 
\] 
\end{theorem}

\begin{proof}
Let $S \in \bZ_{n,\loc}^\infty(X)$, and let $p,p' \in X$. 
By Theorem~\ref{thm:conical-repr},
$S_{p,0}$ and $S_{p',0}$ are the unique representatives of $[S]$ 
that are conical with respect to $p$ and $p'$, respectively. 
Theorem~\ref{thm:lift-cones} together with 
Proposition~\ref{prop:cpt-lim-sets} (compact limit sets) then shows that
there exist unique elements $\Sig,\Sig' \in \bZ_{n,\loc}(\CT X)$ such that 
$\can_{p\#}\Sig = S_{p,0}$ and $\can_{p'\#}\Sig' = S_{p',0}$; furthermore 
$\Sig,\Sig'$ are conical with respect to $o$, 
and $\spt(\d(\Sig \on \B{o}{1})) = \Lam(S_{p,0})$. 
Now $\can_{p\#}\Sig' \in \bZ_{n,\loc}^\infty(X)$ is conical with respect to $p$,
and using the $\si$-homotopy $H \colon [0,1] \times \CT X \to X$
from $\can_p$ to $\can_{p'}$ one can easily check that 
$\can_{p\#}\Sig'$ is $\F$-asymptotic to $S_{p',0}$ and hence to $S$.
It follows from the above uniqueness assertions 
that $\can_{p\#}\Sig' = S_{p,0}$ and $\Sig' = \Sig$.
This shows that $\can_{p'\#}\Sig = S_{p',0}$ for all $p' \in X$. 
In particular, $\Sig$ depends only on $[S]$. 
Viewing $\d(\Sig \on \B{o}{1}) \in \bZ_{n-1,\cs}(\CT X)$ as an element 
$\bT S = \bT[S] \in \bZ_{n-1,\cs}(\bT X)$, we get a map 
$\bT \colon \cZ X \to \bZ_{n-1,\cs}(\bT X)$, and it is easily verified 
that this is an isomorphism. 
\end{proof}

Returning to the asymptotic Plateau problem, we may now reformulate 
Theorem~\ref{thm:a-plateau} as follows.

\begin{theorem}[minimizer with prescribed Tits data] \label{thm:t-plateau}
Let $X$ be a proper metric space with a convex bicombing $\si$ and with
$\asrk(X) = n \ge 2$. Then for every cycle $R \in \bZ_{n-1,\cs}(\bT X)$
there exists an area-minimizing local cycle $S \in \bZ_{n,\loc}^\infty(X)$
with $\bT S = R$. Every such $S$ satisfies $\Lam(S) = \spt(R)$
and $\M(R)/n \le \Gi(S) \le \M(R)$.
\end{theorem}  

\begin{proof}
Let $R \in \bZ_{n-1,\cs}(\bT X)$. By Theorem~\ref{thm:f-classes}
and Theorem~\ref{thm:conical-repr} there is a conical local cycle
$S_0 \in \bZ_{n,\loc}^\infty(X)$ with Tits boundary $\bT S_0 = R$,
and $\Lam(S_0) = \spt(R)$. By Theorem~\ref{thm:a-plateau} there exists
a minimizing $S \in [S_0]$, and every such $S$ satisfies
$\Lam(S) = \Lam(S_0) = \spt(R)$ and $\Gi(S) = \Gi(S_0)$.
Note that $S \in [S_0]$ if and only if $\bT S = R$.
By Theorem~\ref{thm:lift-cones} and the coarea inequality,
\[
\Gi(S_0) = \M(\Sig \on \B{o}{1}) \ge \int_0^1 \lam^{n-1}\M(R) \,d\lam
= \M(R)/n,
\]
and since $\Sig \on \B{o}{1}$ agrees with the cone over $R$,
$\M(\Sig \on \B{o}{1}) \le \M(R)$.
\end{proof}

When $X$ is a $\CAT(0)$ space, the last inequality holds with $\M(R)/n$ in
place of $\M(R)$; then $\Gi(S) = \M(R)/n$ for every
minimizing $S \in \bZ_{n,\loc}^\infty(X)$ with $\bT S = R$.
Furthermore, $\G_{p,r}(S) \le \Gi(S)$ for all
$p \in X$ and $r > 0$ by monotonicity (see Remark~\ref{rem:monotonicity}),
and $\lim_{r \to 0} \G_{p,r}(S) \ge \om_n$ for $\|S\|$-almost every $p$
(see~\cite{Wen3}, (4.28)); thus $\Gi(S) \ge \om_n$ whenever $R \ne 0$.
This proves Theorem~\ref{intro-t-plateau}.

\begin{remark} \label{rem:visual-metrics}
By the above results, we may rephrase
Theorem~\ref{thm:visual-metrics} (visual metrics) in terms of cycles at
infinity. For a reference point $p \in X$ and $[S] \in \cZ X$, we put
$\gp{p}{\bT[S]} := \gp{p}{[S]}$, thus
\[
\gp{p}{R} = \inf\{d(p,\spt(\til S)):
\text{$\til S \in \bZ_{n,\loc}^\infty(X)$ is minimizing, $\bT \til S = R$}\}  
\]
for all $R \in \bZ_{n-1,\cs}(\bT X)$. Let $C > 0$ and $a \ge 0$.
Then, for every sufficiently small $b > 1$, there exist a constant $c \ge 1$
and a metric $\nu$ on $\bT(\cZ_{C,a}X) \sub \bZ_{n-1,\cs}(\bT X)$ satisfying
\[
c^{-1}b^{-\gp{p}{R-R'}} \le \nu(R,R') \le c\,b^{-\gp{p}{R-R'}}
\]
for all $R,R' \in \bT(\cZ_{C,a}X)$; furthermore $\bT(\cZ_{C,a}X)$ is compact
with respect to any such metric.
Note that if $[S] \in \cZ_{C,a}X$ and $\til S \in [S]$ is minimizing,
then $\Gi(\til S) \le C$
by Theorem~\ref{thm:constr-minimizers} (constructing minimizers)
and thus
\[
\bT(\cZ_{C,a}X) \sub \{R \in \bZ_{n-1,\cs}(\bT X): \M(R) \le nC\}
\]
by Theorem~\ref{thm:t-plateau}.
When $X$ is $\CAT(0)$, these two sets agree for each $a \ge 0$.
\end{remark}  


\section{Quasi-isometries} \label{sect:qi}

We now turn to quasi-isometric embeddings of $X$ into another proper metric
space $\bar X$ with a convex bicombing.

The following auxiliary result will be used in conjunction with 
Lemma~\ref{lem:doubling} (doubling). 

\begin{proposition}[Lipschitz extension] \label{prop:lip-ext}
Suppose that $X$ is a metric space, $\bar X$ is a metric space 
with a convex bicombing $\bar\si$, 
and $A \sub X$ is a non-empty closed set that is doubling. 
Then there is a constant $\mu \ge 1$, depending only on the 
doubling constant, such that for every $L$-Lipschitz map 
$f \colon A \to \bar X$ there is a $\mu L$-Lipschitz map 
$g \colon X \to \bar X$ with $g|_A = f$.
\end{proposition}

This follows from Theorem~1.6 in~\cite{LanS} since $\bar X$ is Lipschitz 
$k$-connected for all $k \ge 0$ and doubling sets have finite Nagata dimension 
(in fact, according to Theorem~1.1 in~\cite{LeDR}, the latter is less 
than or equal to the Assouad dimension). 

\begin{remark} \label{rem:lip-ext}
The assumption in Proposition~\ref{prop:lip-ext} that $A$ be doubling can 
be dropped if, for example, 
$\bar X$ is a homogeneous Hadamard manifold or a Euclidean building; 
the constant $\mu$ then depends (only) on $\bar X$. 
See Theorem~1.2 in~\cite{LanPS}.
It is still unknown whether every Hadamard manifold has this property. 
\end{remark}

By virtue of Lemma~\ref{lem:doubling} and Proposition~\ref{prop:lip-ext},
given a quasi-isometric embedding $f \colon X \to \bar X$
and a quasi-minimizer $S \in \bZ_{n,\loc}(X)$ with controlled density,
one can easily produce a Lipschitz map $g \colon X \to \bar X$ with 
$\sup_{x \in \spt(S)} d(f(x),g(x)) < \infty$ by extending $f|_A$ for a suitable 
separated net $A$ in $\spt(S)$. We now show that then 
$g_\#S \in \bZ_{n,\loc}(\bar X)$ is again a quasi-minimizer with controlled
density.

\begin{proposition}[quasi-isometry invariance] \label{prop:qi-inv}
Let $X$ be a proper metric space with a convex bicombing~$\si$
and with $\asrk(X) = n \ge 2$. Then for all $L,Q \ge 1$, $C > 0$, and 
$a \ge 0$ there exist $\bar Q \ge 1$, $\bar C > 0$, and $\bar a \ge 0$ 
such that the following holds. Suppose that $\bar X$ is another proper 
metric space, $S \in \bZ_{n,\loc}(X)$ is a $(Q,a)$-quasi-minimizer with 
$(C,a)$-controlled density, and $g \colon \spt(S) \to \bar X$ is a map 
satisfying 
\[
L^{-1}d(x,y) - a \le d(g(x),g(y)) \le L\,d(x,y)
\]
for all $x,y \in \spt(S)$. Then $\bar S := g_\#S \in \bZ_{n,\loc}(\bar X)$ 
is a $(\bar Q,\bar a)$-quasi-minimizer with $(\bar C,\bar a)$-controlled 
density, and $d(g(x),\spt(\bar S)) \le \bar a$ for all $x \in \spt(S)$.
\end{proposition}

\begin{proof}
If $x \in \spt(S)$ and $r > a$, then
$g^{-1}(\B{g(x)}{r}) \sub \B{x}{2Lr}$ and thus
\[
\|\bar S\|(\B{g(x)}{r}) 
\le L^n \|S\|(\B{x}{2Lr}) \le L^nC_1r^n
\]   
for $C_1 := (2L)^nC$. Hence, given any $\bar p \in \bar X$ and $r > a/2$ 
such that $\B{\bar p}{r} \cap \spt(\bar S) \ne \es$, it follows that 
$\|\bar S\|(\B{\bar p}{r}) \le \|\bar S\|(\B{g(x)}{2r}) \le L^nC_1(2r)^n$
for some $x \in \spt(S)$. This shows that 
$\bar S$ has $((2L)^nC_1,a/2)$-controlled density.

Next we show that there is a Lipschitz map $\bar g \colon \bar X \to X$
such that $h := \bar g \circ g$ is at finite distance from the
identity on $\spt(S)$.
Let $N \sub \spt(S)$ be a $4La$-separated $4La$-net in $\spt(S)$.
By Lemma~\ref{lem:doubling} (doubling), $N$ is doubling,
and $g|_N \colon N \to g(N)$ is $(4L/3)$-bi-Lipschitz,
so $g(N)$ is doubling as well. The doubling constant depends only on
$n,L,C$. Then, by Proposition~\ref{prop:lip-ext}, 
$(g|_N)^{-1}$ admits an $\bar L$-Lipschitz extension 
$\bar g \colon \bar X \to X$ for some constant $\bar L$ depending on $n,L,C$.
For every $x \in \spt(S)$ there is a $y \in N$ such that
$d(x,y) \le 4La$. Then $h(y) = y$, and 
\[
d(h(x),x) \le d(h(x),h(y)) + d(y,x) \le (L\bar L + 1)\,d(x,y) \le b
\]
for $b := 4(L\bar L + 1)La$.

Let again $x \in \spt(S)$ and $r > a$, and put $B_r := \B{g(x)}{r}$. 
For almost every such $r$, both $\bar S' := \bar S \on B_r$ and 
$S' := S \on g^{-1}(B_r)$ are integral currents, $g_\#S' = \bar S'$, and 
\[
\M(S') \le C_1r^n, \quad \M(\bar S') \le L^nC_1r^n, \quad 
\M(h_\#S') \le (L\bar L)^nC_1r^n.
\]
Let $H \colon [0,1] \times \spt(S) \to X$ denote the homotopy from
$\id_{\spt(S)}$ to $h$ given by $H(t,x) = \si(x,h(x),t)$.
The deformation chain $W := H_\#(\bb{0,1} \times S') \in \bI_{n+1,\cs}(X)$ 
satisfies 
\[
\M(W) \le C_2r^n
\]
for $C_2 := (n+1)(L\bar L)^nb\,C_1$. Furthermore, the support of 
the cylinder $R := H_\#(\bb{0,1} \times \d S') = h_\#S' - S' - \d W$
lies in the closed $b$-neighborhood of $\spt(\d S')$, and 
$\spt(\d S') \sub \spt(S - S')$ is at distance at least $r/L$ from $x$
because $g$ is $L$-Lipschitz. 

Suppose now that $\bar T \in \bI_{n,\cs}(\bar X)$ and $\eps > 0$ are such that 
$\d\bar T = \d\bar S'$ and 
\[
\M(\bar T) \le \eps r^n.
\]
Since $\M(\bar g_\#\bar T) \le \bar L^n \eps r^n$ and 
$\d(\bar g_\#\bar T) = \bar g_\#(\d\bar S') = h_\#(\d S')$,
Theorem~\ref{thm:plateau} (minimizing filling) shows that there is a 
minimizing $T \in \bI_{n,\cs}(X)$ with $\d T = h_\#(\d S')$ and 
\[
\M(T) \le \bar L^n \eps r^n,
\]
and if $\eps$ is sufficiently small, then $\spt(T)$ is within distance
$r/(3L)$, say, from $\spt(h_\#(\d S'))$.
For $r > 3Lb$, it follows that $r/L - b - r/(3L) > r/(3L)$ and thus
$\spt(T - R) \cap \B{x}{r/(3L)} = \es$. Note that $\d(T - R) = \d S'$.
By Lemma~\ref{lem:fill-density} there is a constant $c > 0$ 
such that $\F_{x,r/(3L)}(S) > c$ for $r > 12La$.
Put $Z := h_\#S' - T \in \bZ_{n,\cs}(X)$. It follows from 
Theorem~\ref{thm:subeucl} (sub-Euclidean isoperimetric inequality) that 
there is a constant $\bar a \ge 3Lb \ge 12La$ such that if $r > \bar a$,
then $Z$ possesses a filling $V \in \bI_{n+1,\cs}(X)$ with 
\[
\M(V - W) \le \M(V) + \M(W) < c r^{n+1}.
\]
Since $\d(V - W) = S' - (T - R)$ and $\spt(T - R) \cap \B{x}{r/(3L)} = \es$,
this contradicts the fact that $\F_{x,r/(3L)}(S) > c$. Hence, there is an
$\eps_0 > 0$ such that, for almost all $r > \bar a$, $\M(\bar T) \ge \eps_0r^n$
and thus
\[
\M(\bar S') \le L^nC_1r^n \le Q\,\M(\bar T)
\]
for $Q := L^nC_1/\eps_0$. In the case that $g(x) \in \spt(\bar S)$, 
this shows that $S$ is $(\bar Q,\bar a)$-quasi-minimizing.

If $g(x) \not\in \spt(\bar S)$, the same argument 
for $\bar T := \bar S' = \bar S \on \B{g(x)}{r}$ shows that 
$\|\bar S\|(\B{g(x)}{r}) \ge \eps_0r^n > 0$ for almost all $r > \bar a$.
Thus $d(g(x),\spt(\bar S)) \le \bar a$.
\end{proof}

Our next goal is to prove Theorem~\ref{thm:mapping-as-classes} below.
We need the following auxiliary results.

\begin{lemma}[mapping small fillings] \label{lem:map-small-f}
Let $(X,\si)$ be a proper metric space with a convex bicombing.
Suppose that $n \ge 1$, $Z \in \bZ_{n,\loc}(X)$, $p \in X$,
and $g \colon X \to \bar X$ is an $L$-Lipschitz map into a proper metric
space $\bar X$ such that $d(g(p),g(z)) \ge L^{-1}d(p,z) - a$ for all
$z \in \spt(Z)$, for some constants $L \ge 1$ and $a \ge 0$.
If\/ $\Fi(Z) = 0$, then $\bar Z := g_\#Z \in \bZ_{n,\loc}(\bar X)$
satisfies $\Fi(\bar Z) = 0$.
\end{lemma}

\begin{proof}
Let $\eps > 0$. For every sufficiently large $r > 0$ there exists
$V \in \bI_{n+1,\cs}(X)$ such that $\spt(Z - \d V) \cap \B{p}{r} = \es$ and
$\M(V) < (\eps r)^{n+1}$. By Theorem~\ref{thm:plateau} (minimizing filling)
we can assume that $V$ is minimizing and $d(x,\spt(\d V)) < \eps c r$ for
all $x \in \spt(V)$, where $c > 0$ depends only on $n$.
Assuming that $\eps c < 1/2$, we find an $s > r/2$ such that 
$W := V \on \B{p}{s} \in \bI_{n+1,\cs}(X)$, $\spt(Z - \d W) \cap
\B{p}{r/2} = \es$, and $d(x,\spt(Z)) < \eps c r$ for all 
$x \in \spt(W)$.
Put $\bar W := g_\# W \in \bI_{n+1,\cs}(\bar X)$.
Then, for any $x \in \spt(Z - \d W) \sub \spt(Z) \cup \spt(W)$ and 
$z \in \spt(Z)$ with $d(x,z) < \eps c r$,
\begin{align*} 
d(g(p),g(x)) 
&\ge d(g(p),g(z)) - d(g(x),g(z)) \\
& \ge L^{-1}(d(p,x) - d(x,z)) - a - L\, d(x,z) \\
& > (2L)^{-1}r - (L^{-1} + L)\eps c r - a =: \bar r.
\end{align*}
If $\eps$ is sufficiently small and $r$ is sufficiently large,
so that $r \le 3L\bar r$ say, then
$\M(\bar W) < (\eps L r)^{n+1} \le (3\eps L^2 \bar r)^{n+1}$, 
and the support of $\bar Z - \d \bar W = g_\#(Z - \d W)$ 
is disjoint from $\B{g(p)}{\bar r}$. This gives the result.
\end{proof}

The next lemma states a simple general fact about Lipschitz maps.

\begin{lemma}[combining Lipschitz maps] \label{lem:comb-lip}
Let $X$ be a proper metric space, and let $\bar X$ be a metric space with
a convex bicombing $\bar\si$. Suppose that $A_1,A_2 \sub X$ are two closed
non-empty sets, $L,a \ge 0$ are constants, and $g_1,g_2 \colon X \to \bar X$ 
are $L$-Lipschitz maps such that $d(g_1(x_1),g_2(x_2)) \le L\,d(x_1,x_2) + a$
for all $(x_1,x_2) \in A_1 \times A_2$.
Then there exists a $7L$-Lipschitz map $\hat g \colon X \to \bar X$ such that
$d(\hat g(x),g_i(x)) \le a/2$ for all $x \in A_i$, for $i = 1,2$.
\end{lemma}

\begin{proof}
We assume that $a > 0$; the case $a = 0$ requires only minor modifications
(note that then $g_1 = g_2$ on $A_1 \cap A_2$ by assumption).

For $i = 1,2$, let $u_i \colon X \to \R$ be the $L$-Lipschitz function 
defined by $u_i(x) := L\,d(x,A_i) + a/4$.
Put $w := u_1 + u_2$, $\lam := u_1/w$, and define 
$\hat g \colon X \to \bar X$ by
\[
\hat g(x) := \bar\si(g_1(x),g_2(x),\lam(x)).
\]
Let $x,y \in X$, and put $\bar z := \bar\si(g_1(x),g_2(x),\lam(y))$. Then
\begin{align*}
d(\hat g(x),\hat g(y)) &\le d(\hat g(x),\bar z) + d(\bar z,\hat g(y)), \\
d(\bar z,\hat g(y)) &\le (1-\lam(y))\,d(g_1(x),g_1(y)) 
+ \lam(y)\,d(g_2(x),g_2(y)) \\
&\le L\,d(x,y), 
\end{align*}
and $d(\hat g(x),\bar z) = |\lam(x) - \lam(y)|\,d(g_1(x),g_2(x))$.
Furthermore,
\begin{align*}
|\lam(x) - \lam(y)|
&\le \biggl| \lam(x) - \frac{u_1(y)}{w(x)} \biggr|
+ \biggl| \frac{u_1(y)}{w(x)} - \lam(y) \biggr| \\
&\le \frac{1}{w(x)} |u_1(x) - u_1(y)| + \frac{\lam(y)}{w(x)} |w(y) - w(x)| \\
&\le \frac{3L}{w(x)} \,d(x,y),
\end{align*}
and if $x_1 \in A_1$ and $x_2 \in A_2$ are such 
that $d(x,x_i) = d(x,A_i)$, then
\begin{align*}
d(g_1(x),g_2(x)) &\le d(g_1(x),g_1(x_1)) + d(g_1(x_1),g_2(x_2)) 
+ d(g_2(x_2),g_2(x)) \\
&\le L\,d(x,x_1) + \bigl( L\,d(x_1,x_2) + a \bigr) + L\,d(x_2,x) \\
&\le 2L\,d(x,x_1) + 2L\,d(x,x_2) + a \\
&= 2w(x). 
\end{align*}
It follows that $\hat g$ is $7L$-Lipschitz. 
If $x \in A_1$, then $\lam(x) = a/(4w(x))$, 
thus $d(g_1(x),\hat g(x)) \le \lam(x)\,d(g_1(x),g_2(x)) \le a/2$. 
Similarly, $d(\hat g(x),g_2(x)) \le a/2$ for all $x \in A_2$.
\end{proof}

We now consider again the group~$\cZ X$ of 
$\F$-asymptote classes from Definition~\ref{def:f-asymptotic}. 

\begin{theorem}[mapping asymptote classes] \label{thm:mapping-as-classes}
Let $(X,\si)$ and $(\bar X,\bar\si)$ be two proper metric spaces with
convex bicombings and with $\asrk(X) = \asrk(\bar X) = n \ge 2$, and suppose
that $f \colon X \to \bar X$ is a quasi-isometric embedding.
Then there exists a unique monomorphism
\[
\cZ f \colon \cZ X \to \cZ \bar X
\]
with the property that if $S \in \bZ_{n,\loc}^\infty(X)$ and
$g \colon X \to \bar X$ is a Lipschitz map
such that $\sup_{x \in \spt(S)} d(f(x),g(x)) < \infty$, then
$\cZ f\,[S] = [g_\#S]$. 
If $f$ is a quasi-isometry, then $\cZ f$ is an isomorphism.
\end{theorem}

Note that if $S$ and $g$ are as in the theorem, then 
$g_\#S \in \bZ_{n,\loc}^\infty(\bar X)$ by the argument in 
the first paragraph of the proof of Proposition~\ref{prop:qi-inv},
thus the class $[g_\#S] \in \cZ \bar X$ is defined.
Combining Theorem~\ref{thm:mapping-as-classes}
with Theorem~\ref{thm:f-classes} (Tits boundary), 
we get a monomorphism $f_\T$ that makes the diagram  
\[
\begin{CD}
\bZ_{n-1,\cs}(\bT X) @>{f_\T}>> \bZ_{n-1,\cs}(\bT \bar X)\\
@A{\bT}AA @AA{\bT}A\\
\cZ X @>>{\cZ f}> \cZ \bar X 
\end{CD}
\]
commutative. This yields Theorem~\ref{intro-tits-cycles} in the introduction.

\begin{proof}
Due to Theorem~\ref{thm:constr-minimizers} (constructing minimizers) and 
Proposition~\ref{prop:contr-density} (controlled density), 
every class in $\cZ X$ is represented by a minimizer 
$S \in \bZ_{n,\loc}^\infty(X)$ with controlled density.
It then follows from Lemma~\ref{lem:doubling} (doubling) and 
Proposition~\ref{prop:lip-ext} that 
a Lipschitz map $g \colon X \to \bar X$ with
$\sup_{x \in \spt(S)} d(f(x),g(x)) < \infty$ exists. 
In particular, there is at most one map $\cZ f \colon \cZ X \to \cZ \bar X$ 
with the property stated in the theorem.

Suppose now that $S_1,S_2 \in \bZ_{n,\loc}^\infty(X)$ are arbitrary and 
$g_1,g_2 \colon X \to \bar X$ are Lipschitz maps with 
$\sup_{x \in \spt(S_i)} d(f(x),g_i(x)) < \infty$ for $i = 1,2$.
It follows from Lemma~\ref{lem:comb-lip} that there exists a 
Lipschitz map $\hat g \colon X \to \bar X$ such that
$\sup_{\spt(S_1) \cup \spt(S_2)} d(f(x),\hat g(x)) < \infty$.
Using the $\bar\si$-homotopy from $g_i$ to $\hat g$ one can easily check 
that $g_{i\#}S_i \sim_\F \hat g_\#S_i$.  
In the case that $S_1 \sim_\F S_2$, Lemma~\ref{lem:map-small-f} shows that 
$\hat g_\#S_1 \sim_\F \hat g_\#S_2$, thus $g_{1\#}S_1 \sim_\F g_{2\#}S_2$. 
This yields the existence of a unique map 
$\cZ f \colon \cZ X \to \cZ \bar X$ with the property stated in the 
theorem. Furthermore, since 
\begin{align*}
\cZ f\,[S_1] + \cZ f\,[S_2] &= [\hat g_\#S_1] + [\hat g_\#S_2] = 
[\hat g_\#(S_1+S_2)] = \cZ f\,[S_1+S_2] \\
&= \cZ f\,([S_1] + [S_2]),
\end{align*}
$\cZ f$ is a homomorphism. To show that $\cZ f$ is injective, suppose
that $[S] \ne 0$, where $S$ is a minimizer with controlled density.
Then it follows from Proposition~\ref{prop:qi-inv}
that $g_\#S$ is quasi-minimizing and non-zero for any Lipschitz 
map $g \colon X \to \bar X$ with $\sup_{x \in \spt(S)}d(f(x),g(x)) < \infty$. 
Lemma~\ref{lem:fill-density} (filling density) then shows that 
$\Fi(g_\#S) \ne 0$, thus $\cZ f\,[S] = [g_\#S] \ne 0$.

If $f$ is a quasi-isometry, then there is a quasi-isometric 
embedding $\bar f \colon \bar X \to X$ such that 
$\sup_{\bar x \in \bar X}d((f \circ \bar f)(\bar x),\bar x) < \infty$, 
and it is not difficult to show that 
$\cZ f \circ \cZ \bar f$ is the identity on $\cZ \bar X$.
\end{proof}

\begin{remark} \label{rem:hoelder}
Resuming the discussion of visual metrics, we note that when
$f \colon X \to \bar X$ is an $(L,a_0)$-quasi-isometric embedding,
the monomorphism $\cZ f \colon \cZ X \to \cZ \bar X$ maps each of
the subsets $\cZ_{C,a}X \sub \cZ X$ into $\cZ_{\bar C,\bar a}\bar X$, where
$\bar C,\bar a$ depend on $X,L,a_0,C,a$. Furthermore, there is a constant
$\bar D$, depending in addition on $\bar X$, such that if $S,S' \in \cZ_{C,a}X$
and $Z \in [S-S']$, $\bar Z \in \cZ f[S-S']$ are minimizing, then 
$\spt(\bar Z)$ is at Hausdorff distance at most $\bar D$ from $f(\spt(Z))$. 
As a consequence, for every $p \in X$,
\[
L^{-1}\gp{p}{[S-S']} - a_0 - \bar D
\le \gp{f(p)}{\cZ f[S-S']} \le L\,\gp{p}{[S-S']} + a_0 + \bar D.
\]  
It follows readily that both the restriction of $\cZ f$ to $\cZ_{C,a}X$ and
its inverse are H\"older continuous with exponent $1/L$ for any pair
of visual metrics on $\cZ_{C,a}X$ and $\cZ_{\bar C,\bar a}\bar X$ with
the same parameter $b$.

Higher rank visual metrics will be further discussed elsewhere.
\end{remark} 


\section{Mapping limit sets} \label{sect:lim-sets}

We will now describe the effect of a quasi-isometric embedding 
$f \colon X \to \bar X$, or of the associated monomorphism
$\cZ f \colon \cZ X \to \cZ \bar X$, on the collection of limit sets $\cL X$ 
introduced in Definition~\ref{def:can-lim-sets}. 
We associate to every class $[S] \in \cZ X$ a limit set 
$\Lam[S] \sub \di X$ such that
\[
\Lam[S] = \Lam(S')
\]
for every $S' \in [S]$ that is quasi-minimizing or conical; 
in these cases the invariance of $\Lam(S')$ is granted by 
Theorem~\ref{thm:conical-repr} (conical representative) and 
Proposition~\ref{prop:lim-sets} (equal limit sets).
Thus $\Lam[S] \in \cL X$. For any $S \in \bZ_{n,\loc}^\infty(X)$,
the set $\Lam[S]$ also agrees with $\spt(\bT S)$; however,
Theorem~\ref{thm:lift-cones} (lifting cones) and 
Theorem~\ref{thm:f-classes} (Tits boundary) are not needed for the
proof of Theorem~\ref{thm:map-lim-sets} below.

The following preliminary result relies on 
Theorem~\ref{thm:visibility} (visibility property)
and Theorem~\ref{thm:conicality} (asymptotic conicality).

\begin{proposition}[mapping cones] \label{prop:map-cones}
Let $(X,\si)$ and $(\bar X,\bar\si)$ be two proper metric spaces with
convex bicombings and with $\asrk(X) = \asrk(\bar X) = n \ge 2$,
and let $f \colon X \to \bar X$ be a quasi-isometric embedding.
Suppose that $[S] \in \cZ X$ and $\cZ f\,[S] = [\bar S] \in \cZ \bar X$.
Choose base points $p \in X$ and $\bar p \in \bar X$, and consider the 
geodesic cones $K := \C_p(\Lam[S]) \sub X$ and 
$\bar K := \C_{\bar p}(\Lam[\bar S]) \sub \bar X$.
Then for all $\eps > 0$ there exists an $r > 0$ such that 
\[
d(f(x),\bar K) < \eps\,d(p,x)
\]
for all $x \in K$ with $d(p,x) \ge r$ and 
\[
d(\bar x,f(K)) < \eps\,d(\bar p,\bar x)
\]
for all $\bar x \in \bar K$ with $d(\bar p,\bar x) \ge r$.
\end{proposition}

\begin{proof}
We assume that $f$ is an $(L,a)$-quasi-isometric embedding, $f(p) = \bar p$,
$S$ is a quasi-minimizer with controlled density, and $\bar S = g_\#S$ 
for some Lipschitz map $g \colon X \to \bar X$ 
with $\bar b := \sup_{x \in \spt(S)}d(f(x),g(x)) < \infty$. 

Let $\eps' \in (0,1)$. If $r > 0$ is sufficiently large, then it follows from
Theorem~\ref{thm:conicality} that for 
every $x \in K$ with $d(p,x) \ge r$ there is a $y \in \spt(S)$ such that 
$d(x,y) < \eps' d(p,x)$, thus
\[
d(f(x),f(y)) \le L\eps' d(p,x) + a
\]
and $(1-\eps')\,d(p,x) \le d(p,y) \le (1+\eps')\,d(p,x)$.
By Proposition~\ref{prop:qi-inv} (quasi-isometry invariance), 
$\bar S$ is a quasi-minimizer with controlled density, and 
there is a point $\bar y \in \spt(\bar S)$
such that $d(g(y),\bar y) \le \bar a$ for some constant $\bar a \ge 0$, thus
\[
d(f(y),\bar y) \le \bar a + \bar b =: \bar c
\]
and $d(\bar p,\bar y) \ge d(f(p),f(y)) - \bar c \ge
L^{-1}(1-\eps')r - a - \bar c$.
Hence, if $r$ is sufficiently large, then by the second part of 
Theorem~\ref{thm:visibility},
\[
d(\bar y, \bar K) < \eps' \,d(\bar p,\bar y) \le 2L\eps' d(p,x),
\]
as $d(\bar p,\bar y) \le d(f(p),f(y)) + \bar c 
\le L(1+\eps')\,d(p,x) + a + \bar c \le 2L\,d(p,x)$.
Combining these estimates we get the first assertion, and the second is proved 
similarly.
\end{proof}

We now prove that $f$ induces an injective map 
$\cL f \colon \cL X \to \cL \bar X$. 
If $\cL f(\Lam) = \bar \Lam$, then the cones $\R_+\Lam \sub \CT X$
and $\R_+\bar\Lam \sub \CT \bar X$ are bi-Lipschitz homeomorphic. 

\begin{theorem}[mapping limit sets] \label{thm:map-lim-sets}
Let $(X,\si)$ and $(\bar X,\bar\si)$ be two proper metric spaces with
convex bicombings and with $\asrk(X) = \asrk(\bar X) = n \ge 2$, and suppose
that $f \colon X \to \bar X$ is an $(L,a)$-quasi-isometric embedding. 
Then there exists an injective map
\[
\cL f \colon \cL X \to \cL \bar X
\]
such that $\cL f(\Lam[S]) = \Lam[\bar S]$ whenever $\cZ f\,[S] = [\bar S]$.
For every finite union $M := \bigcup_{i=1}^k \Lam_i$ of sets $\Lam_i \in \cL X$
and the corresponding union $\bar M := \bigcup_{i=1}^k \bar \Lam_i$ of the sets
$\bar\Lam_i := \cL f(\Lam_i)$, there 
is a pointed $L$-bi-Lipschitz homeomorphism 
$\Phi \colon \R_+M \to \R_+\bar M$ such that 
$\Phi(\R_+\Lam_i) = \R_+\bar\Lam_i$ for $i = 1,\dots,k$.
If $f$ is a quasi-isometry, then $\cL f$ is a bijection.
\end{theorem}
 
Here $\Phi$ is said to be {\em pointed\/} if $\Phi(o) = \bar o$, where 
$o$ and $\bar o$ are the cone vertices of $\CT X$ and $\CT \bar X$, 
respectively.

\begin{proof}
Choose base points $p \in X$ and $\bar p := f(p) \in \bar X$.
Suppose that $\cZ f\,[S] = [\bar S]$ and $\cZ f\,[T] = [\bar T]$.
We use Proposition~\ref{prop:map-cones}. 
If $\Lam[S] = \Lam[T]$, then for every $\eps \in (0,1)$ and every 
$\bar x \in \C_{\bar p}(\Lam[\bar S])$ with sufficiently large distance 
to $\bar p$ there exists an $x \in C_p(\Lam[S]) = \C_p(\Lam[T])$ such that
\[
d(\bar x,f(x)) < \eps\,d(\bar p,\bar x)
\]
and $(2L)^{-1}d(\bar p,\bar x) \le d(p,x) \le 2L\,d(\bar p,\bar x)$; 
then there is also a point $\bar y \in C_{\bar p}(\Lam[\bar T])$ such that 
\[
d(f(x),\bar y) < \eps\,d(p,x) \le 2L\eps\,d(\bar p,\bar x).
\]
It follows that $\Lam[\bar S] \sub \Lam[\bar T]$, and the reverse 
inclusion holds by symmetry. Conversely, if 
$\Lam[\bar S] = \Lam[\bar T]$, then a similar argument shows that 
$\Lam[S] = \Lam[T]$. This yields the existence of an injective map
$\cL f \colon \cL X \to \cL \bar X$
such that $\cL f(\Lam[S]) = \Lam[\bar S]$ whenever 
$\cZ f\,[S] = [\bar S]$.

Let now $M$ and $\bar M$ be given as in the theorem. 
By Proposition~\ref{prop:cpt-lim-sets} (compact limit sets), 
the cones $\R_+M$ and $\R_+\bar M$ are proper and thus separable.
For $r > 0$, let $\pi_r \colon \CT X \to X$ and 
$\bar\pi_r \colon \CT \bar X \to \bar X$ denote the $r$-Lipschitz maps
defined by
\[
\pi_r(u) := \can_p(ru), \quad \bar \pi_r(\bar u) := \can_{\bar p}(r\bar u).
\]
Let first $N \sub \R_+M$ be a finite set containing $o$, and let $\eps > 0$.
It follows from Proposition~\ref{prop:map-cones} that if we pick $r > 0$
sufficiently large, then for every $u \in N$ and 
$i \in I(u) := \{i: u \in \R_+\Lam_i\}$ there is 
a point $\bar u_{r,i} \in \R_+\bar \Lam_i$ such that 
\[
d(f(\pi_r(u)),\bar\pi_r(\bar u_{r,i})) \le \eps r,
\]
where $\bar o_{r,i} := \bar o$ for $i = 1,\dots,k$. Then, for all 
$u,v \in N$ and $i \in I(u)$, $j \in I(v)$,
\begin{align*}
L^{-1} d(\pi_r(u),\pi_r(v)) - a - 2\eps r 
&\le d(\bar\pi_r(\bar u_{r,i}),\bar\pi_r(\bar v_{r,j})) \\
&\le L\,d(\pi_r(u),\pi_r(v)) + a + 2\eps r
\end{align*}
and $d(\bar p,\bar\pi_r(\bar u_{r,i})) \le L\,d(p,\pi_r(u)) + a + \eps r$,
thus $\dT(\bar o,\bar u_{r,i}) \le L\,\dT(o,u) + r^{-1}a + \eps$.
We infer from Lemma~\ref{lem:unif-conv} (uniform convergence) 
that if $r > a/\eps$ is sufficiently large, then 
\[
L^{-1} \dT(u,v) - 4\eps \le
\dT(\bar u_{r,i},\bar v_{r,j}) \le L\,\dT(u,v) + 4\eps.
\]
For $u = v$, this also shows that the set $\{\bar u_{r,i}: i \in I(u)\}$ 
associated to $u$ has diameter at most $4\eps$. Let $s \ge 2\eps$.
It follows again from Proposition~\ref{prop:map-cones} that 
if $r > a/\eps$ is sufficiently large,
then for every $\bar w \in [\eps,s]\bar\Lam_i$ there is a
$w \in \R_+\Lam_i$ such that
\[
d(f(\pi_r(w)),\bar\pi_r(\bar w)) \le \eps r
\] 
and $\dT(o,w) \le L(\dT(\bar o,\bar w) + 2\eps) \le 2Ls$.
Then, for $u \in N \cap \R_+\Lam_i$, we can conclude as above 
that $\dT(\bar u_{r,i},\bar w) \le L\,\dT(u,w) + 4\eps$, provided $r$ is 
large enough. Hence, if we assume that $N \cap [0,2Ls]\Lam_i$ is an 
$\eps$-net in $[0,2Ls]\Lam_i$, then $\{\bar u_{r,i}: u \in N \cap \R_+\Lam_i\}$
forms an $(L + 4)\eps$-net in $[0,s]\bar\Lam_i$. Repeating this construction 
for some sequences $\eps_l \to 0$ and $s_l \to \infty$ 
and a suitable sequence $N_1 \sub N_2 \sub \ldots$ of subsets 
of $\R_+M$, we get the desired map 
$\Phi \colon \R_+M \to \R_+\bar M$ via a diagonal sequence argument.

Finally, if $f$ is a quasi-isometry, then $\cZ f \colon \cZ X \to \cZ \bar X$
is an isomorphism by Theorem~\ref{thm:mapping-as-classes}. Hence, for every 
$\bar\Lam = \Lam[\bar S] \in \cL \bar X$ there exists a 
$\Lam = \Lam[S] \in \cL X$ such that $\cZ f\,[S] = [\bar S]$ and
thus $\cL f(\Lam) = \bar\Lam$. 
\end{proof}

This result readily implies Theorem~\ref{intro-lim-sets}
in the introduction. Note that if $P = \bigcap_{i=1}^j\Lam_i$, 
$Q = \bigcap_{i=j+1}^k\Lam_i$, and 
$\bar P,\bar Q$ are the corresponding intersections of the sets 
$\bar\Lam_i := \cL f(\Lam_i)$, then the existence of a map 
$\Phi$ as in Theorem~\ref{thm:map-lim-sets} guarantees that $P \sub Q$ 
if and only if $\bar P \sub \bar Q$.

If $X$ and $\bar X$ are symmetric spaces of non-compact type and 
of rank $n \ge 2$, then their Tits boundaries have the structure 
of thick $(n-1)$-dimensional spherical buildings, and every Weyl chamber 
is the intersection of the limit sets of two $n$-flats.
It then follows from Theorem~\ref{intro-lim-sets} that every quasi-isometry 
$f \colon X \to \bar X$ induces an isomorphism (order preserving bijection) 
between the two buildings, which must carry apartments to apartments. 
This shows that the map $\cL f \colon \cL X \to \cL \bar X$ 
takes limit sets of $n$-flats to limit sets of $n$-flats, and it follows from
the case $k = 1$ of Theorem~\ref{thm:struct-qflats} below 
or Theorem~\ref{intro-qflats} that for 
every $n$-flat $F \sub X$ there is an $n$-flat $\bar F \sub \bar X$ at 
uniformly bounded Hausdorff distance from $f(F)$. This constitutes a major 
step in the proof of the quasi-isometric rigidity theorem for symmetric 
spaces of non-compact type without rank one de Rham factors; compare 
Corollary~7.1.5 in~\cite{KleL} and Lemma~8.6 in~\cite{EskF}. 
The proof may then be completed along the lines in these papers, using
Tits' work~\cite{Tit}.

\begin{theorem}[structure of quasiflats] \label{thm:struct-qflats} 
Let $X$ be a proper metric space with a convex bicombing~$\si$ and 
with $\asrk(X) = n \ge 2$. 
Let $f \colon \R^n \to X$ be an $(L,a_0)$-quasi-isometric embedding
with limit set $\Lam := \di(f(\R^n))$. 
Then the cone $K := \R_+\Lam \sub \CT X$ is $L$-bi-Lipschitz equivalent 
to $\R^n$. Suppose that $K$ is the union of closed sets 
$K_1,\dots,K_k$ such that, for some point $p \in X$, $\can_p|_{K_i}$ is a
($1$-Lipschitz) $(L,a_0)$-quasi-isometric embedding for $i = 1,\dots,k$. 
Then $f(\R^n)$ is within distance at most $b$ from $\C_p(\Lam) = \can_p(K)$
for some constant $b$ depending only on $X,L,a_0$ and $k$.
In the case $k = 1$, $f(\R^n)$ is at Hausdorff distance at 
most $b$ from $\C_p(\Lam)$.
\end{theorem}

\begin{proof}
Let $E := \bb{\R^n} \in \bZ_{n,\loc}(\R^n)$. 
By Proposition~\ref{prop:lip-qflats} (Lipschitz quasiflats) there are 
constants $Q,C,a$, depending only on $n,L,a_0$, such that $\cZ f\,[E] = [S]$ 
for some $(Q,a)$-quasi-minimizer $S \in \bZ_{n,\loc}(X)$ with 
$(C,a)$-controlled density and $\Hd(\spt(S),f(\R^n)) \le a$. 
Then $\Lam = \Lam(S) = \Lam[S]$, and Theorem~\ref{thm:map-lim-sets} shows 
that there exists an $L$-bi-Lipschitz homeomorphism
$\phi \colon \R^n \to K = \R_+\Lam$. 

Suppose now that the additional assumption in the theorem holds
for some $p \in X$. By Theorem~\ref{thm:conical-repr} (conical representative),
$\spt(S_{p,0}) \sub \C_p(\Lam)$ and $\Gi(S_{p,0}) \le \Gi(S) \le C$.
Our aim is to show that $S_{p,0}$ has controlled density.
By Theorem~\ref{thm:lift-cones} (lifting cones) there exists 
a local cycle $\Sig \in \bZ_{n,\loc}(K)$ in $K$ such that 
$\can_{p\#}\Sig = S_{p,0}$ and $\Gi(\Sig) \le C$. 
Note, however, that if $k = 1$ and $\can_p|_K$ is bi-Lipschitz or even 
isometric (for example, if $X$ is $\CAT(0)$ and $\C_p(\Lam)$ is a flat),  
then one can simply put $\Sig := (\can_p|_K^{\,-1})_\#S_{p,0}$ and the 
theorem is not needed.
Now $(\phi^{-1})_\#\Sig$ is an element of $\bZ_{n,\loc}(\R^n)$ and hence of 
the form $m E$ for some constant integer multiplicity $m$. 
Since $\phi^{-1}$ is $L$-bi-Lipschitz, it follows that $|m|$ is bounded in 
terms of $C$ and $L$ (there is no need to show that in fact $|m| = 1$). 
For $i = 1,\dots,k$, let $\psi_i$ denote the restriction of 
$\can_p \circ \phi$ to $\phi^{-1}(K_i)$.
Note that $\psi_i$ is $L$-Lipschitz and $(L^2,a_0)$-quasi-isometric
by the assumption on $\can_p|_{K_i}$.
Choose Borel sets $B_i \sub K_i$ such that the union 
$\bigcup_{i=1}^k B_i = K$ is disjoint. Since $\phi_\#(mE) = \Sig$, 
\[
\psi_{i\#}(mE \on \phi^{-1}(B_i)) 
= \can_{p\#}\bigl( \phi_\#(mE \on \phi^{-1}(B_i)) \bigr)
= \can_{p\#}(\Sig \on B_i).
\]
If $q \in X$ and $r > a_0$, then $\psi_i^{\,-1}(\B{q}{r})$
has diameter at most $L^2(2r + a_0) \le 3L^2r$, and it follows that
\[
\|\can_{p\#}(\Sig \on B_i)\|(\B{q}{r}) \le C_0r^n
\] 
for some constant $C_0$ depending only on $m,n,L$. 
Since $\sum_{i=1}^k\can_{p\#}(\Sig \on B_i) = \can_{p\#}\Sig = S_{p,0}$, 
we conclude that $S_{p,0}$ has $(kC_0,a_0)$-controlled density. Now  
Theorem~\ref{thm:morse-1} (Morse Lemma I) yields the first 
conclusion of the theorem.

If $k = 1$, then $\psi_1 = \can_p \circ \phi$ is a Lipschitz quasiflat,
hence $S_{p,0} = \psi_{1\#}(mE)$ is quasi-minimizing and $\spt(S_{p,0})$
is at finite Hausdorff distance from $\psi_1(\R^n) = \C_p(\Lam)$
by Proposition~\ref{prop:lip-qflats} (which extends to higher multiples of 
$E = \bb{\R^n}$). The desired estimate follows again from 
Theorem~\ref{thm:morse-1}.
\end{proof}

Theorem~\ref{intro-qflats} stated in the introduction follows as a 
special case. In the case $k = 1$, this applies in particular 
to $\CAT(0)$ spaces with isolated flats; compare Lemma~3.1 
in~\cite{Sch} (the case $\mathbf{F} = \R$) and Theorem~4.1.1 
in~\cite{HruK}. Furthermore, it follows easily that every $n$-dimensional 
quasiflat in a nonpositively curved symmetric space of rank $n \ge 2$ lies 
within uniformly bounded distance from the union of a finite, uniformly 
bounded number of $n$-flats; compare Theorem~1.2.5 in~\cite{KleL} and 
Theorem~1.1 in~\cite{EskF}. 
We also refer to~\cite{BehHS2, BesKS, Hua, KapL1, LanSc} for various similar 
statements.



\end{document}